\documentclass[12pt,final]{article}

\usepackage[paper=letterpaper,margin=1in,twoside=false,includehead,headheight=18pt]{geometry}

\usepackage[utf8]{inputenc}
\usepackage[draft]{fixme}

\usepackage{amsmath,
	amssymb,
	amsthm,
	bm,			
	enumerate,
	hyperref,
	mathtools,
	nicefrac,
	tabularx,
	thm-restate,
	todonotes
}

\newtheorem{theorem}{Theorem}[section]
\newtheorem{corollary}[theorem]{Corollary}

\newtheorem{lemma}[theorem]{Lemma} 
\newtheorem{proposition}[theorem]{Proposition} 

{\theoremstyle{definition} 
\newtheorem*{remark}{Remark} 
}

\DeclareMathVersion{lifts}
\SetSymbolFont{letters}{lifts}{U}{eur}{m}{n}
\DeclareMathAlphabet{\eu}{U}{eur}{m}{n}

\newcommand{\ent}{{\mathbb Z}}
\newcommand{\real}{{\mathbb R}}

\newcommand{\pr}{\mathbf{Pr}}
\newcommand{\ex}{\mathbf{E}}

\newcommand{\mbf}[1] {\text{\boldmath$#1$}}
\newcommand{\inp}[1]{\left \langle #1 \right \rangle}

\newcommand{\suma}[1]{\sum_{\substack{#1}}}
\newcommand{\proda}[1]{\prod_{\substack{#1}}}

\DeclareMathOperator{\spec}{spec}
\DeclareMathOperator{\Ker}{Ker}
\DeclareMathOperator{\rank}{rank}
\DeclareMathOperator{\tr}{tr}

\begin{document}

\title{The chromatic number of random lifts of complete graphs}
\date{\today}
\author{JD Nir\\
\small{University of Manitoba}\\
\small{\texttt{jd.nir@umanitoba.ca}}
 \and
Xavier P\'erez-Gim\'enez\thanks{Research supported in part by Simons Foundation Grant \#587019.}\\
\small{University of Nebraska-Lincoln}\\
\small{\texttt{xperez@unl.edu}}}
\maketitle
\begin{abstract}
An $n$-lift of a graph $G$ is a graph from which there is an $n$-to-$1$ covering map onto $G$. Amit, Linial, and Matou\v sek (2002) raised the question of whether the chromatic number of a random $n$-lift of $K_5$ is concentrated on a single value.
We consider this problem for $G=K_{d+1}$, and show that for fixed $d\ge 3$ the chromatic number of a random lift of $K_d$ is (asymptotically almost surely) either $k$ or $k+1$, where $k$ is the smallest integer satisfying $d < 2k \log k$. Moreover, we show that, for roughly half of the values of $d$, the chromatic number is concentrated on $k$. The argument for the upper-bound on the chromatic number uses the small subgraph conditioning method, and it can be extended to random $n$-lifts of $G$, for any fixed $d$-regular graph $G$.
\end{abstract}

\newcommand{\and}{\qquad\text{and}\qquad} 

\section{Introduction}
\label{sec:introduction}

Given two loopless multigraphs (graphs, for short) $G$ and $L$, a {\em covering map} is a surjective graph homomorphism $f: L\to G$ which is also a local isomorphism: that is, for each vertex $v$ of $L$, the set of edges incident with $v$ is mapped bijectively to the set of edges incident with $f(v)$. This corresponds to the general topological notion of covering map restricted to the case of graphs, and can be described in a purely combinatorial way.
If such a map exists, we say that $L$ is a {\em covering graph} or a {\em lift} of $G$, and call $G$ the {\em base graph}.
For each vertex $v\in V(G)$, the set $f^{-1}(v) \subseteq V(L)$ is called the {\em fiber} of $v$.
It is well-known and easy to show that if $G$ is connected then all fibers must have the same cardinality. 
We call a lift of a (not necessarily connected) graph $G$ an $n$-lift if all the fibers have size $n$.
In view of all the above, given a fixed base graph $G$, we define a random lift $L_n(G)$ to be an $n$-lift of $G$ chosen uniformly at random.
This model of random lifts was introduced in a series of papers by Amit, Linial, Matou\v sek and Rozenman~\cite{AL02,AL06,ALM02,LR05}.
Roughly speaking, one can generate a random instance of $L_n(G)$ by first replacing each vertex in $G$ by a fiber containing $n$ vertices and then, for each edge of $G$, adding a random perfect matching between the two fibers corresponding to its endpoints, with all the perfect matchings being chosen independently.
Note that, in the case of multiple edges, two edges of $G$ with the same endpoints contribute to two perfect matchings between their fibers.
This model can also be extended to the case in which the base graph $G$ is allowed to have loops.
Here we distinguish two types of loops: {\em half-loops} and {\em whole-loops}. Both half- and whole-loops denote edges that are only incident with one single vertex, but each whole-loop contributes twice to the degree of its endpoint whereas each half-loop contributes only once.
To form a random lift $L_n(G)$ of a base graph $G$ with loops, for each half-loop with endpoint $v\in V(G)$, we add a random perfect matching to the fiber of $v$ (which requires $n$ to be even). Likewise, for each whole-loop with endpoint $v\in V(G)$, we add a random $2$-factor (i.e.~a spanning 2-regular graph, possibly with whole-loops and double edges) to its fiber.
For a more detailed description of this and other related models of random lifts of a graph with loops and multiple edges we refer the reader to~\cite{FK19}.

If $G$ is a $d$-regular graph, so is any lift of $G$. Therefore, $L_n(G)$ provides a model of random $d$-regular graphs on $n |V(G)|$ vertices.
A simple and interesting case is when $G$ is a bouquet $B_d$ (i.e.~one single vertex with $d$ half-loops). Then the corresponding random lift is the union of $d$ independent random perfect matchings on a common set of $n$ vertices, which (conditional upon obtaining a simple graph) is known to be contiguous to the uniform model of random $d$-regular graphs on $n$ vertices~\cite{GJKW02}, which we denote by $G_{n,d}$.
On the other hand, for any $d$-regular base graph $G$ with more than one vertex, $L_n(G)$ is not contiguous to the uniform model $G_{n|V(G)|,d}$. (This can be easily seen by bounding the expected number of covering maps from a uniform random $d$-regular graph to $G$.) The case when $G$ is the complete $d$-regular graph $K_{d+1}$ has been widely studied as well.

The chromatic number $\chi(G)$ of a graph $G$ is the smallest number of colours needed to properly colour the vertices of the graph: that is, so that no pair of adjacent vertices receive the same colour.
Clearly, if $L$ is a lift of $G$, then $\chi(L)\le \chi(G)$, since a proper $k$-colouring of $G$ naturally induces a proper $k$-colouring of $L$ by assigning to all vertices of the fiber $f^{-1}(v)$ the same colour as $v$. In general, this bound can be far from being tight since, for instance, one can easily build a bipartite lift of $K_k$ for any given $k$.
One natural question is to determine the typical asymptotic behaviour of $\chi(L_n(G))$ as $n\to\infty$.
In this direction, Amit, Linial and Matou\v sek~\cite{ALM02} proved that, for any fixed simple graph $G$, $\chi(L_n(G)) \le (1+\epsilon_\Delta)\Delta/\log\Delta$ a.a.s.\footnote{We say that a sequence of events $E_n$ holds {\em asymptotically almost surely} (a.a.s.) if $\lim_{n\to\infty} \pr(E_n)=1$.}, where $\Delta=\Delta(G)$ is the maximum degree of $G$ and $\lim_{\Delta\to\infty}\epsilon_\Delta=0$.
Moreover, they showed that, for any nonempty simple $G$, a.a.s.\ $\chi(L_n(G)) \ge \sqrt{\chi(G)/3\log\chi(G)}$, and they asked whether or not this lower bound can be replaced by $\chi(L_n(G)) \ge C ( \chi(G)/\log\chi(G) )$, for some fixed constant $C>0$. They proved that this is indeed true when $G=K_k$, so combining the two bounds yields $C_1 k/\log k \le \chi(L_n( K_k )) \le C_2 k/\log k$ a.a.s.\ in that case, for some fixed constants $C_1,C_2<0$.
They also conjectured that, for every simple graph $G$, there is some value $k_G$ for which $\chi(L_n(G)) = k_G$ a.a.s. This is known to be true when $G$ is bipartite or when $\Delta(G)\le 3$, but for instance it is not known for $G=K_k$ with $k\ge5$. The case $G=K_5$ has received some attention and, for instance, Farzad and Theis~\cite{FT12} proved the conjecture for $G=K_5-e$.

On the other hand, a lot of work has been done to determine the chromatic number of a random $d$-regular graph from the uniform model $G_{n,d}$ (see~\cite{FL92,AM04,SW07,SW07a,KPW10,CEH13}).
(Recall that, this model is contiguous to $L_n(B_d)$, conditional upon the lift being a simple graph.)
For instance, Kemkes et al.~\cite{KPW10} determined the chromatic number of $G_{n,d}$ to be a.a.s.\ concentrated in a 2-value window $\{k_d,k_d+1\}$, where $k_d$ is the smallest integer $k$ satisfying $d < 2k \log k$. They also showed that for roughly half of the values of $d$, the largest value can be dropped and thus the chromatic number is a.a.s.\ $k_d$. Coja-Oghlan et al.~\cite{CEH13} extended this result to all sufficiently large values of $d$ using ideas arising from statistical physics.

\paragraph{Our contribution.}
We analyze the chromatic number of a random $n$-lift of $G=K_{d+1}$, and show that it is a.a.s.\ in $\{k_d,k_d+1\}$, where $k_d$ is defined as above. Moreover, for roughly half of the values of $d$, the chromatic number of $L_n(K_{d+1})$ is a.a.s.~$k_d$.
In other words, we show that the results on the uniform model of $d$-regular graphs in~\cite{KPW10} (that is, $G=B_d$) also hold for a random lift of $G=K_{d+1}$. Furthermore, our upper bounds on the chromatic number are also valid for a random lift of any fixed $d$-regular loopless graph $G$ (possibly with multiple edges). Specifically, we prove the following two theorems.
\begin{theorem} \label{thm:lower_bound}
Let $k\ge 2$, and define
\[
u_k = \frac{2\log k}{\log k - \log(k-1)} < (2k-1)\log k.
\]
Suppose $d \ge u_k$. Then a.a.s.\ a random $n$-lift of $K_{d+1}$ is not $k$-colourable.
\end{theorem}

\begin{theorem} \label{thm:upper_bound}
Let $k\ge 3$, and define
\[
\ell_k = \frac{2(k-1)^3}{k(k-2)}\log(k-1) > 2(k-1)\log(k-1).
\]
Suppose $d < \ell_k$, and let $G$ be any fixed $d$-regular loopless multigraph. Then a.a.s.\ a random $n$-lift of $G$ is $k$-colourable.
\end{theorem}
We remark that the base graph $G$ is allowed to have multiple edges in Theorem~\ref{thm:upper_bound}. We exclude half- and whole-loops to make some of the technical calculations in the paper simpler and since we are mainly interested in the case $G=K_{d+1}$. It is conceivable that the result also holds for graphs with loops (of either type). In fact, such an extension would imply the main result in~\cite{KPW10} for the uniform model $G_{n,d}$, by taking $G=B_d$.

Combining Theorems~\ref{thm:lower_bound} and~\ref{thm:upper_bound} for $G=K_{d+1}$ immediately yields the following result.
\begin{corollary}\label{cor:main}
For each $k \ge 3$:
\begin{enumerate}[(i)]
\item $u_{k-1} < \ell_k < u_k$,
\item if $d \in [u_{k-1}, \ell_k)$ then a.a.s. $\chi(L_n(K_{d+1})) = k$, and
\item if $d \in [\ell_k, u_k)$ then a.a.s. $\chi(L_n(K_{d+1})) \in \{k,
k+1\}$.
\end{enumerate}
\end{corollary}
In particular, noting that $u_k < (2k-1)\log k$ and $\ell_k > 2(k-1)\log(k-1)$ we can rewrite the above corollary in a slightly weaker but simpler form.
\begin{corollary}\label{cor:main2}
For each $d \ge 3$, let $k_d = \min\{k : d < 2k\log k\}$.
\begin{enumerate}[(i)]
\item Then a.a.s. $\chi(L_n(K_{d+1})) \in \{k_d, k_d+1\}$.
\item If moreover $d>(2k_d-1)\log k_d$, then a.a.s. $\chi(L_n(K_{d+1})) = k_d+1$.
\end{enumerate}
\end{corollary}
As observed in similar discussions in~\cite{AN05} and~\cite{KPW10}, `roughly half' of the integer values $d\ge3$ satisfy the extra condition in part~(ii) of the corollary.
It is conceivable that, by using similar methods as in~\cite{CEH13}, one could extend the one-point concentration conclusion of part~(ii) to hold for every sufficiently large $d$.

\paragraph{Outline of the argument.}
The lower bound on the chromatic number follows from bounding the expected number of $k$-colourings $X$ of a random lift $L_n(G)$ for the case $G=K_{d+1}$. For the upper bound, we use the small subgraph conditioning method introduced by Robinson and Wormald (see~\cite{JLR00}, Chapter~9, and~\cite{Wor99} for a full exposition of the method) applied to the number $Y$ of strongly equitable $k$-colourings of $L_n(G)$ (that is, $k$-colourings in which all colours appear the same number of times in any fiber).
Our argument requires an accurate estimate of $\ex X$, $\ex Y$, $\ex Y^2$ and some additional joint factorial moments of the short-cycle counts of $L_n(G)$.
In order to determine the exponential behaviour of each of these moments, we must solve a non-trivial optimization problem over some bounded polytope.

Somewhat surprisingly, the optimization problem that arises in the calculation of the first moment $\ex X$ has a ``second moment flavour'', and is more involved than the corresponding one for $\ex Y^2$. Here is an intuitive explanation for this fact: in order to compute $\ex X$ we need to count triples $(f,g,L)$ where $L$ is a $d$-regular graph of order $(d+1)n$, $f:L\to K_{d+1}$ is a covering map (which can also be seen as a proper $(d+1)$-colouring of $L$ with the additional property that each vertex sees all the $d+1$ colours in its closed neighbourhood) and $g$ is a proper $k$-colouring of $L$.
So this can be seen as a lopsided version of a similar counting problem that arises when computing the second moment of the number of $k$-colourings in the uniform model $G_{n,d}$ (see~e.g.~\cite{AM04,KPW10}).
To compute $\ex X$, we solve an optimization problem over the set of $(d+1)\times k$ stochastic matrices that generalizes a result of Achlioptas and Naor~\cite{AN05} for square stochastic matrices.
We believe that our generalization is of independent interest for future applications.
On the other hand, the corresponding optimization problem in the calculation of $\ex Y^2$ can be reduced (after some work) to the original optimization over square stochastic matrices in~\cite{AN05}. (This relies on the fact that $Y$ concerns only strongly balanced colourings.)
Finally, note that our analysis of $\ex X$ makes use of the fact that $G=K_{d+1}$ (since our claims are false for general base graphs), whereas our calculations for $\ex Y$ and $\ex Y^2$ generalize to any loopless $d$-regular graph $G$.

Next, we compute the polynomial factors in our estimates of $\ex Y$ and $\ex Y^2$. This can be done using (at least) two different methods, which we apply several times in the argument. One is the saddle-point method, which uses complex integration over an appropriate contour to obtain an asymptotic estimate of the coefficients of a generating function. The other one is a version of the Laplace summation method over lattices given in~\cite{GJR10}. We reformulate the result in~\cite{GJR10} in terms of counting maximal forests in an auxiliary graph that encodes the summation constraints. This reformulation unveils the role of the eigenvalues of this auxiliary graph, in view of the well-known matrix-tree theorem, and simplifies the calculations. We believe that our version of this tool can be useful in future applications.

Finally, the use of the small subgraph conditioning method requires us to investigate how the presence of short cycles in the random lift $L_n(G)$ affects the number of colourings. (Here, $G$ is any fixed $d$-regular graph.)
This requires the enumeration of non-backtracking closed walks in $G$, for which we use algebraic tools developed by Friedman~\cite{Fri08} (see also~\cite{GJR10}).

\paragraph{Notation:}
We adopt the following notations and conventions throughout this article. For any vector (or matrix) $y=(y_i)_{i\in I}$ with nonnegative integer entries and $x=\sum_{i\in I}y_i$, we define
\begin{equation}
y! = \prod_{i\in I} y_i!
\and
\binom{x}{y} = \frac{x!}{y!} = \frac{x!}{\prod_{i\in I} y_i!}
\label{eq:factorial}
\end{equation}
We use the convention $0^0=1$ and $0\log0=0$. For any two vectors $x=(x_i)_{i=1}^m, y=(y_i)_{i=1}^m \in \real^m$, $\langle x,y\rangle$ denotes the usual inner-product $\sum_{i=1}^m x_iy_i$.

\paragraph{Structure:}
In Section~\ref{sec:main} we prove Theorems~\ref{thm:lower_bound} and \ref{thm:upper_bound} in the case that $k$ divides $n$, assuming all the relevant moment estimates. In Section~\ref{sec:tools} we state and prove two useful propositions which may be of independent interest: first, an extension of an optimization result of Achlioptas and Naor~\cite{AN05} from square to rectangular stochastic matrices, and second, a reformulation of a result of Greenhill, Janson and Ruci\'nski~\cite{GJR10} to compute Laplace summations over lattices in terms of counting maximal forests.
In Section~\ref{sec:first} we provide first-moment calculations for both $X$ and $Y$ including the ``second-moment flavour'' optimization involved in the bounding of $\ex X$.
In Section~\ref{sec:second} we obtain a precise estimate for $\ex Y^2$. Then in Section~\ref{sec:joint} we compute the joint factorial moments required to use the small subgraph conditioning method. Finally, we conclude with Section~\ref{sec:extension} in which we extend our results to the case where $k$ does not divide $n$.


\section{Proofs of Theorems~\ref{thm:lower_bound} and \ref{thm:upper_bound} (for $n$ divisible by $k$)} \label{sec:main}

Let $X$ be the number of $k$-colourings of a random $n$-lift of $K_{d+1}$. We prove Theorem~\ref{thm:lower_bound} as a consequence of Proposition~\ref{prop:EX}.

\begin{proposition}\label{prop:EX}
Let $d$ be an integer and suppose $k$ is an integer satisfying $3 \le k \le d+1$ and $\frac{d^2-1}{d \log d} < 2 (k-1)$. Then there is a constant $M$ such that
\[
\ex X = O\left(n^M\right) \left(\frac{(k-1)^d}{k^{d-2}}\right)^{(d+1)n/2}.
\]
\end{proposition}

\begin{proof}[Proof of Theorem~\ref{thm:lower_bound}]
It suffices to prove the statement for $d = \lceil u_k \rceil$. First note that if $k=2$ then $d = 2$. Using standard arguments one can show that a.a.s.\ the random lift of $K_3$ contains some odd cycles, and thus the statement follows.

Now fix $k$ such that $3 \le k \le d+1$. In that case, $u_k\notin\ent$ so $d>u_k$. We claim that
\begin{enumerate}
\item[(i)]
$- \log\log\left(\frac{1}{1-1/k}\right) > \log k - \frac{5}{8k} > 0$ and
\item[(ii)]
$\log k < (k-1) \left( \log\log k^2 - \frac{5}{8k}\right)$.
\end{enumerate}
In view of these facts,
\begin{align*}
\frac{d^2-1}{d \log d} < \frac{d}{\log d} < \frac{u_k+1}{\log u_k}  &< \frac{2k\log k}{\log\log k^2 - \log\log\left(\frac{1}{1-1/k}\right)}\\
	&< \frac{2k\log k}{\log\log k^2 + \log k - \frac{5}{8k}}\\
	&< \frac{2k\log k}{\log k + (\log k)/(k-1)} = 2(k-1).
\end{align*}
Then we can apply Proposition~\ref{prop:EX} and conclude that $\ex X = O(n^M) \left(\frac{(k-1)^d}{k^{d-2}}\right)^{(d+1)n/2}$. \linebreak Moreover, since $d>u_k$, $\frac{(k-1)^d}{k^{d-2}} <1$ and thus $\ex X=o(1)$.
\end{proof}

Let $L$ be an $n$-lift of a fixed graph $G=G(V,E)$ with covering map $\Pi$. We call a proper $k$-colouring of $L$ {\em strongly equitable} if for every $v\in V$, the set $\Pi^{-1}(v)$ is equitably coloured: that is, each colour is assigned to exactly $n/k$ vertices in $\Pi^{-1}(v)$. This requires $n$ to be divisible by $k$, which we always assume when discussing strongly equitable colourings.

Let $Y$ be the number of strongly equitable $k$-colourings of a random lift of $G$, where $G$ is any fixed $d$-regular graph (not necessarily $G=K_{d+1}$). Let $A$ be the adjacency matrix of $G$, and let $\alpha_1, \ldots, \alpha_{|V|}$ be the eigenvalues of $A$. (Note that several arguments will contain other matrices named $A$ with various subscripts. When we refer to $A$ alone, it will exclusively refer to the adjacency matrix of $G$.)

 For ease of notation, throughout the paper we use
\begin{equation}
\lambda = (k-1)^2+1 \quad \text{and} \quad \lambda' = (k-1)^2-1. \label{eq:lambdas}
\end{equation}

\begin{proposition}\label{prop:EY}
Let $d \ge 2$ and $k \ge 3$ be integers. There is a constant $C_1 = C_1(d,k)$ such that
\[
\ex Y \sim C_1 \left( 2\pi n\right)^{-(k-1)|V|/2} \left( k^{|V|} \left(\frac{k-1}{k}\right)^{|E|} \right)^n,
\]
where
\[ C_1 = k^{k|V|/2} \left(\frac{(k-1)^2}{k(k-2)}\right)^{(k-1)|E|/2}.
\]
\end{proposition}

\begin{proposition}\label{prop:expectation_y_squared}
Suppose $d < \ell_k$. Then there is a constant $C_2=C_2(d,k)$ such that
\[
\ex Y^2 \sim C_2 (2\pi n)^{-(k-1)|V|} \left( k^{|V|} \left(\frac{k-1}{k}\right)^{|E|} \right)^{2n}
\]
where
\[ C_2 =
\frac{ k^{(k^2-k+1)|V|}
(k-1)^{(2k^2-2k)|E|}
}{(\lambda)^{\frac{1}{2}(k-1)^2|E|}
(\lambda')^{\frac{1}{2}(k^2-1)|E|}
h(d,k)^{\frac{(k-1)^2}{2}}}
\]
and
\begin{equation}\label{eq:det_hessian}
h(d,k) = \left(\frac{k^2}{\lambda\lambda'}\right)^{|V|} \prod_{i=1}^{|V|} (\lambda\lambda' + d - \alpha_i(k-1)^2)
\end{equation}
\end{proposition}

For fixed $j \ge 3$, denote the number of $j$-cycles in a random lift by $Z_j$.

\begin{proposition}\label{prop:all_joint_moments}
For $i=1,\ldots, |V|$, let ${\beta_i}^+$ and ${\beta_i}^-$ denote the roots of the quadratic $x^2-\alpha_ix+d-1 = 0$. That is,
\begin{equation}\label{eq:beta_def}
{\beta_i}^+ = \tfrac{1}{2}\alpha_i + \sqrt{\tfrac{1}{4}\alpha_i^2-(d-1)} \quad \text{and} \quad \beta_i^- = \tfrac{1}{2}\alpha_i - \sqrt{\tfrac{1}{4}\alpha_i^2-(d-1)}.
\end{equation}
For all $j \ge 3$ and $p_3, \ldots, p_j$, each a non-negative integer,
\begin{equation}
\label{eq:fac_moms}
\frac{\ex(Y[Z_3]_{p_3}\cdots[Z_j]_{p_j})}{\ex(Y)} \sim \prod_{\ell=3}^j (\lambda_\ell(1+\delta_\ell))^{p_\ell}
\end{equation}
where
\begin{equation}\label{eq:sscm_constants}
\lambda_j = \frac{(|E|-|V|)(1+(-1)^j)+\sum_{i=1}^{|V|}((\beta_i^+)^j+(\beta_i^-)^j)}{2j}, \quad
\delta_j = \frac{(-1)^j}{(k-1)^{j-1}},
\end{equation}
and $[A]_b$ indicates the falling factorial moment $A(A-1)\cdots(A-b+1)$.
\end{proposition}

We now prove Theorem~\ref{thm:upper_bound} using the small subgraph conditioning method of Robinson and Wormald. (See \cite{Wor99} for a full exposition of the method.) We must verify that
\[ \frac{\ex Y^2}{(\ex Y)^2} \sim \exp\left(\sum_{j \ge 1} \lambda_j\delta_j^2\right). \]

\begin{proof}
We begin with the right-hand side. We have
\begin{align*}
\sum_{j\ge1} \lambda_j\delta_j^2 &= \sum_{j\ge1} \lambda_j \cdot \frac{1}{(k-1)^{2j-2}}\\
	&= \frac{(k-1)^2}{2}\left( (|E|-|V|) \sum_{j \ge 1} \frac{1}{j(k-1)^{2j}} +  (|E|-|V|)  \sum_{j \ge 1} \frac{(-1)^j}{j(k-1)^{2j}} \quad + \right.\\
	& \qquad \qquad + \quad \left. \sum_{j \ge 1} \sum_{i=1}^{|V|} \frac{(\beta_i^+)^j+(\beta_i^-)^j}{j(k-1)^{2j}} \right)\\
	&= \frac{(k-1)^2}{2}\left( (|E|-|V|) \sum_{j \ge 1} \frac{1}{j(k-1)^{2j}} + (|E|-|V|)  \sum_{j \ge 1} \frac{(-1)^j}{j(k-1)^{2j}} \quad + \right.\\
	& \qquad \qquad + \quad \left. \sum_{i=1}^{|V|} \left(\sum_{j \ge 1} \frac{(\beta_i^+)^j}{j(k-1)^{2j}}+ \sum_{j \ge 1}\frac{(\beta_i^-)^j}{j(k-1)^{2j}} \right) \right)\\
	&= \frac{(k-1)^2}{2} \log\left( \left(\frac{(k-1)^2}{(k-1)^2-1}\right)^{|E|-|V|} \left(\frac{(k-1)^2}{(k-1)^2+1}\right)^{|E|-|V|} \quad \times \right.\\
	&\qquad \qquad \times \quad \left. \prod_{i=1}^{|V|} \left(\frac{(k-1)^2}{(k-1)^2-\beta_i^+} \right)\left(\frac{(k-1)^2}{(k-1)^2-\beta_i^-}\right)\right)\\
	&= \frac{(k-1)^2}{2} \log\left( \left(\frac{(k-1)^2}{\lambda'}\right)^{|E|-|V|} \left(\frac{(k-1)^2}{\lambda}\right)^{|E|-|V|} \quad \times \right.\\
	&\qquad \qquad \times \quad \left. \prod_{i=1}^{|V|} \left(\frac{(k-1)^4}{(k-1)^4-(\beta_i^++\beta_i^-)(k-1)^2+\beta_i^+\beta_i^-} \right) \right)\\
	&= \frac{(k-1)^2}{2} \log\left( \left(\frac{(k-1)^2}{\lambda'}\right)^{|E|-|V|} \left(\frac{(k-1)^2}{\lambda}\right)^{|E|-|V|} \quad \times \right.\\
	&\qquad \qquad \times \quad \left. \prod_{i=1}^{|V|} \left(\frac{(k-1)^4}{\lambda\lambda'+d-\alpha_i(k-1)^2} \right) \right)
\end{align*}
Exponentiating,
\begin{align*}
\exp\left(\sum\limits_{j \ge 1} \lambda_j\delta_j^2\right) &= \left( \frac{((k-1)^2)^{|E|-|V|+|E|-|V|+2|V|}}{(\lambda\lambda')^{|E|-|V|}\prod_{i=1}^{|V|} (\lambda\lambda'+d-\alpha_i(k-1)^2)}\right)^{\frac{1}{2}(k-1)^2}\\
	&= \left( \frac{((k-1)^2)^{2|E|}}{(\lambda\lambda')^{|E|-|V|}\prod_{i=1}^{|V|} (\lambda\lambda'+d-\alpha_i(k-1)^2)}\right)^{\frac{1}{2}(k-1)^2}
\end{align*}

Now we turn to the left-hand side of the proposition. Using Propositions~\ref{prop:EY} and \ref{prop:expectation_y_squared},
\[
\frac{\ex Y^2}{(\ex Y)^2} \sim \frac{C_2 (2\pi n)^{-(k-1)|V|} \left( k^{|V|} \left(\frac{k-1}{k}\right)^{|E|} \right)^{2n}}{\left(C_1 (2 \pi n)^{-(k-1)|V|/2} \left( k^{|V|} \left(\frac{k-1}{k}\right)^{|E|} \right)^n\right)^2} = \frac{C_2}{C_1^2}.
\]
and
\begin{align*}
\frac{C_2}{C_1^2} &= \frac{ k^{(k^2-k+1)|V|-\frac{1}{2}(k^2-1)|E|}(k-1)^{(2k^2-2k)|E|}}
{\lambda^{\frac{1}{2}(k-1)^2|E|}(k-2)^{\frac{1}{2}(k^2-1)|E|}h(d,k)^{\frac{(k-1)^2}{2}}\left(k^{k|V|/2}
\left(\frac{(k-1)^2}{k(k-2)}\right)^{(k-1)|E|/2}\right)^2}\\
	&= \frac{k^{(k-1)^2|V|-\frac{1}{2}(k-1)^2|E|}(k-1)^{2(k-1)^2|E|}}
{\lambda^{\frac{1}{2}(k-1)^2|E|}(k-2)^{\frac{1}{2}(k-1)^2|E|}h(d,k)^{\frac{(k-1)^2}{2}}}\\
	&=\frac{k^{(k-1)^2|V|}}{\left(\left(\frac{k^2}{\lambda\lambda'}\right)^{|V|}
\prod_{i=1}^{|V|}(\lambda\lambda'+d-\alpha_i(k-1)^2)\right)^{\frac{(k-1)^2}{2}}} \cdot \left( \frac{(k-1)^{4|E|}}{(\lambda\lambda')^{|E|}}\right)^{(k-1)^2/2}\\
	&=\left( \frac{((k-1)^2)^{2|E|}}{(\lambda\lambda')^{|E|-|V|}\prod_{i=1}^{|V|}(\lambda\lambda'+d-\alpha_i(k-1)^2)}\right)^{\frac{1}{2}(k-1)^2}
\end{align*}
from which we conclude
\[ \frac{\ex Y^2}{(\ex Y)^2} \sim \exp\left(\sum_{j \ge 1} \lambda_j\delta_j^2\right) \]
as required.
\end{proof}


\section{Useful tools}\label{sec:tools}

In this section we prove two results which, in addition to being necessary for our arguments, may prove to be of independent interest.

First, in~\cite{AN05}, Achlioptas and Naor investigate the chromatic number of a random graph of average bounded degree, $G(n,d/n)$. In that article they prove a theorem regarding the optimization of square stochastic matrices which has
also been used in the context of uniform random regular graphs
(see~\cite{AM04,KPW10}). We extend their result to non-square stochastic matrices.

Then, in~\cite{GJR10}, Greenhill, Janson and Ruci\'nski write about the number of perfect matchings in random lifts of graphs. To facilitate their calculations, they prove a theorem for estimating a summation over multiple dimensions using Laplace's method.  We reformulate their result in terms of counting maximal forests in an auxiliary graph that encodes the summation constraints. This reformulation unveils the role of the eigenvalues of this auxiliary graph, in view of the well-known matrix-tree theorem, and simplifies the calculations.

\subsection{Optimization over stochastic matrices}

In this section we will introduce some inequalities regarding stochastic
matrices. A matrix $M$ is called \emph{row-stochastic} if all of its entries are
nonnegative and the sum of the entries in each row equals~$1$ (i.e.~each row of
$M$ defines a probability distribution). Similarly, $M$ is
\emph{column-stochastic} if $M^T$ is row-stochastic, and $M$ is
\emph{doubly-stochastic} if it is both row- and column-stochastic.
Stochastic matrices arise naturally when estimating the second moment of the
number of colourings of a random graph or in other related problems. For
instance, suppose that $C_1$ and $C_2$ are two equitable $k$-vertex-colourings of the
same graph (i.e.~colourings in which all the colour classes are of the same size).
Then we can describe how much $C_1$ and $C_2$ correlate in terms of a
doubly-stochastic matrix $M=(m_{i,j})_{i,j\in[k]}$, where each entry $m_{i,j}$
denotes the (appropriately rescaled) proportion of vertices that receive colour
$i$ in $C_1$ and colour $j$ in $C_2$.
In second moment calculations, one typically needs to determine which pairs of
colourings $C_1,C_2$ give the main contribution, which can be formulated as an
optimization problem over the set of doubly-stochastic $k\times k$ matrices.

The first inequality in this section, concerning square row-stochastic matrices,
was proved by Achlioptas and Naor~\cite{AN05}. They used this tool to obtain an
accurate second moment estimate of the number of colourings of an
Erd\H{o}s-R\'enyi random graph of constant average degree. Their inequality has
also been used in the context of uniform random regular graphs
(see~\cite{AM04,KPW10}).
Before stating the result, we need some notation.
Given a $q\times k$ matrix $M=(m_{i,j})_{i\in[q],j\in[k]}$, let
\[
\rho(M) = \sum_{i=1}^q \sum_{j=1}^k m_{i,j}^2
\qquad\text{and}\qquad
h(M) = - \sum_{i=1}^q \sum_{j=1}^k m_{i,j}\log m_{i,j}.
\]
If in addition $M$ is row-stochastic, then a standard convexity argument shows that
\begin{equation}\label{eq:rhobounds}
\frac{q}{k} \le \rho(M) \le q,
\end{equation}
where the minimum $\rho(M)=q/k$ is uniquely attained when all $m_{i,j}=1/k$, and
the maximum $\rho(M)=q$ is achieved precisely at those row-stochastic matrices
with all entries in $\{0,1\}$.
For each integer $q\ge3$, define
\begin{equation}\label{eq:cq}
c_q = \frac{(q-1)^3}{q(q-2)} \log(q-1).
\end{equation}
\begin{theorem}[Achlioptas and Naor~\cite{AN05}]\label{thm:achlioptas-naor}
Let $M=(m_{i,j})$ be a $q \times q$ row-stochastic matrix, where $q\ge3$.
Then, for any $c < c_q$,
\[
\frac{1}{q}h(M) + c\log\left( q^2 - 2q + \rho(M) \right) \le \log q + c\log\left((q-1)^2\right),
\]
Furthermore, we have equality if and only if $m_{i,j} = \frac{1}{q}$ for all $i,j\in[q]$.
\end{theorem}
\begin{remark}
The second claim in the proposition is not explicitly stated in~\cite{AN05}.
However, it follows from their proof that equality holds if and only if
$\rho(M)=1$, which in view of~\eqref{eq:rhobounds} corresponds to the case when
all $m_{i,j} = \frac{1}{q}$.
\end{remark}
We will use Theorem~\ref{thm:achlioptas-naor} in our second moment
calculations in Section~\ref{sec:second} for the number $Y$ of strongly
equitable $k$-colourings of a random lift of an arbitrary $d$-regular graph. Our
estimate of $\ex Y^2$  will involve solving an optimization problem
which we will manage to reduce to the inequality in
Theorem~\ref{thm:achlioptas-naor}.
Furthermore, and somewhat surprisingly, while Achlioptas and Naor's inequality
typically arises in second moment arguments, we will require a more general
version of the inequality to handle a first moment calculation. Because
we must concern ourselves with both a random lift and a random coloring of that
lift, this first moment calculation takes on more of a second moment flavor.
When we bound the expected number $X$ of (not necessarily equitable)
$k$-colorings of a random lift of $K_{d+1}$ in Section~\ref{sec:first}, we will
use a generalization of Theorem~\ref{thm:achlioptas-naor} that applies to
rectangular $q\times k$ row-stochastic matrices.

\begin{proposition}\label{prop:achlioptas-naor_corollary}
Let $M = (m_{i,j})$ be a $q \times k$ row-stochastic matrix, where $q \ge 3$ and $k \le q$. Then, for any $c < \frac{k-1}{q-1} c_q$,
\[
\frac{1}{q} h(M) + c\log\left(kq - k - q + \frac{k}{q}\, \rho(M)  \right) \leq \log k + c\log\left((q-1)(k-1)\right)
\]
or
\[
c\log\left(1+\frac{\frac{k}{q}\rho(M)-1}{(q-1)(k-1)}\right) \leq \log k - \frac{1}{q} h(M)
\]
with equality if and only if $m_{i,j} = \frac{1}{k}$.
\end{proposition}
\begin{proof}
Our first step is to extend $M$ to a $q \times q$ row-stochastic matrix $\hat M$ by adding $q-k$ extra columns and scaling each column appropriately. More precisely, let $\hat{M} = (\hat m_{i,j})_{i,j\in[q]}$ with
\[
\hat{m}_{i,j} = \begin{cases} m_{i,j} \cdot \frac{k}{q} & 1 \leq j \leq k \\
\frac{1}{q} & k+1 \leq j \leq q. \end{cases}
\]
Clearly, $\hat{M}$ is a $q \times q$ row-stochastic matrix. Hence, we can apply Theorem~\ref{thm:achlioptas-naor} to $\hat M$ with $\hat c := \frac{q-1}{k-1} c < c_q$, and obtain
\begin{equation}\label{eq:AN}
\hat c\log\left(1+\frac{\rho(\hat M)-1}{(q-1)^2}\right) \le \log q - \frac{1}{q}h(\hat M),
\end{equation}
which holds with equality if and only if $\hat m_{i,j} = \frac{1}{q}$ (which is in turn equivalent to $m_{i,j} = \frac{1}{k}$).
Our next task is to derive relations between $\rho(\hat M)$ and $\rho(M)$ and also between $h(\hat M)$ and $h(M)$, which will allow us to express the inequality in~\eqref{eq:AN} in terms of $\rho(M)$ and $h(M)$.
We have
\begin{align}
\rho(\hat M) &= \sum_{i,j=1}^q \hat m_{i,j}^2 \notag\\
	&= \frac{k^2}{q^2} \sum_{i,j=1}^k m_{i,j}^2 + q\, (q-k)\, \frac{1}{q^2} \notag\\
	&= \frac{k}{q}\left(\frac{\rho(M) k}{q}\right) + \frac{q-k}{q} \notag\\
	&= \frac{k}{q}\left(\frac{\rho(M) k}{q} - 1 \right) + 1 \label{eq:rho}
\end{align}
Given a probability distribution $x=(x_1,\ldots,x_\ell)$, let $H(x) = -\sum_{j=1}^k x_j \log x_j$ denote the entropy of $x$.
For each $i\in[q]$, let $m_i = (m_{i,1}, \ldots, m_{i,k})$ and $\hat m_i = (\hat m_{i,1}, \ldots, \hat m_{i,q})$. These are the $i$-th rows of matrices $M$ and $\hat M$, which we regard as probability distributions. Also, let $u=(\frac{1}{q-k},\ldots,\frac{1}{q-k})$ be the uniform distribution on a set of $q-k$ elements.

Note that $\hat m_i$ is a mixture of $m_i = (m_{i,1}, \ldots, m_{i,k})$ and $u=(\frac{1}{q-k},\ldots,\frac{1}{q-k})$, in the sense that we can sample from $\hat m_i$ by sampling from $m_i$ with probability $\frac{k}{q}$ and from $u$ with probability $\frac{q-k}{q}$.
Hence, by the chain rule,
\begin{align*}
H(\hat m_i) &= \frac{k}{q} H(m_i) + \frac{q-k}{q} H(u) + H(\tfrac{k}{q},\tfrac{q-k}{q})
\\
	&= \frac{k}{q} H(m_i) + \frac{q-k}{q} \log(q-k) + \frac{k}{q} (\log q - \log k) + \frac{q-k}{q}(\log q - \log (q-k))
	\\
	&= \log q - \frac{k}{q} \log k + \frac{k}{q} H(m_i).
\end{align*}
Noting that $h(M) = \sum_{i=1}^q H(m_i)$ and $h(\hat M) = \sum_{i=1}^q H(\hat m_i)$, we conclude
\begin{align}
 \log k - \frac{1}{q} h(M)
&= \frac{1}{k} \sum_{i=1}^q \left( \frac{k}{q} \log k - \frac{k}{q} H(m_i)\right)
\notag\\
&= \frac{1}{k} \sum_{i=1}^q \left(\log q - H(\hat m_i)\right)
= \frac{q}{k} \left( \log q - \frac{1}{q} h(\hat M) \right).
\label{eq:hM}
\end{align}
In view of~\eqref{eq:rho} and~\eqref{eq:hM}, we can use the inequality
in~\eqref{eq:AN} to write
\begin{align*}
\log k - \frac{1}{q} h(M) &= \frac{q}{k} \left( \log q - \frac{1}{q} h(\hat M) \right)\\
	&\ge \frac{q}{k} \hat c\log\left(1+\frac{\rho(\hat{M})-1}{(q-1)^2}\right)\\
	&= \frac{q}{k} \hat c\log\left(1+\frac{\frac{k}{q}\left(\frac{\rho(M) k}{q} - 1 \right)}{(q-1)^2}\right)
\end{align*}
Note that $\frac{k}{q}\left(\frac{\rho(M) k}{q} - 1 \right) > - \frac{k}{q} > -(q-1)^2$.
Since $f(x)=\log\left(1+\frac{x}{(q-1)^2}\right)$ is a concave function for $x>-(q-1)^2$ and $f(0)=0$, Jensen's inequality gives
\[
\log\left(1+\frac{\frac{k}{q}\left(\frac{\rho(M) k}{q} - 1 \right)}{(q-1)^2}\right)
\ge \frac{k}{q} \log\left(1+\frac{\left(\frac{\rho(M) k}{q} - 1 \right) }{(q-1)^2}\right),
\]
which combined with the inequality above yields

\begin{equation}\label{eq:almostthere}
\log k - \frac{1}{q} h(M) \ge \hat c \log\left(1+\frac{\left(\frac{\rho(M) k}{q} - 1 \right) }{(q-1)^2}\right).
\end{equation}
In order to further bound the right-hand side from below, we introduce a new function
\[
g(y) = y\log\left(1+\frac{a}{y}\right),
\qquad\text{where}\quad
a = \frac{1}{q-1}\left(\frac{\rho(M) k}{q}-1\right).
\]
From~\eqref{eq:rhobounds}, we know $0 \le a \le 1$. We claim that $g(y)$ is nondecreasing for all $y>0$.
Indeed, its derivative satisfies
\[
g'(y) = -\log\left(1 - \frac{a}{y+a}\right) - \frac{a}{y+a} \ge 0,
\]
where we've used $\log(1+x)\le x$ for $x>-1$. Hence, from~\eqref{eq:almostthere},
\[
\log k - \frac{1}{q} h(M)
\ge \frac{\hat c}{q-1} \, g(q-1) \ge \frac{\hat c}{q-1} \, g(k-1)
= c \log\left(1+\frac{\frac{\rho(M) k}{q}-1}{(q-1)(k-1)}\right),
\]
which is equivalent to
\[
\frac{1}{q} h(M) + c\log\left(kq - k - q + \frac{k}{q}\, \rho(M)  \right) \leq \log k + c\log\left((q-1)(k-1)\right),
\]
as desired. Finally, in view of our remark below~\eqref{eq:AN}, equality holds
if and only if all $m_{i,j}=1/k$. This completes the proof of the proposition.
\end{proof}

\subsection{Laplace summation over lattices}
\label{ssec:laplace_summation}

Here we will prove Proposition~\ref{prop:laplacian_summation_gamma}, which we will use in Sections \ref{ssec:equi_expectation} and \ref{ssec:outer} to precisely estimate expectations expressed as sums, as a consequence of a similar result of Greenhill, Janson and Ruci\'nski~\cite{GJR10}.

We include here the definitions required to state their theorem, but refer the reader to~\cite{GJR10} for additional background.
A {\em lattice} $\Lambda$ is an additive subgroup of $\real^N$ such that every
bounded region in $\real^N$ contains a finite number of elements of $\Lambda$.
It is well-known that every lattice $\Lambda$ is isomorphic to $\ent^r$, for
some integer $r$ ($0\le r\le N$) which we call the {\em rank} of $\Lambda$.
Moreover, every lattice $\Lambda$ of rank $r\ge 1$ admits a {\em basis}
$u_1,\ldots,u_r \in \Lambda$ such that every point in $\Lambda$ can be uniquely
represented as a linear combination $m_1u_1+\cdots+m_ru_r$ with integer
coefficients $m_i$.

Given a lattice $\Lambda$ with a basis $u_1,\ldots,u_r$, we
define the {\em determinant} of the lattice as
\[
\det(\Lambda) = \sqrt{\det U^TU},
\]
where $U$ is an $N\times r$ matrix with columns $u_1,\ldots,u_r$. This quantity does not depend on our choice of the basis. If $\mathbb V \subseteq \mathbb{R}^N$ is an $r$-dimensional vector space spanning $\Lambda$ (and thus with basis $u_1,\ldots,u_r$), for any symmetric $N\times N$ matrix $H$ we define
\begin{equation}
\det(H|_{\mathbb V}) = \frac{ \det U^T H U}{\det U^TU}.
\label{eq:detHV}
\end{equation}
This value also does not depend on our choice of basis.

\newpage

\begin{theorem}[Greenhill, Janson and Ruci\'nski~\cite{GJR10}]\label{laplacian_summation}
Suppose the following:
\begin{enumerate}[(i)]
\item $\Lambda \subset \real^N$ is a lattice with rank $1\le r \le N$.
\item $\mathbb V \subseteq \real^N$ is the $r$-dimensional subspace spanned by $\Lambda$.
\item $\mathbb W = \mathbb V + \mbf w$ is an affine subspace parallel to $\mathbb V$, for some $\mbf w \in \real^N$.
\item $K \subset \real^N$ is a compact convex set with non-empty interior $K^\circ$.
\item $\phi : K \to \real$ is a continuous function and the restriction of $\phi$ to $K \cap \mathbb W$ has a unique maximum at some point $\mbf z_0 \in K^\circ \cap \mathbb W$.
\item $\phi$ is twice continuously differentiable in a neighbourhood of $\mbf z_0$ and $H := D^2\phi(\mbf z_0)$ is its Hessian at $\mbf z_0$.
\item $\psi : K_1 \to \real$ is a continuous function on some neighbourhood $K_1 \subseteq K$ of $\mbf z_0$ with $\psi(\mbf z_0) > 0$.
\item For each positive integer $n$ there is a vector $\ell_n \in \real^N$ with $\ell_n/n \in \mathbb W$,
\item For each positive integer $n$ there is a positive real number $b_n$ and a function \linebreak $a_n : (\Lambda + \ell_n) \cap nK \to \real$ such that, as $n \to \infty$,
\[ a_n(\ell) = O(b_n e^{n\phi(\ell/n)+o(n)}),\quad \ell \in (\Lambda +\ell_n) \cap nK \]
and
\[ a_n(\ell) = b_n(\psi(\ell/n) + o(1)) e^{n\phi(\ell/n)}, \quad \ell \in (\Lambda +\ell_n) \cap nK_1,\]
uniformly for $\ell$ in the indicated sets.
\end{enumerate}
Then provided $\det(-H|_{\mathbb V} )\neq 0$, as $n \to \infty$,
\[ \sum_{\ell\in(\Lambda+\ell_n)\cap nK}a_n(\ell) \sim \frac{(2\pi)^{r/2}\psi(\mbf z_0)}{\det(\Lambda)\det(-H|_{\mathbb V})^{1/2}}b_nn^{r/2}e^{n\phi(\mbf z_0)}. \]
\end{theorem}

In order to state our proposition, we need some additional definitions. Let
$\Gamma=(V_\Gamma, E_\Gamma)$ be a non-empty multigraph (possibly with multiple
edges, but no loops). Fix an arbitrary orientation of the edges in $E_\Gamma$.
The {\em signed incidence matrix} of $\Gamma$  (with respect to that
orientation) is a $|V_\Gamma|\times |E_\Gamma|$ matrix $\tilde D=(\tilde
D_{v,e})_{v\in V_\Gamma, e\in E_\Gamma}$, where $\tilde D_{v,e}=1$ if $v$ is the
tail of $e$, $\tilde D_{v,e}=-1$ if $v$ is the head of $e$ and $\tilde
D_{v,e}=0$ otherwise. Similarly, the {\em unsigned incidence matrix} of $\Gamma$
is a matrix $D=(D_{v,e})_{v\in V_\Gamma, e\in E_\Gamma}$, where $D_{v,e}=|\tilde
D_{v,e}|$, and does not depend on the orientation of the edges. Finally, recall
the definition of $\det(H|_{\mathbb V})$ from~\eqref{eq:detHV}.

\begin{proposition}\label{prop:laplacian_summation_gamma}
Suppose the following:
\begin{enumerate}[(i)]
\item $\Gamma=(V_\Gamma, E_\Gamma)$ is a non-empty bipartite multigraph with at least one cycle.
\item $D$ is the unsigned incidence matrix of $\Gamma$ .
\item $\tau(\Gamma)$ is the number of maximal forests in $\Gamma$.
\item $\mathbb V = \Ker(D) \subseteq \mathbb{R}^{|E_\Gamma|}$ is a vector space of dimension $r$.
\item $\mbf y\in\mathbb{R}^{|V_\Gamma|}$ such that
\begin{equation}
D\mbf x = \mbf y
\label{eq:Dxy}
\end{equation}
is a consistent linear system.
\item $K \subset \mathbb{R}^{|E_\Gamma|}$ is a compact convex set with non-empty interior $K^\circ$.
\item $\phi : K \to \mathbb{R}$ is a continuous function and the maximum of $\phi$ in $K$ subject to~\eqref{eq:Dxy} is attained at a unique maximizer $\mbf{\hat x} \in K^\circ$.
\item $\phi$ is twice continuously differentiable in a neighbourhood of $\mbf{\hat x}$ and $H$ is its Hessian matrix at $\mbf{\hat x}$.
\item $\psi : K_1 \to \mathbb{R}$ is a continuous function on some neighbourhood $K_1 \subseteq K$ of $\mbf{\hat x}$ with $\psi(\mbf{\hat x}) > 0$.
\item For each positive integer $n$,
\[
\mathbb X_n = \left\{ \mbf x\in K \cap \frac{1}{n} \ent^{|E_\Gamma|} : D\mbf x = \mbf y \right\}
\]
is non-empty, and there is a positive real number $c_n$ and a function  $T_n : \mathbb X_n \to \mathbb{R}$ such that, as $n \to \infty$,
\[ T_n(\mbf x) = O( c_n e^{n\phi(\mbf x)+o(n)}),\quad \mbf x \in \mathbb X_n \]
and
\[ T_n(\mbf x) = c_n (\psi(\mbf x) + o(1)) e^{n\phi(\mbf x)}, \quad \mbf x \in \mathbb X_n \cap K_1,\]
uniformly for $\mbf x$ in the indicated sets.
\end{enumerate}
Then provided $\det(-H|_{\mathbb V} )\neq 0$, as $n \to \infty$,
\[
\sum_{\mbf x \in \mathbb X_n} T_n(\mbf x) \sim \frac{\psi(\mbf{\hat x})}{\tau(\Gamma)^{1/2} \det(-H|_{\mathbb V})^{1/2}} (2\pi n)^{r/2} c_n e^{n\phi(\mbf{\hat x})}.
\]
Furthermore, in the case that $\Gamma$ is not bipartite, the proposition remains valid if we replace $D$ by the signed incidence matrix $\tilde D$ of $\Gamma$ (with respect to a fixed orientation of $E_\Gamma$).
\end{proposition}

\begin{proof}[Proof of Proposition~\ref{prop:laplacian_summation_gamma}]

We start by proving the more general result where $\Gamma$ is not bipartite in which we use the signed incidence matrix $\tilde{D}$.

Let $\Lambda = \mathbb V \cap \ent^{|E_\Gamma|}$ so that $\Lambda$ is the set of
all integer solutions of $\tilde D\mbf x = \mbf 0$. Clearly, $\Lambda$ is a
lattice that spans $\mathbb V$. As it spans $\mathbb V$, we have $\dim \mathbb V
= \rank\Lambda = r$. As $\Gamma$ is nonempty, $D$ is not the zero matrix and so
$r \ge 1$.

Set $\mathbb W = \{\mbf x \, : \, \tilde{D}\mbf{x}= \mbf{y}\}$. As we are guaranteed $\tilde{D}\mbf{x}= \mbf{y}$ is consistent, there is some $\mbf w \in \real^N$ such that $\tilde{D}\mbf{w}= \mbf{y}$. For any $\mbf v \in \mathbb V$ we have
\[ \tilde{D}(\mbf v + \mbf w) = \tilde{D}\mbf v + \tilde{D}\mbf w = 0 + \mbf{y} = \mbf y \]
so $\mathbb V + \mbf w \subseteq \mathbb W$. Furthermore, for any $\mbf x \in \mathbb W$ we have
\[ \tilde{D}(\mbf x - \mbf w) = \tilde{D}\mbf x - \tilde{D}\mbf w = \mbf{y} - \mbf{y}= 0 \]
so $\mbf x - \mbf w \in \mathbb V$ and $\mbf x = (\mbf x - \mbf w) + \mbf w \in \mathbb V + \mbf w$. We conclude $\mathbb W = \mathbb V + \mbf w$ is an affine subspace parallel to $\mathbb V$.

Note that the conditions on $K, \phi$ and $\psi$ in Proposition~\ref{prop:laplacian_summation_gamma} exactly match those of Theorem~\ref{laplacian_summation} with $\mbf{\hat{x}}$ replaced with $\mbf{z}_0$.

For each positive integer $n$ we have that $\mathbb X_n$ is nonempty, so choose $\mbf{x}_n \in \mathbb X_n$ and set $\ell_n = n\mbf{x}_n$. Then $\ell_n/n = \mbf{x}_n$ satisfies $\tilde{D}\mbf{x}_n = \mbf y$ by the definition of $\mathbb X_n$ so $\ell_n/n \in \mathbb W$.

Let $b_n = c_n$ and define $a_n: (\Lambda + \ell_n) \cap nK \to \real$ by $a_n(\ell) = T(\ell/n)$. To see $a_n$ is well-defined, let $\ell \in (\Lambda + \ell_n) \cap nK$ and take $\mbf x = \frac{1}{n}\ell$. Then certainly $\mbf x \in K$. As $\ell_n = n\mbf{x}_n$, we have $\mbf x \in \frac{1}{n}\Lambda + \mbf{x}_n$, and as $\Lambda \subseteq \ent^{|E_\Gamma|}$ and $\mbf{x}_n \in \mathbb X \subseteq \frac{1}{n}\ent$ we get $\mbf x \in \frac{1}{n} \ent$ as well. Finally, setting $\mbf{x} = \frac{1}{n}\lambda + \mbf{x}_n$ for some $\lambda \in \Lambda$, we have
\[ \tilde{D}\mbf x = \tilde{D}(\tfrac{1}{n}\lambda + \mbf{x}_n) = \tfrac{1}{n} \tilde{D} \lambda + \tilde{D}\mbf{x}_n = 0+\mbf y = \mbf y \]
from which we conclude $\frac{1}{n}\ell = \mbf x \in \mathbb{X}_n$ and $a_n$ is
well-defined. Furthermore,
\[ a_n(\ell) = T_n(\mbf x) = O(c_ne^{n\phi(\mbf x)+o(n)}) =
O(b_Ne^{n\phi(\ell/n)+o(n)}), \quad \ell \in (\Lambda + \ell_n) \cap nK \]
and
\[ a_n(\ell) = T_n(\mbf x) = c_n(\psi(\mbf x)+o(1))e^{n\phi(\mbf{x})} = b_n(\psi(\ell/n) + o(1)) e^{n\phi(\ell/n)}, \quad \ell \in (\Lambda +\ell_n) \cap nK_1 \]
uniformly for $\ell$ as required.

Thus as Proposition~\ref{prop:laplacian_summation_gamma} requires $\det(-H|_{\mathbb V}) \neq 0$, we apply Theorem~\ref{laplacian_summation} to get
\begin{align*}
\sum_{x \in \mathbb X_n} T_n(\mbf x) &=  \sum_{\ell\in(\Lambda+\ell_n)\cap nK}a_n(\ell)\\
	&\sim \frac{(2\pi)^{r/2}\psi(\mbf z_0)}{\det(\Lambda)\det(-H|_{\mathbb V})^{1/2}}b_nn^{r/2}e^{n\phi(\mbf z_0)}\\
	&=\frac{\psi(\mbf{\hat{x}})}{\det(\Lambda)\det(-H|_{\mathbb V})^{1/2}}(2\pi n)^{r/2}c_ne^{n\phi(\mbf{\hat x})}
\end{align*}
All that remains to complete the proof is to show $\det(\Lambda) = \tau(\Gamma)^{1/2}$. Lemma 14.7.3 in \cite{GR01} gives this result in the case where $\Gamma$ is connected. To extend the result to disconnected $\Gamma$ we can simply apply the result to each component.

Lastly, we consider the case where $\Gamma$ is bipartite in which we claim we may use the unsigned incidence matrix $D$ rather than $\tilde{D}$. Select one of the parts, say $L \subseteq V_\Gamma$, and orient $E_\Gamma$ from $L$ toward the other part $R = V_\Gamma \setminus L$. Then the coefficients in each of the first $|L|$ rows of $\tilde{D}$ are positive while the last $|R|$ rows contain all negative entries. Thus one can obtain $D$ from $\tilde D$ by multiplying the last $R$ rows by $-1$. These are elementary row operations and do not change the kernel, so we conclude $\Ker D = \Ker \tilde{D}$ and our definition of $\mathbb V$ is unchanged. Now suppose $\mbf y$ is the vector we're given such that $D \mbf x = \mbf y$ is a consistent linear system. Define $\mbf{\tilde{y}}$ by
\[ \tilde{y}_i = \begin{cases} y_i & 1 \le i \le |L| \\ -y_i & |L|+1 \le i \le |L|+ |R|\end{cases} \]
Then $\tilde{D}\mbf x = \mbf{\tilde{y}}$ if and only if $D\mbf x = \mbf y$ so we proceed with the rest of the proof above replacing $\mbf y$ with $\mbf{\hat{y}}$.

\end{proof}

We prove an additional result that demonstrates all but $o(1)$ of the weight of the sums of such functions occurs very close to the maximum value. 

\begin{proposition}\label{prop:central_sum}
Suppose the same set of conditions hold as in Proposition~\ref{prop:laplacian_summation_gamma}. In addition, suppose there is $N > 0$ such that $\psi(x) \le n^N$ for each $\mbf x \in K_1$. For any $\gamma > 0$, define $\mathbb{Y}_n(\gamma) \subseteq \mathbb{X}_n$ as
\[ \mathbb{Y}_n(\gamma) = \left\{ \mbf x \in \mathbb{X}_n : ||\mbf{\hat{x}} - \mbf x||_{\infty} < \gamma \frac{\log n}{\sqrt{n}} \right\}. \]
Then
\[ \sum_{\mbf x \in \mathbb{Y}_n(\gamma)} T_n(\mbf x)  \sim \sum_{\mbf x \in \mathbb{X}_n} T_n(\mbf x). \]
\end{proposition}

\begin{proof}
For any $\gamma > 0$, $\mbf{\hat{x}} \in \mathbb{Y}_n(\gamma)$ maximizes $\phi$ in $K^\circ$ and thus it suffices to show
\[ \sum_{\mbf x \in \mathbb{X}_n \setminus \mathbb{Y}_n(\gamma)} T_n(\mbf x) = o\left(e^{n \phi(\mbf{\hat{x}})}\right). \]
As $\mbf{\hat{x}}$ is a local maximum, $\nabla \phi = 0$ at $\mbf{\hat x}$. By Taylor's Theorem, as the Hessian matrix of $\phi$ is nonvanishing (because $\det(-H|_{\mathbb{V}}) \ne 0$), there is $\varepsilon > 0$ such that if $\delta := ||\mbf{\hat x} - \mbf x||_{\infty} < \varepsilon$ then $ \phi(\mbf x) = \phi(\mbf{\hat x}) - \Theta(\delta^2)$. We partition $\mathbb{X}_n \setminus \mathbb{Y}_n(\gamma) = A \cup B$ where $A$ is the annulus consisting of $\mbf x$ satisfying
\[ \gamma\frac{\log n}{\sqrt{n}} < ||\mbf{\hat x} - \mbf x||_{\infty} < \varepsilon. \]
For any $\mbf x \in A$, we have
\[ \phi(\mbf x) =\phi(\mbf{\hat x}) - \Theta(\delta^2) = \phi(\mbf{\hat x}) - \Omega\left(\gamma^2\frac{\log^2 n}{n}\right) \]
 and therefore
\begin{align*}
\sum_{\mbf x \in A} T_n(\mbf x) &= \sum_{\mbf x \in A} c_n(\psi(\mbf x)+o(1)) \exp\left(n\left(\phi(\mbf{\hat x})-\Omega\left(\gamma^2\frac{\log^2 n}{n}\right)\right)\right)\\
	&= \sum_{\mbf x \in A} c_n(\psi(\mbf x)+o(1)) n^{-\gamma^2\log n} e^{n\phi(\mbf{\hat x})}
\end{align*}
As there are at most $|E_{\gamma}|$ many terms in the sum and $n^{|E_{\Gamma}|+N} = o\left(n^{\gamma^2\log n}\right)$
we see
\[ \sum_{\mbf x \in A} T_n(\mbf x) = o\left(e^{n \phi(\mbf{\hat{x}})}\right). \]

Now consider $\mbf x \in B$. These $\mbf x$ satisfy $||\mbf x - \mbf{\hat x}||_\infty \ge \varepsilon$. As $\phi$ is a continuous function it attains a maximum value on the compact set $K \setminus B_{\varepsilon}(\mbf{\hat{x}})$. But $\phi$ has a unique maximum in $K^\circ$ at $\mbf{\hat x} \notin B$. Therefore there is some $\alpha > 0$ such that
\[ \max_{\mbf x \in B} \phi(\mbf x) < \phi(\mbf{\hat x}) - \alpha. \]
Therefore
\begin{align*}
\sum_{\mbf x \in B} T_n(\mbf x) &< \sum_{\mbf x \in B} c_n(\psi(\mbf x)+o(1)) \exp(n(\phi(\mbf{\hat x})-\alpha))\\
	&= O\left(n^{|E_{\gamma}|+N}\right) e^{-n\alpha} e^{n\phi(\mbf{\hat x})}\\
	&= o\left(e^{n \phi(\mbf{\hat{x}})}\right). 
\end{align*}
\end{proof}

\section{First moment ingredients}
\label{sec:first}

In this section we give two first moment arguments, one for $X$, the number of $k$-colourings of a random lift of $K_{d+1}$, and another for $Y$, the number of strongly equitable $k$-colourings of a random lift of an arbitrary $d$-regular graph.

\subsection{Colouring Optimization and Proof of Proposition~\ref{prop:EX}}\label{ssec:opti1}

Let $X$ be a random variable denoting the number of proper $k$-colourings of a
random lift $L$ of $K_{d+1}$. The main goal of this section is to prove
Proposition~\ref{prop:EX} by providing an upper bound on $\ex X$. Additionally, some
of the ideas developed here will be utilized again in
Section~\ref{ssec:equi_expectation} when we study equitable $k$-colourings of a
random lift of a general $d$-regular graph. Throughout this section, we use $V$
and $E$ to denote the vertex and edge sets of $K_{d+1}$, respectively, so in
particular $|V|=d+1$ and $|E|=\binom{d+1}{2}$. For convenience, we fix an
arbitrary orientation of the edges in $E$, so that for each pair of different
vertices $v,v'\in V$ exactly one of $vv'$ and $v'v$ belongs to $E$.

To calculate $\ex X$, we will count the number of pairs $(L,C)$ such that $L$ is a lift of $K_{d+1}$ and $C$ is a proper $k$-colouring of $L$, and then divide by the
number of lifts of $K_{d+1}$.
First observe that we can build any lift $L$ of $K_{d+1}$ by selecting a perfect
matching between the fibers $\Pi^{-1}(v)$ and $\Pi^{-1}(v')$, for each edge
$vv'\in E$. As a result, there are exactly $n!^{|E|}$ possible lifts.
Let us fix a pair $(L,C)$, where $L$ is a lift of $K_{d+1}$ and $C$ is a proper
$k$-colouring of $L$. For each $v \in V$, we consider a (row) vector
$a_v=(a_{v,i})_{i\in [k]}$, where $a_{v,i}$ denotes the proportion of vertices
in fiber $\Pi^{-1}(v)$ that receive colour $i$. By construction, the entries of
each $a_v$ are in $\frac{1}{n}\ent$ and define a probability distribution on the
vertices in the fiber, which is to say
\begin{equation}
a_{v,i} \ge 0 \quad \forall v\in V, i\in[k]
\and
\sum_{i\in[k]} a_{v,i} = 1 \quad \forall v\in V.
\label{eq:acons1}
\end{equation}
We write $\mbf{a} = (a_v)_{v\in V}\in \frac{1}{n}\ent^{k|V|}$, which we also
regard as a $|V|\times k$ stochastic matrix where each row $a_v\in
\frac{1}{n}\ent^k$ is a probability distribution. Similarly, for each $e=vv'\in
E$, we define $b_e = (b_{e,i,i'})_{i,i'\in[k], i\ne i'}$, where each
$b_{e,i,i'}$ denotes the proportion of edges in $\Pi^{-1}(e)$ that connect a
vertex of colour $i$ in $\Pi^{-1}(v)$ to a vertex of colour $i'$ in
$\Pi^{-1}(v')$. We write $\mbf{b} = (b_e)_{e \in E}$. The entries of each $b_e$
must be in $\frac{1}{n}\ent$ and satisfy
\begin{align}
b_{e,i,i'} &\ge 0 \quad \forall e\in E, i,i'\in[k], i\ne i'
\notag
\\
\sum_{i'\ne i} b_{e,i,i'} &= a_{v,i} \quad \forall e=vv'\in E, i\in [k]
\notag
\\
\sum_{i\ne i'} b_{e,i,i'} &= a_{v',i'} \quad \forall e=vv'\in E, i'\in [k].
\label{eq:abcons1}
\end{align}
For each choice of parameters $\mbf a \in \frac{1}{n}\ent^{k|V|}$ and $\mbf b\in \frac{1}{n}\ent^{k(k-1)|E|}$ satisfying~\eqref{eq:acons1} and~\eqref{eq:abcons1}, we will enumerate all pairs $(L,C)$ which agree on those parameters and then sum over all possible choices of $\mbf a$ and $\mbf b$.

Given $\mbf{a}$ and $\mbf{b}$, we generate a pair $(L,C)$ in three steps. First, we assign colours to the vertices of $L$ so that, for each $v\in V$ and $i\in[k]$, exactly $a_{v,i}n$ vertices in $\Pi^{-1}(v)$ receive colour $i$. There are
\[
\prod_{v\in V} \frac{n!}{\prod_{i\in[k]}(a_{v,i} n)!} = \prod_{v\in V}\binom{n}{a_v n}
\]
ways to do this. (Recall the definitions for factorials and binomials of vectors in \eqref{eq:factorial}.)

Next, for every $e=vv' \in E$ and distinct colours $i,i'\in[k]$, we need to decide which sets of $b_{e,i,i'}n$ vertices in $\Pi^{-1}(v)$ and $\Pi^{-1}(v')$ will be matched. We can do that in
\[
\proda{e\in E\\e=vv'} \bigg( \prod_{i\in[k]} \frac{(a_{v,i} n)!}{ \prod_{i'\ne i} (b_{e,i,i'} n)!} \bigg) \bigg( \prod_{i'\in[k]} \frac{(a_{v',i'} n)!}{ \prod_{i\ne i'} (b_{e,i,i'} n)!} \bigg)
= \proda{e\in E\\e=vv'} \frac{(a_v n)!}{(b_{e} n)!} \frac{(a_{v'} n)!}{(b_{e} n)!}
\]
many ways.
Finally, we need to choose a perfect matching between these sets, which can be done in
\[
\proda{e\in E}  \proda{i,i'\in[k]\\ i\ne i'}  (b_{e,i,i'} n)! = \proda{e\in E}  (b_e n)! 
\]
different ways. Putting everything together, we get
\begin{equation}
\ex X = \frac{1}{n!^{|E|}} \sum_{\mbf{a}, \mbf{b}} \prod_{v
\in V} \binom{n}{a_v n}
\proda{e\in E\\e=vv'} \frac{(a_v n)! (a_{v'} n)!}{(b_{e} n)!},
\label{eq:EX_pre_stirling}
\end{equation}
where the sum is over all $\mbf a \in \frac{1}{n}\ent^{k|V|}$ and $\mbf b\in \frac{1}{n}\ent^{k(k-1)|E|}$ satisfying~\eqref{eq:acons1} and~\eqref{eq:abcons1}.
Recall that one can write Stirling's formula as $x! = \xi(x)(x/e)^x$, where $\xi$ is a function satisfying $\xi(x)\sim\sqrt{2\pi x}$ as $x\to\infty$ and $\xi(x)\ge1$ for all $x\ge0$.
Then, using~\eqref{eq:acons1} and~\eqref{eq:abcons1}, and after some tedious but simple calculations,
\begin{align}
\ex X &= \big(\xi(n) (n/e)^n\big)^{-|E|} \sum_{\mbf{a}, \mbf{b}} \prod_{v
\in V} \frac{\xi(n) (n/e)^n}{ \prod_{i} \xi(a_{v,i} n) (a_{v,i} n/e)^{a_{v,i} n}} \quad \times
\notag\\&\qquad\qquad \times\quad
\proda{e\in E\\e=vv'} \frac{ \big( \prod_{i}\xi(a_{v,i} n) (a_{v,i} n/e)^{a_{v,i} n} \big) \big( \prod_{i'} \xi(a_{v',i'} n) (a_{v',i'} n/e)^{a_{v',i'} n} \big)}{\proda{i,i'\\i\ne i'}\xi(b_{e,i,i'} n) (b_{e,i,i'} n/e)^{b_{e,i,i'} n}}
\notag\\
	&= \sum_{\mbf{a}, \mbf{b}} \prod_{v \in V} \frac{\xi(n) }{ \prod_{i} \xi(a_{v,i} n) {a_{v,i}}^{a_{v,i} n}} \quad \times \notag\\
	&\qquad\qquad\times\quad
\proda{e\in E\\e=vv'} \frac{ \big( \prod_{i}\xi(a_{v,i} n) \big) \big( \prod_{i'} \xi(a_{v',i'} n) \big) \left( \proda{i,i'\\i\ne i'} (a_{v,i} a_{v',i'})^{b_{e,i,i'} n} \right)}
{\xi(n)  \proda{i,i'\\i\ne i'} \xi(b_{e,i,i'} n) {b_{e,i,i'}}^{b_{e,i,i'} n}}
\notag\\
&= \sum_{\mbf{a}, \mbf{b}} p(\mbf a, \mbf b, n) e^{n f(\mbf a, \mbf b)},
\label{eq:EX}
\end{align}
where
\begin{equation}
p(\mbf a, \mbf b, n) = \prod_{v \in V} \frac{\xi(n) }{ \prod_{i} \xi(a_{v,i} n)}
\proda{e\in E\\e=vv'} \frac{ \big( \prod_{i}\xi(a_{v,i} n) \big) \big( \prod_{i'} \xi(a_{v',i'} n) \big)} {\xi(n) \proda{i,i'\\i\ne i'}\xi(b_{e,i,i'} n)}
\label{eq:pab}
\end{equation}
and
\begin{equation}
f(\mbf a, \mbf b) = 
- \sum_{v \in V} \sum_{i}  a_{v,i} \log a_{v,i} + 
\suma{e\in E\\e=vv'} \suma{i,i'\\i\ne i'}  b_{e,i,i'} \log \left(\frac{ a_{v,i} a_{v',i'}} {b_{e,i,i'}}\right).
\label{eq:fab}
\end{equation}
In order to bound the exponential behaviour of $\ex X$, we will maximize $f(\mbf a, \mbf b)$ and show that the main contribution to the sum in~\eqref{eq:EX} comes from the term where all the parameters are equal.
Let $\hat a = (\hat a_i)_{i\in[k]}$ with all $\hat a_i=1/k$, and $\hat b = (\hat b_{i,i'})_{i,i'\in[k],i\ne i'}$ with all $\hat b_{i,i'}=1/k(k-1)$.
Also, define $\mbf{\hat a} = (\hat a)_{v\in V}$ and $\mbf{\hat b} = (\hat b)_{e\in E}$.
Note that $\mbf{\hat a}, \mbf{\hat  b}$ are only a valid set of parameters provided that $k(k-1)\mid n$, since otherwise their entries are not in $\frac{1}{n}\ent$.  However, since we are only interested in an upper bound on $\ex X$ we may proceed with these values by bounding $\ex X$ in a larger space that contains $\frac{1}{n}\ent$.
\begin{proposition}\label{prop:lower_bound_optimization}
Let $d$ and $k$ be integers such that $d \ge 2$ and $\frac{d^2-1}{d \log d} < 2 (k-1)$.
Let $f(\mbf a, \mbf b)$ be defined as in~\eqref{eq:fab}. Then the maximum of $f(\mbf a, \mbf b)$ for $\mbf a\in\real^{k|V|}$ and $\mbf b\in\real^{k(k-1)|E|}$ subject to~\eqref{eq:acons1} and~\eqref{eq:abcons1}
is uniquely attained at the point in which $\mbf a = \mbf{\hat a}$ and $\mbf b = \mbf{\hat b}$, and equals
\[
f(\mbf{\hat a}, \mbf{\hat  b}) = \log\left( \frac{ (k-1)^d }{k^{d-2}}\right)^{(d+1)/2}.
\]
\end{proposition}
Before proving Proposition~\ref{prop:lower_bound_optimization}, we need the following technical result.

\begin{lemma}\label{lem:technical_bound}
Let $d,k \ge 1$ be integers. For any $\mbf a = (a_v)_{v\in V}$ with $a_v=(a_{v,i})_{i\in[k]}\in\real^k$ satisfying~\eqref{eq:acons1},
\[ 
\suma{e\in E\\e=vv'} \log\big(1-\inp{a_v,a_{v'}}\big) \le \binom{d+1}{2} \log \left( 1 - \frac{d+1}{dk} +  \frac{1}{d(d+1)}\sum_{v \in V} \inp{a_v, a_v}  \right).
\]
\end{lemma}

\begin{proof}
Let $c = \sum_{v\in V} a_v$ and $j=(1,\ldots,1)\in\real^{k}$. From~\eqref{eq:acons1}, we have that $\inp{j, a_v}=1$ and thus  $\inp{j, c}=d+1$.
In particular,
\begin{equation}
1-\inp{a_v,a_{v'}} \ge 1- \inp{a_v,j} = 0,
\label{eq:pos}
\end{equation}
so the left-hand side of the inequality in the lemma is well-defined (but possibly $-\infty$).
Moreover, by the Cauchy-Schwarz inequality,
\[
\inp{c, c} = \frac{1}{k} \inp{j, j} \inp{c, c} \ge  \frac{1}{k} \inp{j,c}^2 = \frac{(d+1)^2}{k}.
\]
Hence, by writting
\[
\rho = \sum_{v\in V} \inp{a_v, a_v},
\]
we obtain
\begin{equation}
\suma{v,v'\in V\\v\neq v'} \inp{a_v, a_{v'}} = \sum_{v,v'\in V} \inp{a_v, a_{v'}} - \sum_{v\in V} \inp{a_v, a_v} = \inp{c,c}  - \rho \ge  \frac{(d+1)^2}{k}  - \rho.
\label{eq:inpavs}
\end{equation}
Before we proceed to prove our main inequality, recall that $E$ is the edge set of $K_{d+1}$ (with a fixed orientation). Thus we can write
\begin{align*}
\suma{e\in E\\e=vv'} \log\big(1-\inp{a_v,a_{v'}}\big)
&= \frac{1}{2} \suma{v,v'\in V\\v\ne v'} \log\big(1-\inp{a_v,a_{v'}}\big)
\\
&= \frac{1}{2} \log  \proda{v,v'\in V\\v\ne v'} \big(1-\inp{a_v,a_{v'}}\big)
\\
&= \binom{d+1}{2} \log  \left( \proda{v,v'\in V\\v\ne v'} \big(1-\inp{a_v,a_{v'}}\big) \right)^{1/d(d+1)}.
\end{align*}
We apply the inequality of arithmetic and geometric means to all $1-\inp{a_v,a_{v'}}$ above (which are non-negative, as observed in~\eqref{eq:pos}), and conclude that
\[
\suma{e\in E\\e=vv'} \log\big(1-\inp{a_v,a_{v'}}\big)
\le \binom{d+1}{2} \log  \left( \frac{1}{d(d+1)} \suma{v,v'\in V\\v\ne v'} \big(1-\inp{a_v,a_{v'}}\big) \right).
\]
Combining this and~\eqref{eq:inpavs} yields
\[
\suma{e\in E\\e=vv'} \log\big(1-\inp{a_v,a_{v'}}\big)
\le \binom{d+1}{2} \log  \left( \frac{1}{d(d+1)}  \Big( d(d+1) - \frac{(d+1)^2}{k}  + \rho \Big) \right),
\]
which completes the proof of the lemma.
\end{proof}

\begin{proof}[Proof of Proposition~\ref{prop:lower_bound_optimization}]
Assume throughout the proof that $\frac{d^2-1}{d \log d} < 2 (k-1)$.
Fix any $\mbf a\in\real^{k|V|}$ satisfying~\eqref{eq:acons1}, and let
\[
g(\mbf a) = h(\mbf a) +
\binom{d+1}{2} \log \left( \frac{d(d+1) - (d+1)^2/k + \rho(\mbf a)}{d(d+1)} \right),
\]
where
\[
h(\mbf a) = - \sum_{v \in V} \sum_{i\in[k]} a_{v,i} \log a_{v,i}
\and
\rho(\mbf a) = \sum_{v \in V} \inp{a_v, a_v}.
\]
We will maximize $f(\mbf a, \mbf b)$ for such fixed $\mbf a$ and with $\mbf b\in\real^{k(k-1)|E|}$ subject only to
\begin{equation}
b_{e,i,i'}\ge0
\and
\suma{i,i'\in[k]\\i\ne i'}b_{e,i,i'}=1.
\label{eq:brelaxedcons1}
\end{equation}
In view of this relaxation, we can regard each $b_e=(b_{e,i,i'})_{i,i'\in[k], i\ne i'}$ as an arbitrary probability distribution, and maximize each term
$\suma{i,i'\\i\ne i'} b_{e,i,i'}\log\left( \frac{a_{v,i}a_{v',i'}}{b_{e,i,i'}}\right)$ in~\eqref{eq:fab} separately.
For each $e=vv'\in E$, we define another probability distribution given by
\[
b^*_{e,i,i'} =  \frac{a_{v,i}a_{v',i'}}{z_e}, \qquad\text{for}\quad i,i'\in[k],i\ne i',
\]
where
\[
z_e = \suma{i,i'\\i\ne i'} a_{v,i}a_{v',i'}
\]
is the normalizing factor, and we write $b^*_e=(b^*_{e,i,i'})_{i,i'\in[k], i\ne i'}$.
Then
\[
\suma{i,i'\\i\ne i'} b_{e,i,i'}\log\left( \frac{a_{v,i}a_{v',i'}}{b_{e,i,i'}}\right) = \log z_e - D_{KL}(b_e \| b^*_e),
\]
where $D_{KL}(b_e \| b^*_e) = \suma{i,i'\\i\ne i'} b_{e,i,i'}\log\left( \frac{b_{e,i,i'}}{b^*_{e,i,i'}}\right)$ is the Kullback-Leibler divergence from $b_e$ to $b^*_e$.
By Gibb's inequality, $D_{KL}(b_e \| b^*_e)\ge0$ with equality iff and only if $b_e = b^*_e$. As a result,
\begin{equation}\label{eq:maxfb}
\max_{\mbf b\text{ s.t.~\eqref{eq:brelaxedcons1}}} f(\mbf a, \mbf b) =  h(\mbf a) + \sum_{e\in E} \log z_e,
\end{equation}
with one unique maximizer at $\mbf b = \mbf{b^*} := (b^*_e)_{e\in E}$. Note that if $\mbf a = \mbf{\hat a}$ then $\mbf{b^*}=\mbf{\hat b}$.
We proceed to bound the right-hand side of~\eqref{eq:maxfb}. In view of~\eqref{eq:acons1}, we can write
\[
\log z_e 
= \log \sum_{i\in[k]} a_{v,i} \left( 1 - a_{v',i} \right)
= \log \left( 1 - \sum_{i\in[k]} a_{v,i} a_{v',i} \right)
= \log ( 1 - \langle a_{v}, a_{v'}\rangle ),
\]
and then, by Lemma~\ref{lem:technical_bound},
\[
\sum_{e\in E} \log z_e = \suma{e\in E\\e=vv'} \log ( 1 - \langle a_{v}, a_{v'}\rangle )
\le \binom{d+1}{2} \log \left( 1 - \frac{d+1}{dk} +  \frac{1}{d(d+1)} \rho(\mbf a)  \right),
\]
so
\[
h(\mbf a) + \sum_{e\in E} \log z_e \le g(\mbf a).
\]
Combining this with~\eqref{eq:maxfb} yields
\begin{equation}\label{eq:maxfb2}
\max_{\mbf  b\text{ s.t.~\eqref{eq:brelaxedcons1}}} f(\mbf a, \mbf b) \le g(\mbf a).
\end{equation}
We will bound $g(\mbf a)$ by applying Proposition~\ref{prop:achlioptas-naor_corollary} to $\mbf a=(a_{v,i})_{v\in V,i\in[k]}$, which is a row-stochastic $(d+1)\times k$ matrix.
Note that our assumption $\frac{d^2-1}{d \log d} < 2 (k-1)$ implies
\[
\frac{d}{2}  <  \frac{k-1}{d} \frac{d^3 \log d}{d^2-1} = \frac{k-1}{d} c_{d+1},
\]
where $c_{d+1}$ is defined as in~\eqref{eq:cq}.
Therefore, Proposition~\ref{prop:achlioptas-naor_corollary} (with $q=d+1$, $c=d/2$ and $A=\mbf a$) yields
\[
\frac{d}{2}\log\left( \frac{d(k-1) + \frac{k}{d+1}\rho(\mbf a) - 1}{d(k-1)} \right) \le \log k - \frac{1}{d+1}h(\mbf a),
\]
with equality iff $\mbf a = \mbf{\hat a}$. After some elementary manipulations, we can rewrite this as
\[
h(\mbf a) + \binom{d+1}{2}\log\left( \frac{d(k-1) + \frac{k}{d+1}\rho(\mbf a) - 1}{dk} \right)
\le (d+1)\log k + \binom{d+1}{2}\log\left( \frac{d(k-1)}{dk} \right),
\]
which immediately implies
\[
h(\mbf a) + \binom{d+1}{2}\log\left( 1 - \frac{d+1}{dk} + \frac{ \rho(\mbf a)}{d(d+1)} \right)
\le (d+1)\log k + \binom{d+1}{2}\log\left( \frac{k-1}{k} \right).
\]
Therefore, noting that $\rho(\mbf{\hat a}) = (d+1)/k$,
\[
g(\mbf a) \le g(\mbf{\hat a}) = \log \left( \frac{(k-1)^d}{k^{d-2}} \right)^{(d+1)/2} = f(\mbf{\hat a},\mbf{\hat b}),
\]
with equality iff $\mbf a = \mbf{\hat a}$.
Combining this with~\eqref{eq:maxfb2}, we obtain the desired bound
\begin{equation*}\label{eq:ga}
\max_{\mbf  a, \mbf  b\text{ s.t.~\eqref{eq:acons1} and~\eqref{eq:abcons1}}} f(\mbf a, \mbf b)
\le \max_{\mbf  a\text{ s.t.~\eqref{eq:acons1}}}  \max_{\mbf  b\text{ s.t.~\eqref{eq:brelaxedcons1}}} f(\mbf a, \mbf b) \le \max_{\mbf  a\text{ s.t.~\eqref{eq:acons1}}} g(\mbf a) = f(\mbf{\hat a}, \mbf{\hat b}).
\end{equation*}
Note that for $\mbf a \ne \mbf{\hat a}$, we have $g(\mbf a) < g(\mbf{\hat a})$ and thus $f(\mbf a, \mbf b) < f(\mbf{\hat a}, \mbf{\hat b})$. Furthermore, if $\mbf a = \mbf{\hat a}$ but $\mbf b \ne \mbf{\hat b}$ then, recalling that the maximum in~\eqref{eq:maxfb} is uniquely attained at $\mbf{b^*}=\mbf{\hat b}$, we conclude that $f(\mbf{\hat a}, \mbf b) < g(\mbf{\hat a})  = f(\mbf{\hat a}, \mbf{\hat b})$. This finishes the proof.
\end{proof}

We now proceed to prove the main result in this section.
\begin{proof}[Proof of Proposition~\ref{prop:EX}]
Recall from~\eqref{eq:EX} that
\[
\ex X = \sum_{\mbf{a}, \mbf{b}} p(\mbf a, \mbf b, n) e^{n f(\mbf a, \mbf b)},
\]
where we sum over all $\mbf a \in \frac{1}{n}\ent^{k|V|}$ and $\mbf b\in \frac{1}{n}\ent^{k(k-1)|E|}$ satisfying~\eqref{eq:acons1} and~\eqref{eq:abcons1}. In particular, a crude upper bound on the number of terms is given by
\[
O\left(n^{k|V|+k(k-1)|E|}\right),
\]
since each entry in $\mbf a$ or $\mbf b$ can take at most $n+1$ values.
Moreover, from~\eqref{eq:pab} and since $1\le \xi(x) = O(\sqrt x)$ as $x\to\infty$, we can bound the polynomial factor $p(\mbf a, \mbf b, n)$ in each term by
\begin{equation}
p(\mbf a, \mbf b, n) = O\left( n^{|V|+2k|E|} \right),
\label{eq:boundp}
\end{equation}
where the hidden constant in the big $O$ notation does not depend on $\mbf a$ or $\mbf b$. Hence, by Proposition~\ref{prop:lower_bound_optimization},
\[
\ex X \le e^{n f(\mbf{\hat a}, \mbf{\hat b})} \sum_{\mbf{a}, \mbf{b}} p(\mbf a, \mbf b, n)
=
O\left(n^{(d+1)(k+1)+\binom{d+1}{2}k(k+1)} \right)
\left( \frac{ (k-1)^d }{k^{d-2}}\right)^{(d+1)n/2},
\]
which completes the proof of the proposition.
\end{proof}

\subsection{Strongly equitable colourings. (Proof of Proposition~\ref{prop:EY}.)} \label{ssec:equi_expectation}

In this section we prove Proposition~\ref{prop:EY} regarding the expected number of
strongly equitable colourings of a random lift. Here we allow $G$ to be any fixed
$d$-regular graph, not necessarily $G=K_{d+1}$, and always assume that $n$ is
divisible by $k$. Let $Y$ be the number of strongly equitable $k$-colourings of a
random $n$-lift of $G$. 

Note that the only place in Section~\ref{ssec:opti1} where we used the fact that $G=K_{d+1}$ was in the proof of Lemma~\ref{lem:technical_bound}. Therefore, equation~\eqref{eq:EX} is still valid for a general $d$-regular graph $G$. In particular, restricting the sum to the terms in which $\mbf a=\mbf{\hat a}$ gives the expected number of strongly equitable colourings, that is
\begin{equation}
\ex Y = \sum_{\mbf{b}} p(\mbf{\hat a}, \mbf b, n) e^{n f(\mbf{\hat a}, \mbf b)},
\label{eq:EY}
\end{equation}
with
\begin{align}
p(\mbf{\hat a}, \mbf b, n) &= 
\left(\frac{\xi(n) }{  \xi(n/k)^k}\right)^{|V|} \left(\frac{\xi(n/k)^{2k} }{  \xi(n)}\right)^{|E|}
\prod_{e\in E}\proda{i,i'\\i\ne i'} \frac{1} {\xi(b_{e,i,i'} n)}
\label{eq:pabequi}
\\
f(\mbf a, \mbf b) &= 
|V| \log k - |E|\log k^2 - \suma{e\in E} \suma{i,i'\\i\ne i'}  b_{e,i,i'} \log b_{e,i,i'},
\label{eq:fabequi}
\end{align}
and where the sum is over all $\mbf b\in \frac{1}{n}\ent^{k(k-1)|E|}$ satisfying
\begin{align}
b_{e,i,i'} &\ge 0 \quad \forall e\in E, i,i'\in[k], i\ne i'
\notag
\\
\sum_{i'\ne i} b_{e,i,i'} &= 1/k \quad \forall e\in E, i\in [k]
\notag
\\
\sum_{i\ne i'} b_{e,i,i'} &= 1/k \quad \forall e\in E, i'\in [k].
\label{eq:bcons1equi}
\end{align}
Moreover, the discussion in the proof of Proposition~\ref{prop:lower_bound_optimization} leading to~\eqref{eq:maxfb} still holds for general $G$, and gives for $\mbf a=\mbf{\hat a}$
\[
\max_{\mbf b\text{ s.t.~\eqref{eq:brelaxedcons1}}} f(\mbf{\hat a}, \mbf b) =  h(\mbf{\hat a}) + \sum_{e\in E} \log z_e =  |V|\log k + |E| \log\left(\frac{k-1}{k}\right) = f(\mbf{\hat a}, \mbf{\hat b}),
\]
with one unique maximizer at $\mbf b = \mbf{\hat b}$. (Note that we did not use Lemma~\ref{lem:technical_bound} to bound $\sum_{e\in E} \log z_e$, since we are only concerned about the case $\mbf a = \mbf{\hat a}$.) Since $\mbf{\hat b}$ trivially satisfies~\eqref{eq:bcons1equi}, we obtain the following analogue of Proposition~\ref{prop:lower_bound_optimization} for general $d$-regular $G$ but restricted to strongly equitable colourings.
\begin{proposition}\label{prop:lower_bound_optimization_G_equitable}
Let $f(\mbf{\hat a}, \mbf b)$ be defined as in~\eqref{eq:fab}. Then the maximum of $f(\mbf{\hat a}, \mbf b)$ subject to~\eqref{eq:bcons1equi}
is uniquely attained at the point in which $\mbf b = \mbf{\hat b}$, and equals
\[
f(\mbf{\hat a}, \mbf{\hat b}) =  \log \left( k^{|V|} \left(\frac{k-1}{k}\right)^{|E|} \right).
\]
\end{proposition}

We will estimate the sum in~\eqref{eq:EY} by using our version of Laplace's method, Proposition~\ref{prop:laplacian_summation_gamma}.

\begin{proof}[Proof of Proposition~\ref{prop:EY}]

We begin by defining a bipartite graph $\Gamma=\Gamma(V_\Gamma, E_\Gamma)$ so that we may express the equality constraints in~\eqref{eq:bcons1equi} in terms of its unsigned incidence matrix $D$ and use Proposition~\ref{prop:laplacian_summation_gamma}.
The idea is to associate each equation in~\eqref{eq:bcons1equi} to a vertex of $\Gamma$ and every variable to an edge in a way that preserves the incidence relations. To do this, we assign label $w_{e,1,i}$ to equation $\sum_{i'\ne i} b_{e,i,i'} = 1/k$ and label $w_{e,2,i'}$ to equation $\sum_{i\ne i'} b_{e,i,i'} = 1/k$.
The vertex set of $\Gamma$ is $V_\Gamma = V_{\Gamma,1} \cup  V_{\Gamma,2}$, where
\[
 V_{\Gamma,1} = \{ w_{e,1,i} : e\in E, i\in[k] \}
\and
 V_{\Gamma,2} = \{ w_{e,2,i'} : e\in E, i'\in[k] \}
\]
are the two sides of a bipartition.
The edge set is
\[
E_\Gamma = \{ b_{e,i,i'}: e\in E, i,i'\in [k], i\ne i'\},
\]
where each edge $b_{e,i,i'}$ has endpoints $w_{e,1,i}$ and $w_{e,2,i'}$ (i.e.~the labels of the two equations in which variable $b_{e,i,i'}$ appears).
Then the equality constraints in~\eqref{eq:bcons1equi} are equivalent to
\begin{equation}
D \mbf b = \mbf y,
\label{eq:Dby}
\end{equation}
where $D$ is the unsigned incidence matrix of $\Gamma$ and $\mbf y$ is the vector in $\real^{|V_\Gamma|}$ whose entries are all $1/k$.
The equations in~\eqref{eq:Dby} are consistent, since they admit the solution
\begin{align}
b_{e,i,i'} &= 0 \quad \forall e\in E, i,i'\in[k], i'\notin\{i, i+1\}
\notag\\
b_{e,i,i+1} &= \frac{1}{k} \quad \forall e\in E, i\in[k],
\label{eq:sol1}
\end{align}
where we use arithmetic modulo $k$ for indices $i,i'$.
We observe a few easy facts about $\Gamma$. First,
\[
|V_\Gamma| = 2k|E|
\and
|E_\Gamma| = k(k-1)|E|
\]
Also, $\Gamma$ has exactly $|E|$ connected components. More precisely, for each $e\in E$, the set of all vertices of the form $w_{e,1,i}$ or $w_{e,2,i'}$ induces a connected component of $\Gamma$. Each of these components is isomorphic to $F := K_{k,k}-M$, i.e.~the complete bipartite graph $K_{k,k}$ minus one perfect matching $M$. In particular, $\Gamma$ has at least one cycle (since $k\ge 3$). Since $\Gamma$ is bipartite, it is well-known (see~e.g.~Theorem~8.2.1 in~\cite{GR01}) that $D$ has rank $|V_\Gamma| - |E|$, and therefore $\mathbb V=\Ker(D)$ has dimension
\[
r = |E_\Gamma| - |V_\Gamma| + |E| = (k^2-3k+1) |E|. 
\]

Now we calculate $\tau(\Gamma)$. Since each maximal forest in $\Gamma$ is bijectively determined by selecting a spanning tree in each component, we conclude that the number of maximal forests in $\Gamma$ is
\[ \tau(\Gamma) = \tau(F)^{|E|}. \]
Naturally, next we count the number of spanning trees of $F$.

Let $I_k$ and $J_k$ denote the $k\times k$ identity matrix and the $k\times k$ matrix whose entries are all $1$s, respectively.
With an appropriate ordering of the vertices, the adjacency matrix of $F$ is
\[
A = \begin{bmatrix}0&1\\1&0\end{bmatrix} \otimes (J_k-I_k),
\]
and has eigenvalues
\[
\{ k-1, (1)_{k-1}, (-1)_{k-1}, 1-k \},
\]
where the subindices indicate multiplicities. Therefore, the Laplacian matrix
\[
Q = (k-1)I_{2k} - A
\]
of $F$ has eigenvalues
\[
\{ 0, (k-2)_{k-1}, (k)_{k-1}, 2k-2 \}.
\]
By Kirchhoff's Matrix Tree Theorem~\cite{BP98}, the number of spanning trees of $F$ is
\[
\tau(F) = \frac{1}{2k} (k-2)^{k-1} k^{k-1} (2k-2) =  (k-2)^{k-1} k^{k-2} (k-1)
\]
and thus
\begin{equation}
\tau(\Gamma) = \left( (k-1) k^{k-2} (k-2)^{k-1} \right)^{|E|}.
\label{eq:tau1}
\end{equation}

Let
\[
K = \{ \mbf b \in\real^{|E_\Gamma|} : 0 \le b_{e,i,i'} \le 1/k\} \]
and
\[
K_1 = \left\{ \mbf b \in\real^{|E_\Gamma|} : \tfrac{0.9}{k(k-1)} \le b_{e,i,i'} \le \tfrac{1.1}{k(k-1)}\right\}.
\]
Clearly, $K$ is a compact convex set with non-empty interior $K^\circ$, and any choice of \mbf b that satisfies \eqref{eq:bcons1equi} lies inside $K$. Let
\[
\phi(\mbf b) = f(\mbf{\hat a},\mbf b) = |V|\log k -  |E| \log k^2 - \sum_{e\in E} \suma{i,i'\\i\ne i'}  b_{e,i,i'} \log  b_{e,i,i'},
\]
which is continuous on $K$, and
\[
\psi(\mbf b) = 
\prod_{e\in E} \proda{i,i'\\i\ne i'} \frac{1} {\sqrt{ b_{e,i,i'} }},
\]
which is continuous and positive on $K_1$.
By Proposition~\ref{prop:lower_bound_optimization_G_equitable}, the maximum of $\phi(\mbf b)$ in $K$ subject to~\eqref{eq:Dby} is uniquely attained at $\mbf{\hat b} \in K_1 \subset K^\circ$. Moreover, $\phi(\mbf b)$ is twice continuously differentiable in the interior $K^\circ$, and its partial derivatives are
\[
\frac{\partial \phi}{\partial b_{e,i,i'}} = - (\log b_{e,i,i'} + 1)
\]
and
\[
\frac{\partial^2 \phi}{\partial b_{e,i,i'} \partial b_{e',j,j'}} =
\begin{cases}
- 1 / b_{e,i,i'} & \text{if $(e,i,i') = (e',j,j')$}
\\
0 & \text{otherwise}.
\end{cases}
\]
Hence, the Hessian matrix of $\phi$ at $\mbf b = \mbf{\hat b}$ is
\[
H = - k(k-1) I_{|E_\Gamma|}.
\]
Then, for any $|E_\Gamma| \times r$ matrix $U$ whose columns are a basis of $\mathbb V$,
\[
\det(-H|_{\mathbb V}) = \frac{ \det (- U^T H U)}{\det U^TU} = \frac{ \det ( k(k-1) U^T U)}{\det U^TU} = \big( k(k-1) \big)^r \ne 0.
\]
Let
\[
\mathbb X_n = \left\{ \mbf b\in K \cap \frac{1}{n} \ent^{|E_\Gamma|} : D\mbf b = \mbf y \right\}.
\]
The solution described in~\eqref{eq:sol1} belongs to $K$ and, since $k\mid n$, also to $\frac{1}{n} \ent^{|E_\Gamma|}$, so $\mathbb X_n$ is not empty. For each $\mbf b\in \mathbb X_n$, let
\[
T_n(\mbf v) = p(\mbf{\hat a}, \mbf b, n) e^{n f(\mbf{\hat a}, \mbf b)}
\]
and
\[
c_n = k^{k|V|/2 - k|E|}  (2\pi n)^{-(k-1)|V|/2 - r/2}.
\]
First, in view of~\eqref{eq:boundp}, for $\mbf b\in \mathbb X_n$,
\[
p(\mbf{\hat a}, \mbf b, n) / c_n = O(n^{|V|+2k|E|} /c_n) = e^{o(n)}
\]
and combined with Proposition~\ref{prop:lower_bound_optimization_G_equitable} this gives
\[ T_n(\mbf x) = c_n (p(\mbf{\hat a}, \mbf b, n) / c_n)e^{n f(\mbf{\hat a}, \mbf b)} = O(c_ne^{n\phi(\mbf{\hat b})+o(n)}). \]
Now note that we have chosen $c_n$ such that due to \eqref{eq:pabequi} and the fact that $\xi(x)\sim\sqrt{2\pi x}$ as $x\to\infty$, and after a few computations, we have that for $\mbf b\in \mathbb X_n \cap K_1$,
\[
p(\mbf{\hat a}, \mbf b, n) = c_n (\psi(\mbf b)+o(1)).
\]

Finally, as we've met all of the conditions of Proposition \ref{prop:laplacian_summation_gamma},
\begin{align*}
\ex Y &= \sum_{\mbf b \in \mathbb X_n} T_n(\mbf b)
\\
&\sim \frac{\psi(\mbf{\hat b})}{\tau(\Gamma)^{1/2} \det(-H|_{\mathbb V})^{1/2}} (2\pi n)^{r/2} c_n e^{n\phi(\mbf{\hat b})}
\\
&= \frac{ \left( k(k-1) \right)^{k(k-1)|E|/2}  k^{k|V|/2 - k|E|}}{\left( (k-1) k^{k-2} (k-2)^{k-1} \right)^{|E|/2} \big( k(k-1) \big)^{r/2}}  (2\pi n)^{-(k-1)|V|/2} \left( k^{|V|} \left(\frac{k-1}{k}\right)^{|E|} \right)^n
\\
&=  k^{k|V|/2} \left( \frac{  (k-1)^2 }{ k(k-2)  } \right)^{(k-1)|E|/2}   (2\pi n)^{-(k-1)|V|/2} \left( k^{|V|} \left(\frac{k-1}{k}\right)^{|E|} \right)^n.
\end{align*}
This completes the proof of Proposition~\ref{prop:EY}.
\end{proof}

\section{Second moment ingredients}
\label{sec:second}

As in the previous section, let $Y$ denote the number of strongly equitable
$k$-colourings of a random $n$-lift of $G$. In this section we continue to 
assume that $G$ is some fixed $d$-regular graph, not necessarily $K_{d+1}$, and that $k$ divides $n$. Our goal in this section is to prove Proposition~\ref{prop:expectation_y_squared}.

The proof is similar in nature to the proof of Proposition~\ref{prop:EY}. However,
in that proof we fixed $\mbf a = \mbf{\hat a}$ approximated the single summation over the $\mbf b$'s. In this proof, we will not be able to avoid a double summation.
We therefore require a more intricate argument, arranged as follows: In
Section~\ref{ssec:outline}, we give a counting argument for $\ex Y^2$ similar to
the argument in Section~\ref{ssec:opti1}. We next optimize the exponential contribution in Section \ref{ssec:optimization}. Then we approximate the inner sum in Section~\ref{ssec:inner} before completing the proof by approximating the outer sum in Section~\ref{ssec:outer}.

\subsection{Counting Argument} \label{ssec:outline}

In order to calculate $\ex Y^2$, we count pairs of balanced colourings. To each $v\in V$
we assign a $k\times k$ matrix $A_v = (a_{v,i,j})_{i=1}^k{}_{j=1}^k$ where $a_{v,i,j}$
is the proportion of the vertices in $\Pi^{-1}(v)$ that receive colour $(i,j)$ (that is, colour $i$ in the first
colouring and colour $j$ in the second colouring).
Each matrix $A_v$ must satisfy
\begin{align}
a_{v,i,j} &\ge 0, \quad \forall v\in V, i,j\in[k]
\notag\\
\sum_{j} a_{v,i,j} &= 1/k, \quad \forall v\in V, i\in[k]
\notag\\
\sum_{i} a_{v,i,j} &= 1/k, \quad \forall v\in V, j\in[k].
\label{eq:acons2}
\end{align}
In particular, $kA_v$ is a doubly-stochastic matrix.

For each $e=(v,v')\in E$, let $b_{e,i,j,i',j'}$ denote the proportion of edges in $\Pi^{-1}(e)$ which join a vertex of $\Pi^{-1}(v)$ with colour $(i,j)$ to a vertex of $\Pi^{-1}(v')$ with colour $(i',j')$.
Here $i,j,i',j'\in[k]$, but we require $i$ to be distinct from $i'$ and $j$ to be distinct from $j'$ to assure the colourings are proper.
Hence, to each $e=(v,v')\in E$, we assign a four-dimensional array $B_e = (b_{e,i,j,i',j'})_K$, where
\[
K = \{(i,j,i',j')\in [k]^4 : i\ne i', j\ne j'\}.
\]
Furthermore, each $B_e$ must satisfy:
\begin{align}
b_{e,i,j,i',j'}&\ge 0, \quad \forall e=(v,v')\in E, (i,j,i',j')\in K
\notag\\
\sum_{i'\ne i,j'\ne j}b_{e,i,j,i',j'}&=a_{v,i,j}, \quad \forall e=(v,v')\in E, i,j\in[k]
\notag\\
\sum_{i\ne i',j\ne j'}b_{e,i,j,i',j'}&=a_{v',i',j'}, \quad \forall e=(v,v')\in E, i',j'\in[k].
\label{eq:abcons2}
\end{align}
We will write $\mbf A = (A_v)_{v\in V}$ and $\mbf B = (B_e)_{e\in E}$ for short. Note that in addition to~\eqref{eq:acons2} and~\eqref{eq:abcons2}, each entry of $A_v$ and $B_e$ must be in $\frac{1}{n}\ent$.

In the following calculation, we obtain $\ex Y^2$  by summing, for every valid choice of $\mbf A$ and $\mbf B$,
the number of triples (lift, colouring 1, colouring 2) compatible with such $\mbf A$ and $\mbf B$ divided by the total number of lifts.
\begin{align}
\ex Y^2 &= \frac{1}{n!^{|E|}}
\sum_{\mbf A, \mbf B} \prod_{v\in V}\binom{n}{n A_v}
\prod_{e\in E}\frac{(n A_v)!(nA_{v'})!}{(nB_e)!} \notag
\\
&= \frac{1}{n!^{|E|}} \sum_{\mbf A} \prod_{v\in V}\binom{n}{n A_v}
\prod_{e\in E}(n A_v)!(nA_{v'})! \sum_{\mbf B}\prod_{e\in E}\frac{1}{(nB_e)!} \label{eq:ABsplit}\\
&= \sum_{\mbf A,\mbf B} \mathrm{poly}(n) \prod_{v\in V}\frac{1}{\prod_{i,j}{a_{v,i,j}}^{na_{v,i,j}}}
\prod_{e\in E}\prod_{(i,j,i',j')\in K} \left( \frac{a_{v,i,j}a_{v',i',j'}}{b_{e,i,j,i',j'}}\right)^{nb_{e,i,j,i',j'}}
 \notag \\
&= \sum_{\mbf A,\mbf B} \mathrm{poly}(n)
\exp\big( n f(\mbf A, \mbf B)\big) \label{eq:imprecise_formula}
\end{align}
where $\mathrm{poly}(n)$ is some function that is polynomial in $n$ and
\begin{equation}
f(\mbf A, \mbf B) = -\sum_{v\in V}\sum_{i,j}a_{v,i,j}\log a_{v,i,j}
+\sum_{e\in E}\sum_{i,j,i',j'} b_{e,i,j,i',j'}\log\left( \frac{a_{v,i,j}a_{v',i',j'}}{b_{e,i,j,i',j'}}\right) \label{eq:fAB}
\end{equation}
is the function we optimize in the next section.

\subsection{Optimization} \label{ssec:optimization}

We will show that the exponential part in $\ex Y^2$ is maximized by the term in which each $a_{v,i,j}=1/k^2$ and each $b_{e,i,j,i',j'}= 1/k^2(k-1)^2$.
We introduce some notation.
Let $\hat A = (\hat a_{i,j})_{i,j\in[k]}$ with all $\hat a_{i,j}=1/k^2$, and $\hat B = (\hat b_{i,j,i',j'})_{(i,j,i',j')\in K}$ with all $\hat b_{i,j,i',j'}=1/k^2(k-1)^2$.
We also write $\mbf{\hat A} = (\hat A,\ldots,\hat A)$ and $\mbf{\hat B} = (\hat B,\ldots,\hat B)$. As in Section \ref{ssec:opti1}, we note $\mbf{\hat A}$ and
$\mbf{\hat B}$ are only valid assignments to $\mbf{A}$ and $\mbf{B}$,
respectively, if $k^2(k-1)^2$ divides $n$, but as we merely seek an upper bound
for $f$ we may do so on a larger space that does include $\mbf{\hat A}$ and
$\mbf{\hat B}$.

\begin{proposition}\label{prop:exp_bound}
Suppose $d < \ell_k$. Let $f(\mbf A, \mbf B)$ be defined as in~\eqref{eq:fAB}. Then the maximum of $f(\mbf A, \mbf B)$ subject to~\eqref{eq:acons2} and~\eqref{eq:abcons2}
is uniquely attained at the point in which $\mbf A = \mbf{\hat A}$ and $\mbf B = \mbf{\hat B}$, and equals
\[
f(\mbf{\hat A}, \mbf{\hat  B}) = \log\left(k^{|V|}\left(\frac{k-1}{k}\right)^{|E|}\right)^{2n}.
\]
\end{proposition}

\begin{proof}
This is easier than for $\ex X$.

Fix any $\mbf A$ satisfying~\eqref{eq:acons2}, and let
\[
g(\mbf A) = \sum_{v \in V} \left( h(A_v) + \frac{d}{2}\log(1 - 2/k+\rho(A_v))\right),
\]
where
\[
h(A_v) = - \sum_{i,j\in[k]} a_{v,i,j} \log a_{v,i,j}
\qquad\text{and}\qquad
\rho(A_v) = \sum_{i,j\in[k]} {a_{v,i,j}}^2.
\]
We will maximize $f(\mbf A, \mbf B)$ for such fixed $\mbf A$ and with $\mbf B$ subject only to
\begin{equation}
b_{e,i,j,i',j'}\ge0, \qquad \sum_{i,j,i',j'}b_{e,i,j,i',j'}=1.
\label{eq:brelaxedcons}
\end{equation}
In view of this relaxation, we can regard each $B_e=(b_{e,i,j,i',j'})_{(i,j,i',j')\in K}$ as an arbitrary probability distribution, and maximize each term
$\sum_{i,j,i',j'} b_{e,i,j,i',j'}\log\left( \frac{a_{v,i,j}a_{v',i',j'}}{b_{e,i,j,i',j'}}\right)$ in~\eqref{eq:fAB} separately.
For each $e\in E$, we define another probability distribution given by
\[
b^*_{e,i,j,i',j'} =  \frac{a_{v,i,j}a_{v',i',j'}}{z_e}, \qquad\text{for}\quad (i,j,i',j')\in K,
\]
where
\[
z_e = \sum_{(i,j,i',j')\in K} a_{v,i,j}a_{v',i',j'}
\]
is the normalizing factor, and we write $B^*_e=(b^*_{e,i,j,i',j'})_{(i,j,i',j')\in K}$.
Then
\[
\sum_{i,j,i',j'} b_{e,i,j,i',j'}\log\left( \frac{a_{v,i,j}a_{v',i',j'}}{b_{e,i,j,i',j'}}\right) = \log z_e - D_{KL}(B_e \| B^*_e),
\]
where $D_{KL}(B_e \| B^*_e) = \sum_{i,j,i',j'} b_{e,i,j,i',j'}\log\left( \frac{b_{e,i,j,i',j'}}{b^*_{e,i,j,i',j'}}\right)$ is the Kullback-Leibler divergence from $B_e$ to $B^*_e$.
By Gibb's inequality, $D_{KL}(B_e \| B^*_e)\ge0$ with equality if and only if $B_e = B^*_e$. As a result,
\begin{equation}\label{eq:maxfB}
\max_{\mbf B\text{ s.t.~\eqref{eq:brelaxedcons}}} f(\mbf A, \mbf B) = \sum_{v\in V} h(A_v) + \sum_{e\in E} \log z_e,
\end{equation}
with one unique maximizer at $\mbf B = \mbf{B^*} := (B^*_e)_{e\in E}$. Note that if $\mbf A=\mbf{\hat A}$ then $\mbf{B^*}=\mbf{\hat B}$.
We proceed to bound $\log z_e$. Using inclusion-exclusion, the fact that $\mbf
A$ satisfies~\eqref{eq:acons2}, and the Cauchy-Schwarz inequality, we obtain
\begin{align*}
\log z_e &= \log \sum_{i,j\in[k]} a_{v,i,j} \left( \sum_{i',j'\in[k]} a_{v',i',j'} - \sum_{j'\in[k]} a_{v',i,j'} - \sum_{i'\in[k]} a_{v',i',j} + a_{v',i,j} \right)
\\
&= \log \sum_{i,j\in[k]} a_{v,i,j} \left( 1 - 2/k + a_{v',i,j} \right)
\\
&= \log \left( 1 - 2/k + \sum_{i,j\in[k]} a_{v,i,j}a_{v',i,j} \right)
\\
&\le  \log \left( \sqrt{ 1 - 2/k + \sum_{i,j\in[k]} a_{v,i,j}^2 } \sqrt{ 1 - 2/k + \sum_{i,j\in[k]} a_{v',i,j}^2 } \right)
\\
&=  \frac12 \log (1 - 2/k + \rho(A_v) ) + \frac12 \log (1 - 2/k + \rho(A_{v'}) ).
\end{align*}
Since each $v\in V$ has degree $d$,
\begin{equation*}\label{eq:sumlogze}
\sum_{e\in E} \log z_e \le \frac{d}2 \log (1 - 2/k + \rho(A_v) ),
\end{equation*}
and combining this with~\eqref{eq:maxfB} yields
\begin{equation}\label{eq:maxfB2}
\max_{\mbf  B\text{ s.t.~\eqref{eq:brelaxedcons}}} f(\mbf A, \mbf B) \le g(\mbf A).
\end{equation}
We will bound $g(\mbf A)$ by applying Theorem~\ref{thm:achlioptas-naor} to matrix $kA_v$ for each $v\in V$, which is a doubly-stochastic matrix.
Note that $d<\ell_k$ implies $d/2 < c_k$, where $c_k$ is defined in~\eqref{eq:cq}.
Therefore, Theorem~\ref{thm:achlioptas-naor} (with $q=k$, $c=d/2$ and $A=kA_v$) yields
\[
\frac{d}{2}\log\left( \frac{(k-1)^2 + \rho(k A_v) - 1}{(k-1)^2} \right) \le \log k - \frac{1}{k}h(k A_v),
\]
with equality if and only if $A_v = \hat A$. Noting that $\rho(kA_v) = k^2\rho(A_v)$ and $\frac{1}{k}h(kA_v) = h(A_v) - \log k$, we obtain
\[
h(A_v) + \frac{d}{2}\log\left( \frac{(k-1)^2 + k^2 \rho(A_v) - 1}{k^2} \right) \le 2\log k + \frac{d}{2}\log\left( \frac{(k-1)^2}{k^2} \right).
\]

Thus, after summing over $v\in V$ and simplifying, we get
\begin{equation*}\label{eq:maxfAB}
g(\mbf A) \le g(\mbf{\hat A}) = \log\left(k^{|V|}\left(\frac{k-1}{k}\right)^{|E|}\right)^{2n} = f(\mbf{\hat A}, \mbf{\hat B}),
\end{equation*}
with equality if and only if $\mbf A=\mbf{\hat A}$. Combining this with~\eqref{eq:maxfB2}, we obtain the desired bound
\begin{align*}\label{eq:gA}
\max_{\mbf  A, \mbf  B\text{ s.t.~\eqref{eq:acons2} and~\eqref{eq:abcons2}}} f(\mbf A, \mbf B)
&\le \max_{\mbf  A\text{ s.t.~\eqref{eq:acons2}}}  \max_{\mbf  B\text{ s.t.~\eqref{eq:brelaxedcons}}} f(\mbf A, \mbf B)\\
	&\le \max_{\mbf  A\text{ s.t.~\eqref{eq:acons2}}} g(\mbf A) = f(\mbf{\hat A}, \mbf{\hat B}).
\end{align*}

Note that for $\mbf A \ne \mbf{\hat A}$, we have $g(\mbf A) < g(\mbf{\hat A})$ and thus $f(\mbf A, \mbf B) < f(\mbf{\hat A}, \mbf{\hat B})$. Furthermore, if $\mbf A = \mbf{\hat A}$ but $\mbf B \ne \mbf{\hat B}$ then, recalling that the maximum in~\eqref{eq:maxfB} is uniquely attained at $\mbf{B^*}=\mbf{\hat B}$, we conclude that $f(\mbf{\hat A}, \mbf B) < g(\mbf{\hat A})  = f(\mbf{\hat A}, \mbf{\hat B})$. This finishes the proof.
\end{proof}

\subsection{Inner Sum} \label{ssec:inner}

In order to accurately approximate $\ex Y^2$, we now return to \eqref{eq:ABsplit} and find asymptotics for the inner sum over the ${\bf B}_e$'s. In particular, we prove the following proposition:

\begin{proposition} \label{prop:inner_sum}
Define $S$ to be the set of $\mbf A$ satisfying \eqref{eq:acons2} and $||\mbf{\hat{A}}-\mbf A||_{\infty} < \frac{\log n}{\sqrt{n}}$ and fix $\mbf A \in S$. Then
\begin{align*}
\mathcal{I}(\mbf A) &:= \sum_{\mbf B} \prod_{\substack{e \in E\\ (i,j,i',j')\in K}} \frac{1}{(b_{e,i,j,i',j'} n)!}\\
	&\sim (\nicefrac{n}{e})^{-n|E|} \gamma(n, k)^{\frac{1}{2}|E|}  \exp
\left( \frac{- nk^2(k-1)^2}{2}\left(\sum_{vv' \in E} \left( \frac{1}{2\lambda} \sum_{i,j}\left (a_{v,i,j} + a_{v',i,j} - \frac{2}{k^2}\right)^2 \quad + \right. \right. \right. \\
	& \qquad \qquad + \quad \left. \left. \left. \frac{1}{2\lambda'} \sum_{i,j} (a_{v,i,j} - a_{v',i,j})^2 + \frac{2}{k^2(k-1)^2}\log \frac{1}{k^2(k-1)^2}\right) \right) \right )
\end{align*}
where
\begin{equation}
\gamma(n, k) := \frac{k^{(3k^2+1)}(k-1)^{4k(k-1)}}{(2\pi n)^{(2k^2-1)}((k-1)^2+1)^{(2k^2-1)}(k-2)^{(k^2-1)}} \label{eq:smallc}
\end{equation}
and
$\lambda$ and $\lambda'$ are as defined in \eqref{eq:lambdas}.
\end{proposition}

\begin{proof}
We prove the proposition using the saddle point method. We start by expressing the
sum as the coefficient of a generating function.

Consider the generating function on variables $\{x_{e,i,j}\} \cup \{x'_{e,i',j'}\}$ for each $e \in E$ and $(i,j,i',j') \in K$:
\[ \prod_{e \in
E} \prod_{(i,j,i',j')\in K} \sum_{t \in \mathbb{N}_0}
\frac{(x_{e,i,j}x'_{e,i',j'})^t}{t!} \]
We want to extract the coefficient when $t = b_{e,i,j,i',j'}n$. Recalling \eqref{eq:abcons2}, we have
\begin{align*}
\sum_{\mbf B} \prod_{\substack{e \in E\\ (i,j,i',j')\in K}} \frac{1}{(b_{e,i,j,i',j'} n)!}
&=
\left[\prod_{e=(v,v')} \left( \prod_{i,j\in[k]} {x_{e,i,j}}^{a_{v,i,j}n}
\prod_{i',j'\in[k]} {x'_{e,i',j'}}^{a_{v',i',j'}n} \right)\right] \quad \times
\\
&
\qquad\qquad \times \quad \prod_{e \in
E} \prod_{(i,j,i',j')\in K} \sum_{t \in \mathbb{N}_0}
\frac{(x_{e,i,j}x'_{e,i',j'})^t}{(\underbrace{b_{e,i,j,i',j'}n}_t)!}
\\
&= \left[\prod_{e=(v,v')} \left( \prod_{i,j\in[k]} {x_{e,i,j}}^{a_{v,i,j}n}
\prod_{i',j'\in[k]} {x'_{e,i',j'}}^{a_{v',i',j'}n} \right)\right] \quad \times
\\&
\qquad \qquad \times \quad \exp\left(  \sum_{e \in E}\sum_{(i,j,i',j')\in K}
x_{e,i,j} x'_{e,i',j'} \right)
\end{align*}
We extract this coefficient using the residue theorem:
\[ \frac{1}{(2\pi i)^{|E|2k^2}} \int \frac{\exp(
\sum\limits_{(i,j,i',j')\in K} z_{e,i,j} z'_{e,i',j'} )}{\prod_e\left(\prod_{i,j}
{z_{e,i,j}}^{a_{v,i,j}n+1} \prod_{i',j'}
{z'_{e,i',j'}}^{a_{v',i',j'}n+1}\right)} d \vec{\bf z} \]
where we've set $x_{e,i,j} = z_{e,i,j}$ and $x'_{e,i',j'} = z'_{e,i',j'}$ to
emphasize that they are complex variables.

In using the saddle point method, we will use a circular path of radius $\rho$.
We use the same radius in every dimension. The method permits us to make this
choice, and we do so because it works.

We can now set $z_{e,i,j} = \rho e^{i\theta_{e,i,j}}$ and $z'_{e,i',j'} = \rho
e^{i\theta'_{e,i',j'}}$. The change of variables gives $d z_{e,i,j} = i \rho
e^{i \theta_{e,i,j}} d \theta_{e,i,j}$ (and similarly for the primed variables) so we cancel all
$|E|2k^2$ copies of $i$ in the denominator as well as one copy of each $z_{e,i,j}$ and $z'_{e,i',j'}$ from the denominator and get

\[ \frac{1}{(2\pi)^{|E|2k^2}} \int \frac{\exp( \rho^2 \sum\limits_e
\sum\limits_{(i,j,i',j')\in K} e^{i(\theta_{e,i,j} + \theta'_{e,i',j'})})}{
\rho^{|E|2n} \exp\left( in \sum\limits_e( \sum\limits_{i,j}
\theta_{e,i,j} a_{v,i,j} + \sum\limits_{i',j'} \theta'_{e,i',j'} a_{v',i',j'} )
\right) } d\vec{\bm \theta} \]
where $\vec{\bm \theta}$ is a vector of all of the $\theta$s and $\theta'$s.

Now we set all of the $\theta$s to zero to find the value on the real line and
choose $\rho$ to optimize (the log of) this value:
\[ \left(\rho^2 |E| k^2(k-1)^2 - |E| 2n \log \rho\right)' = 0 \]
which is accomplished, after some elementary calculus, by
\[ \rho = \sqrt{\frac{n}{k^2(k-1)^2}}. \]

Plugging in $\rho$, we can rewrite the equation above as:
\[ \frac{1}{(2\pi)^{|E|2k^2}}
\left(\frac{k^2(k-1)^2}{n}\right)^{|E|n} \int e^{h(\vec{\bm
\theta})} \ d \vec{\bm \theta} \]
so that we may analyze $h({\bm \theta})$:
\begin{align*}
h({\bm \theta}) &= \frac{n}{k^2(k-1)^2} \sum_e \sum_{(i,j,i',j') \in K}
e^{i(\theta_{e,i,j}+\theta'_{e,i',j'})} \quad -\\
	&\qquad\qquad - \quad in \sum_e\left( \sum_{i,j}
\theta_{e,i,j} a_{v,i,j} + \sum_{i',j'} \theta'_{e,i',j'} a_{v',i',j'}\right)
\end{align*}

Consider
\begin{align*}
|e^{h({\bm \theta})}| &= \exp\left(\Re\left(\frac{n}{k^2(k-1)^2} \sum_e \sum_{(i,j,i',j') \in K}
e^{i(\theta_{e,i,j}+\theta'_{e,i',j'})} \right. \right .\quad -\\
	&\qquad \qquad -  \left. \left. \quad in \sum_e\left( \sum_{i,j}
\theta_{e,i,j} a_{v,i,j} + \sum_{i',j'} \theta'_{e,i',j'} a_{v',i',j'}\right) \right)\right)\\
	&= \exp\left( \frac{n}{k^2(k-1)^2} \sum_e \sum_{(i,j,i',j') \in K}
\Re (e^{i(\theta_{e,i,j}+\theta'_{e,i',j'})}) \right)\\
	&= \exp\left( \frac{n}{k^2(k-1)^2} \sum_e \sum_{(i,j,i',j') \in K}
\cos(\theta_{e,i,j}+\theta'_{e,i',j'}) \right)
\end{align*}
where we use $\Re(z)$ to denote the real part of $z$.

Thus in order to maximize $|e^{h({\bm \theta})}|$, we should make $\theta_{e,i,j} +
\theta'_{e,i',j'}$ close to zero for every pair $\theta_{e,i,j}, \theta'_{e,i',j'}$. If some pair has sum far from zero, we should be able to
show $|e^{h({\bm \theta})}|$ is negligible. We now formalize this notion.

For each edge $e$, define
\[
\theta_e = \frac{1}{2k^2} \left( \sum_{i,j} \theta_{e,i,j} -  \sum_{i',j'} \theta'_{e,i',j'} \right)
\]
to be a weighted average of the $\theta$s. Then define $\delta_{e,i,j}$ and $\delta'_{e,i',j'}$ by
\[
\theta_{e,i,j} = \theta_e + \delta_{e,i,j}
\qquad
\theta'_{e,i',j'} = -\theta_e - \delta'_{e,i',j'}.
\]
Clearly, any choice of the $\theta_{e,i,j}$s and $\theta'_{e,i',j'}$ determines
$\theta_e$s and the $\delta$s. Furthermore, as
\begin{align*}
\theta_e &= \frac{1}{2k^2} \left( \sum_{i,j} \theta_{e,i,j} -  \sum_{i',j'}
\theta'_{e,i',j'} \right)\\
&= \frac{1}{2k^2} \left( \sum_{i,j} (\theta +
\delta_{e,i,j}) -  \sum_{i',j'} ( -\theta - \delta'_{e,i',j'}) \right)
\\
&= \theta_e + \frac{1}{2k^2} \left(\sum_{i,j} \delta_{e,i,j} + \sum_{i',j'}
\delta'_{e,i',j'}\right)
\end{align*}
we see that every choice of the $\theta_e$s and $\delta$s satisfying the linear
constraints
\[ \left(\sum_{i,j} \delta_{e,i,j} + \sum_{i',j'} \delta'_{e,i',j'}\right) = 0 \]
determines a valid choice of $\theta_{e,i,j}$s and $\theta'_{e,i',j'}$s.

Let $R \subseteq \mathbb{R}^{|E|(2k^2+1)}$ be the subspace of choices
of $\theta_e$'s and $\delta$'s that correspond to valid choices of
$\theta_{e,i,j}$'s and $\theta'_{e,i',j'}$'s. Partition $R$ into $R_1$ and $R_2$,
where each point in $R_1$ satisfies
\[ |\delta_{e,i,j}|, |\delta'_{e,i',j'}| \leq \varepsilon = \frac{\log
n}{\sqrt{n}} \]
for all $e,i,j,i'$ and $j'$ (and the $\theta_e$s take any value in $[0,2\pi]$)
and $R_2$ is the complementary region.

We claim that
\[ \int_{R_2} e^{h({\bm \theta})} \ d\theta d{\bm \delta} =
o\left(\left(\frac{e^{n}}{\sqrt{n}}\right)^{|E|} \right).\]

For each point in $R_2$, there is some pair $(\delta_{e,i,j},
\delta'_{e,i',j'})$ such that
\[ |\delta_{e,i,j} - \delta'_{e,i',j'}| > 2\varepsilon. \]
Then
\begin{align*}
|e^{h({\bm \theta})}| &= \exp\left( \frac{n}{k^2(k-1)^2} \sum_e \sum_{(i,j,i',j') \in K}
\cos(\theta_{e,i,j}+\theta'_{e,i',j'}) \right)\\
	&= \exp\left( \frac{n}{k^2(k-1)^2} \sum_e \sum_{(i,j,i',j') \in K} \cos(\theta +
\delta_{e,i,j}-\theta - \delta'_{e,i',j'}) \right)\\
	&< \exp\left( \frac{n}{k^2(k-1)^2} \sum_e \sum_{(i,j,i',j') \in K} \cos(2\varepsilon)
\right)\\
	&= \exp\left( \frac{n}{k^2(k-1)^2} \sum_e \sum_{(i,j,i',j') \in K} 1-\Omega(\varepsilon^2)
\right)\\
	&= \exp\left( n |E|
\left(1-\Omega\left(\frac{\log^2 n}{n}\right)\right) \right)\\	
	&= \exp\left(|E| \left(n-\Omega(\log^2 n)\right) \right)
\end{align*}
and so
\[
\left| \int_{R_2} e^{h({\bm \theta})} \ d{\bm \theta} \right| \leq (2\pi
\rho)^{|E|2k^2} e^{|E|\left(n - \Omega(\log^2 n)\right)} =
o\left( \left(\frac{e^{n}}{\sqrt{n}}\right)^{|E|} \right)
\]
as each $\theta_{e,i,j}, \theta'_{e,i',j'} \in [0,2\pi \rho]$, $\rho =
O(\sqrt{n})$ and $n^{-\Omega(\log^2n)} = o(\frac{1}{n})$.

Next, we show that this value is negligible compared to the integral of the
function over $R_1$.

For points in $R_1$, we use a Taylor expansion for the exponential term:
\begin{align*}
h({\bm \theta}) &= \frac{n}{k^2(k-1)^2} \sum_e \sum_{(i,j,i',j') \in K} \hspace{-10pt} \left( 1 + i(\theta_{e,i,j} -
\theta'_{e,i',j'}) - \tfrac{1}{2}(\theta_{e,i,j}+\theta'_{e,i',j'})^2 +
O(\varepsilon^3)\right) \quad -
\\
& \qquad\qquad - \quad in \sum_e\left( \sum_{i,j} \theta_{e,i,j}
a_{v,i,j} + \sum_{i',j'} \theta'_{e,i',j'} a_{v',i',j'}\right)
\end{align*}
as
\[ |\theta_{e,i,j} + \theta'_{e,i',j'}| = |\theta + \delta_{e,i,j} - \theta -
\delta'_{e,i',j'}| \leq |\delta_{e,i,j}| + |\delta'_{e,i',j'}| \leq
2\varepsilon. \]

Now we sum the relevant terms to get

\begin{align*}
h({\bm \theta}) &= \frac{n}{k^2(k-1)^2}\cdot |E|k^2(k-1)^2
	+ \frac{n}{k^2(k-1)^2} \cdot (k-1)^2 \left(\sum_{i,j} \theta_{e,i,j} +
\sum_{i',j'} \theta'_{e,i',j'}\right) \quad \hspace{-1pt}-\\
	&\qquad \qquad - \quad\frac{n}{2k^2(k-1)^2} \sum_e \sum_{(i,j,i',j') \in K} (\theta_{e,i,j} +
\theta'_{e,i',j'})^2
	+O(n\varepsilon^3) \quad -\\
	&\qquad \qquad - \quad in \sum_e\left( \sum_{i,j} \theta_{e,i,j}
a_{v,i,j} + \sum_{i',j'} \theta'_{e,i',j'} a_{v',i',j'}\right)
\end{align*}

or, cleaning up and combining terms,
\begin{align*}
 h({\bm \theta}) &= \sum_e \left(n + in \left( \sum_{i,j}
\theta_{e,i,j}(\tfrac{1}{k^2}-a_{v,i,j}) + \sum_{i',j'} \theta'_{e,i',j'}
(\tfrac{1}{k^2} - a_{v',i',j'})\right) \quad- \right. \\
& \qquad\qquad-\quad \left.\frac{n}{2k^2(k-1)^2} \sum_{(i,j,i',j') \in K}
(\theta_{e,i,j} + \theta'_{e,i',j'})^2 \right) + O(n\varepsilon^3)
\end{align*}

Note that
\[ n\varepsilon^3 = n\frac{\log^3 n}{n^{3/2}} = \frac{\log^3 n}{\sqrt{n}} = o(1). \]

Recall that $h({\bm \theta})$ is the exponent of the integral in which we're
interested. Thus we define $h_e$ to be the term corresponding to edge $e$ and
note that
\[e^{h({\bm \theta})} = (1+O(1))\prod_e e^{h_e}.\]
For convenience, we now drop the $e$ from the subscripts of $\theta_{e,i,j}, \theta'_{e,i',j'}, a_{e,i,j}$ and $a'_{e,i',j'}$; each variable will be associated with the edge indicated by $h_e$. Furthermore, let $\alpha_{i,j} = a_{v,i,j}-\frac{1}{k^2}$ and $\alpha'_{i',j'} =
\alpha_{v',i',j'}-\frac{1}{k^2}$. Then

\[ h_e = n - in \left( \sum_{i,j}
\theta_{i,j}\alpha_{i,j} + \sum_{i',j'} \theta'_{i',j'}\alpha'_{i',j'}\right) -
\frac{n}{2k^2(k-1)^2} \sum_{(i,j,i',j') \in K} (\theta_{i,j} + \theta'_{i',j'})^2 \]

Let ${\bm \alpha}_{\textrm{all}}$ be a vector of each $\alpha$ and $\alpha'$ so that we may write
\begin{align*}
h_e &= n - in \langle {\bm \theta}, {\bm \alpha}_{\textrm{all}} \rangle -
\frac{n}{2k^2(k-1)^2} \quad \times\\
&\qquad\qquad\times\quad\left( (k-1)^2 \sum_{i,j} \theta_{i,j}^2 + (k-1)^2
\sum_{i',j'}{\theta'}^2_{i',j'} + 2\sum_{(i,j,i',j') \in K} \theta_{i,j} \theta'_{i',j'}\right)
\end{align*}
The last term of $h_e$ is a quadratic form. With $I_m$ and $J_m$ denoting the $m\times m$ identity matrix and the $m\times m$ matrix whose entries are all $1$s, respectively, set
\[ B = (k-1)^2 I_{2k^2} + \begin{bmatrix}0 & 1 \\ 1 &
0\end{bmatrix} \otimes (J_k - I_k)^{\otimes 2} \]
so that
\[ h_e = n - in \langle {\bm \theta}, {\bm \alpha}_{\textrm{all}} \rangle -
\frac{n}{2k^2(k-1)^2} {\bm \theta}^T B {\bm \theta} \]

We now analyze the spectrum of $B$. First
\[ \spec\{J_k-I_k\} = \{k-1, (-1)_{k-1}\}\]
as $J_k$ has eigenvalues $k$ with multiplicity $1$ and $0$ with multiplicity $k-1$ while $I_k$ has eigenvalue $1$ with multiplicity $k$. Then
\[ \spec\{(J_k-I_k)^{\otimes 2}\} = \{(k-1)^2, (1-k)_{2k-2},
1_{(k-1)^2}\} \]
as the Kronecker product of two matrices has an eigenvalue for each pair of eigenvalues of base matrices. Similarly,
\begin{align*}
\spec\{\genfrac{[}{]}{0pt}{1}{ 0 \ 1}{ 1 \ 0} \otimes (J_k-I_k)^{\otimes 2}\} = \{&
(k-1)^2, -(k-1)^2, (1-k)_{2k-2},\\
&  (k-1)_{2k-2}, 1_{(k-1)^2}, (-1)_{(k-1)^2}\}
\end{align*}
as $\genfrac{[}{]}{0pt}{1}{ 0 \ 1}{ 1 \ 0}$ has eigenvalues $-1$ and $1$.

Finally, as $(k-1)^2 I_{2k}$ has eigenvalue $(k-1)^2$ with multiplicity $2k$,
\begin{align*}
\spec\{B\} = \{& 2(k-1)^2, 0, ((k-1)^2+(1-k))_{2k-2},
((k-1)^2+(k-1))_{2k-2},\\
	&((k-1)^2+1)_{(k-1)^2}, ((k-1)^2-1)_{(k-1)^2}\}
\end{align*}

Let $\mbf f_{i,j}$ be vectors in $\mathbb R^{k^2}$ as defined in Lemma~6 from
\cite{KPW10}. They are an orthonormal basis of eigenvectors of $(J_k-I_k)^{\otimes2}$. Let $\mbf{w}_{i,j} = \begin{bmatrix}1/\sqrt2\\1/\sqrt2\end{bmatrix} \otimes
\mbf f_{i,j} = \genfrac{[}{]}{0pt}{}{\mbf f_{i,j} /\sqrt2}{\mbf f_{i,j}/\sqrt2}$
and $\mbf{w'}_{i',j'} = \begin{bmatrix}1/\sqrt2\\-1/\sqrt2\end{bmatrix} \otimes
\mbf f_{i',j'} = \genfrac{[}{]}{0pt}{}{\mbf f_{i',j'} /\sqrt2}{-\mbf f_{i',j'}/\sqrt2 }$. The $\mbf{w}_{i,j}$ and the $\mbf{w'}_{i',j'}$ are an
orthonormal basis of $\mathbb R^{2k^2}$ and moreover are eigenvectors of $B$.
The corresponding eigenvalues are
\[
\lambda_{i,j} = \begin{cases}
(k-1)^2 + 1 & i,j\ne k
\\
(k-1)^2 - (k-1) = (k-1)(k-2) & \text{$( i=k)$ xor $(j=k)$}
\\
(k-1)^2 + (k-1)^2 = 2(k-1)^2 & i=j=k
\end{cases}
\]
and
\[
\lambda'_{i',j'} = \begin{cases}
(k-1)^2 - 1 = k(k-2) & i',j'\ne k
\\
(k-1)^2 + (k-1) = k(k-1) & \text{$( i'=k)$ xor $(j'=k)$}
\\
(k-1)^2 - (k-1)^2 = 0 & i'=j'=k
\end{cases}
\]

Express $\vec{\bm \theta}$ in terms of this new basis:
\[ \vec{\bm \theta} = \sum_{i,j} \tau_{i,j} \mbf{w}_{i,j} + \sum_{i',j'}
\tau'_{i',j'} \mbf{w}'_{i,j} \]
so that we may write
\begin{align*}
h_e &= n - in \left( \sum_{i,j} \tau_{i,j} \langle \mbf{\hat
w}_{i,j}, \vec{\bm \alpha} \rangle + \sum_{i',j'} \tau'_{i',j'} \langle
\mbf{w}'_{i',j'}, \vec{\bm \alpha} \rangle \right) \quad - \\
	&\qquad \qquad - \quad \frac{n}{2k^2(k-1)^2}\left( \sum_{i,j}\lambda_{i,j} \tau_{i,j}^2 + \sum_{i',j'}
\lambda'_{i',j'} {\tau'}^2_{i',j'} \right)
\end{align*}

Recall that our definition of
\[ \theta = \theta_e = \frac{1}{2k^2}\left(\sum_{i,j} \theta_{i,j} - \sum_{i',j'} \theta'_{i',j'}\right) \]
forced
\[ \sum_{i,j} \delta_{i,j} - \sum_{i',j'} \delta'_{i',j'} = 0. \]
Noting that $\mbf{w}'_{k,k} = \frac{1}{\sqrt{2k^2}} \begin{bmatrix}
\mbf{1}_{k^2}\\-\mbf{1}_{k^2}\end{bmatrix}$, this is equivalent to
\[ \langle \vec{\bm \delta}, \mbf{w}'_{k,k} \rangle = 0. \]
Thus
\[
\vec{\bm \theta} = \theta \sqrt{2k^2} \mbf{w}'_{k,k} + \vec{\bm \delta}.
\]
We conclude $\tau'_{k,k} = \theta \sqrt{2k^2}$ and that we can express
\[ \vec{\bm \delta} = \sum_{i,j} \tau_{i,j} \mbf{w}_{i,j} +
\sum_{\mathclap{\substack{i',j' \\ (i',j') \neq (k,k)}}} \tau'_{i',j'}
\mbf{w}'_{i,j} \]

In order to integrate over $R_1$, we must express $R_1$ in terms of the $\tau$.
As $\theta$ can take any value in $[0,2\pi]$, we have $\tau'_{k,k} \in
[0,2\pi\sqrt{2k^2}]$. For each other $\tau$ and $\tau'$, we make the following
argument:

Because we are in $R_1$, each $\delta_{i,j}$ and $\delta'_{i',j'}$ satisfies
$|\delta_{i,j}|, |\delta'_{i',j'}| \leq \varepsilon$. Therefore, $||\vec{\bm
\delta}||_{\infty} \leq \varepsilon$. Note that
\[ ||\vec{\bm \delta}||_2^2 = \sum_{i,j} \delta_{i,j}^2 + \sum_{i',j'}
{\delta'_{i',j'}}^2 \leq 2k^2 ||\vec{\bm \delta}||_\infty^2 \leq 2k^2
\varepsilon^2 \]
and thus $||\vec{\bm \delta}||_2 = O(\varepsilon)$. As the $\mbf{w}$ are
orthonormal, we see $||\vec{\bm \delta}||_2 = ||\vec{\bm \tau}^{-}||_2$, where
$\vec{\bm \tau}^{-}$ is a vector of the $\tau_{i,j}$ and $\tau'_{i',j'}$ without
$\tau'_{k,k}$. Finally, as
\[ ||\vec{\bm \tau}^{-}||_2^2 = \sum_{i,j} \tau_{i,j}^2 +
\sum_{\mathclap{\substack{i',j' \\ (i',j') \neq (k,k)}}} {\tau'_{i',j'}}^2 \geq
\max\biggl( \max_{i,j} \tau_{i,j}^2, \max_{\mathclap{\substack{i',j' \\ (i',j')
\neq (k,k)}}} {\tau'_{i',j'}}^2\biggr) = ||\vec{\bm \tau}^{-}||_\infty^2 \]
we conclude
\[ ||\vec{\bm \tau}^{-}||_\infty \leq ||\vec{\bm \tau}^{-}||_2 = O(\varepsilon) \]
so that each $\tau$ and $\tau'$ (except $\tau'_{k,k}$) is $O(\varepsilon)$.

Note further that as the $\mbf{w}$ are orthogonal, we have $d \vec{\bm
\theta} = d \vec{\bm \tau}$, so by setting
\[ I_{i,j} = \int_{-O(\varepsilon)}^{O(\varepsilon)} e^{-in \tau_{i,j}
\langle \mbf{w}_{i,j}, \vec{\bm \alpha} \rangle - \frac{n}{2k^2(k-1)^2}
\lambda_{i,j} \tau_{i,j}^2} \ d\tau_{i,j}, \]
\[ I'_{i',j'} = \int_{-O(\varepsilon)}^{O(\varepsilon)} e^{-in
\tau'_{i',j'}
\langle \mbf{w}'_{i',j'}, \vec{\bm \alpha} \rangle - \frac{n}{2k^2(k-1)^2}
\lambda'_{i',j'} {\tau'_{i',j'}}^2} \ d\tau'_{i',j'} \]
for $(i',j') \neq (k,k)$ and
\[ I'_{k,k} = \int_0^{2\pi\sqrt{2k^2}} e^{-in \tau'_{k,k}
\langle \mbf{w}'_{k,k}, \vec{\bm \alpha} \rangle - \frac{n}{2k^2(k-1)^2}
\lambda'_{k,k} {\tau'_{k,k}}^2} \ d\tau'_{k,k}\]

We have
\[ \mathcal{I}(\mbf A) \sim \frac{1}{(2\pi)^{|E|2k^2}}
\left(\frac{k^2(k-1)^2}{n}\right)^{|E|n} \prod_e \left(e^n
\prod_{i,j} I_{i,j} \prod_{i',j'} I'_{i',j'} \right). \]

We wish to approximate each integral. First, we claim that for $i = k$ or $j =
k$, we have $\langle \mbf{w}_{i,j}, \vec{\bm \alpha} \rangle = 0$ and
similarly for $i' = k$ or $j' = k$, we have $\langle \mbf{w}'_{i',j'},
\vec{\bm \alpha} \rangle = 0$. This can be seen in the proof of Lemema 7 of
\cite{KPW10}. Furthermore, note that $\lambda'_{k,k} = 0$. Thus
\[ I'_{k,k} =  \int_0^{2\pi\sqrt{2k^2}} e^0 \ d\tau'_{k,k} = 2\pi \sqrt{2k^2}. \]

For the other integrals, we make another substitution $x_{i,j} = \sqrt{n}
\tau_{i,j}$ and $x'_{i',j'} = \sqrt{n} \tau'_{i',j'}$, so noting $\sqrt{n}
\varepsilon = O(\log n)$, the limits of the new integral go from $-O(\log n)$ to
$O(\log n)$, which we approximate by $-\infty$ and $\infty$, to get
\[ I_{i,j} \sim \frac{1}{\sqrt{n}} \int_{-\infty}^\infty
e^{-i\sqrt{n} \langle \mbf{w}_{i,j}, \vec{\bm \alpha} \rangle
x_{i,j} - \frac{\lambda_{i,j}}{2k^2(k-1)^2} x_{i,j}^2} \ d x_{i,j} \]
and
\[ I'_{i',j'} \sim \frac{1}{\sqrt{n}} \int_{-\infty}^\infty
e^{-i\sqrt{n} \langle \mbf{w}'_{i',j'}, \vec{\bm \alpha} \rangle
x'_{i',j'} - \frac{\lambda'_{i',j'}}{2k^2(k-1)^2} {x'_{i',j'}}^2} \ d x'_{i',j'} \]
for $(i',j') \neq (k,k)$.

Using the equation
\[ \int_{-\infty}^\infty e^{ax-bx^2} = \sqrt{\frac{\pi}{b}}
\exp\left(\frac{a^2}{4b}\right) \]
we get
\[ I_{i,j} = \sqrt{\frac{2\pi k^2(k-1)^2}{n\lambda_{i,j}}}
\exp\left(-\frac{nk^2(k-1)^2 \langle \mbf{w}_{i,j}, \vec{\bm \alpha} \rangle^2}
{2\lambda_{i,j}} \right) \]
and
\[ I'_{i',j'} = \sqrt{\frac{2\pi k^2(k-1)^2}{n\lambda'_{i',j'}}}
\exp\left(-\frac{nk^2(k-1)^2 \langle \mbf{w}'_{i',j'}, \vec{\bm \alpha} \rangle^2}
{2\lambda'_{i',j'}} \right) \]
for $(i',j') \neq (k,k)$.

Thus
\begin{align*}
\prod_{i,j} I_{i,j} \prod_{i',j'} I'_{i',j'} &\sim 
2\pi \sqrt{2k^2}\prod_{i,j} \sqrt{\frac{2\pi k^2(k-1)^2}{n\lambda_{i,j}}}
\exp\left(-\frac{nk^2(k-1)^2 \langle \mbf{w}_{i,j}, {\bm \alpha} \rangle^2}
{2\lambda_{i,j}} \right) \quad \times
\\
& \qquad \qquad \times \quad \prod_{\mathclap{\substack{i',j' \\ (i',j') \neq (k,k)}}} \sqrt{\frac{2\pi
k^2(k-1)^2}{n\lambda'_{i',j'}}} \exp\left(-\frac{nk^2(k-1)^2 \langle \mbf{\hat
w}'_{i',j'}, {\bm \alpha} \rangle^2} {2\lambda'_{i',j'}} \right)\\
	&= \frac{2\pi}{\det \mathcal L}\left(\sqrt{\frac{2\pi k^2(k-1)^2}{n}}\right)^{2k^2-1}\hspace{-6pt}\prod_{i,j} \exp\left(-\frac{nk^2(k-1)^2 \langle
\mbf{w}_{i,j}, {\bm \alpha} \rangle^2} {2\lambda_{i,j}} \right) \quad \hspace{-2pt}\times
\\
&  \qquad  \qquad \times \quad \prod_{\mathclap{\substack{i',j' \\ (i',j') \neq (k,k)}}} \exp\left(-\frac{n(k-1)^2 \langle \mbf{w}'_{i',j'}, {\bm
\alpha} \rangle^2} {2\lambda'_{i',j'}} \right)
\end{align*}
where
\[ \det \mathcal{L} = \sqrt{ \frac{1}{2k^2}\prod_{i,j} \lambda_{i,j}
\prod_{\mathclap{\substack{i',j' \\ (i',j') \neq (k,k)}}} \lambda'_{i',j'} }
\]

Recall that for $i,i' = k$ or $j,j' = k$, we have
\[ \langle \mbf{w}_{i,j}, \vec{\bm \alpha} \rangle = 0 = \langle \mbf{\hat
w}'_{i',j'}, \vec{\bm \alpha} \rangle.\]
Thus we need only consider the
eigenvalues where neither $i$ nor $j$ (or neither $i'$ not $j'$) are equal to
$k$; the numerator in the other terms is zero. This
gives us common denominators, allowing us to write
\begin{align*}
\prod_{i,j} I_{i,j} \prod_{i',j'} I_{i',j'} & \sim 2\pi\left(\sqrt{\frac{2\pi
k^2(k-1)^2}{n}}\right)^{2k^2-1} \frac{1}{\det\mathcal L}\quad \times
\\
& \qquad \qquad \times \quad  \exp \left(
-nk^2(k-1)^2\left( \frac{\sum_{i,j}
\langle \mbf{w}_{i,j}, \vec{\bm \alpha} \rangle^2}{2\lambda} + \frac{\sum_{i',j'}
\langle \mbf{w}'_{i',j'}, \vec{\bm \alpha} \rangle^2}{2\lambda'} \right)\right)
\end{align*}
with $\lambda$ and $\lambda'$ as defined in \eqref{eq:lambdas}.

We may now return to our evaluation of $\mathcal{I}(\mbf A)$:
\begin{align*}
\mathcal{I}(\mbf A) \sim& \frac{1}{(2\pi)^{|E|2k^2}}
\left(\frac{k^2(k-1)^2}{n}\right)^{|E|n} \prod_e \left( e^n
\frac{2\pi}{\det\mathcal L}\left(\sqrt{\frac{2\pi k^2(k-1)^2}{n}}\right)^{2k^2-1} \quad \times \right.\\
	&\qquad\qquad\times\quad\left.\exp \left( -nk^2(k-1)^2\left( \frac{\sum_{i,j}
\langle \mbf{w}_{i,j}, {\bm \alpha} \rangle^2}{2\lambda} + \frac{\sum_{i',j'}
\langle \mbf{w}'_{i',j'}, {\bm \alpha} \rangle^2}{2\lambda'} \right)\right)
\right)\\
	=& \left(\frac{(k^2(k-1)^2)^{n +
k^2-\frac{1}{2}}}{(\nicefrac{n}{e})^n(\sqrt{2\pi n})^{2k^2-1}
\det\mathcal L}\right)^{|E|} \quad \times\\
	&\qquad \qquad \times \quad \prod_e \left( \exp \left(
\frac{- nk^2(k-1)^2}{2} \right. \right. \quad \times\\
	&\qquad \qquad \times \quad \left. \left. \left(\frac{\sum_{i,j}\langle \mbf{\hat
w}_{i,j}, {\bm \alpha}_{\textrm{all}} \rangle^2}{\lambda} +
\frac{\sum_{i',j'}\langle \mbf{w}'_{i',j'}, {\bm
\alpha}_{\textrm{all}} \rangle^2}{\lambda'}\right) \right) \right)
\end{align*}
Recall that $\mbf{w}_{i,j} = \genfrac{[}{]}{0pt}{}{\mbf f_{i,j} /\sqrt2}{\mbf
f_{i,j}/\sqrt2 }$, $\mbf{w'}_{i',j'} = \genfrac{[}{]}{0pt}{}{\mbf f_{i',j'}
/\sqrt2}{-\mbf f_{i',j'}/\sqrt2 }$ and ${\bm \alpha}_{\textrm{all}}$ is a
column vector of the $\alpha$s atop the $\alpha'$s. We now break ${\bm \alpha}_{\textrm{all}}$ into two column vectors, ${\bm \alpha}$ containing the $\alpha_{i,j}$ and ${\bm \alpha}'$ containing the $\alpha'_{i',j'}$ so that we may write:
\[ \langle \mbf{w}_{i,j}, {\bm \alpha}_{\textrm{all}} \rangle =
\frac{1}{\sqrt{2}}\left(\langle \mbf f_{i,j}, {\bm \alpha} \rangle + \langle
\mbf f_{i,j}, {\bm \alpha}' \rangle\right) = \frac{1}{\sqrt{2}} \langle \mbf
f_{i,j}, {\bm \alpha} + {\bm \alpha}' \rangle\]
and
\[\langle \mbf{w}'_{i,j}, {\bm \alpha}_{\textrm{all}} \rangle =
\frac{1}{\sqrt{2}}\left(\langle \mbf f_{i,j}, {\bm \alpha} \rangle - \langle
\mbf f_{i,j}, {\bm \alpha}' \rangle \right)= \frac{1}{\sqrt{2}} \langle \mbf
f_{i,j}, {\bm \alpha} - {\bm \alpha}' \rangle \]
Furthermore, as the $\mbf f_{i,j}$ form a basis,
\[ \sum_{i,j} \langle \mbf f_{i,j}, \vec{\bm \alpha} + \vec{\bm \alpha}' \rangle
= ||\vec{\bm \alpha} + \vec{\bm \alpha}'||^2 \quad \text{ and } \sum_{i,j}
\langle \mbf f_{i,j}, \vec{\bm \alpha} - \vec{\bm \alpha}' \rangle = ||\vec{\bm
\alpha} - \vec{\bm \alpha}'||^2 \]
This simplification gives
\begin{align}
\mathcal{I}(\mbf A) &\sim \left(\frac{(k^2(k-1)^2)^{n +
k^2-\frac{1}{2}}}{(\nicefrac{n}{e})^n(\sqrt{2\pi n})^{2k^2-1}
\det\mathcal L}\right)^{|E|} \quad \times \notag \\
	& \qquad \qquad \times \quad \prod_e \left( \exp
\left( \frac{- nk^2(k-1)^2}{2} \left(\frac{||\vec{\bm \alpha}+\vec{\bm
\alpha}'||^2}{2\lambda} + \frac{||\vec{\bm \alpha} - \vec{\bm \alpha}'||^2}{2\lambda'}
\right) \right)
\right) \notag\\
	&= \left(\frac{(k^2(k-1)^2)^{n + k^2-\frac{1}{2}}}{(\nicefrac{n}{e})^n(\sqrt{2\pi n})^{2k^2-1}
\det\mathcal L}\right)^{|E|} \quad \times \notag \\
	& \qquad \qquad \times \quad \exp
\left( \frac{- nk^2(k-1)^2}{2} \sum_e \left(\frac{||\vec{\bm
\alpha}+\vec{\bm \alpha}'||^2}{2\lambda} + \frac{||\vec{\bm \alpha} - \vec{\bm
\alpha}'||^2}{2\lambda'} \right) \right)\notag\\
	&= (\nicefrac{n}{e})^{-n|E|}\left(\frac{(k^2(k-1)^2)^{
k^2-\frac{1}{2}}}{(2\pi n)^{\frac{1}{2}(2k^2-1)}
\det\mathcal L}\right)^{|E|} (k^2(k-1)^2)^{n|E|} \quad \times\notag\\
	&\qquad \qquad \times \quad \exp
\left( \frac{- nk^2(k-1)^2}{2} \sum_e \left(\frac{||\vec{\bm
\alpha}+\vec{\bm \alpha}'||^2}{2\lambda} + \frac{||\vec{\bm \alpha} - \vec{\bm
\alpha}'||^2}{2\lambda'} \right) \right)\notag\\
	&= (\nicefrac{n}{e})^{-n|E|}\left(\frac{(k^2(k-1)^2)^{
k^2-\frac{1}{2}}}{(2\pi n)^{\frac{1}{2}(2k^2-1)}
\det\mathcal L}\right)^{|E|} \exp
\left( n|E|\log(k^2(k-1)^2) \right.\quad -\notag\\
	&\qquad \qquad - \quad \left.\frac{nk^2(k-1)^2}{2}  \sum_e \left(\frac{||\vec{\bm
\alpha}+\vec{\bm \alpha}'||^2}{2\lambda} + \frac{||\vec{\bm \alpha} - \vec{\bm
\alpha}'||^2}{2\lambda'} \right) \right)\notag\\
	&=(\nicefrac{n}{e})^{-n|E|}\left(\frac{(k^2(k-1)^2)^{
k^2-\frac{1}{2}}}{(2\pi n)^{\frac{1}{2}(2k^2-1)}
\det\mathcal L}\right)^{|E|}  \exp
\left( -\frac{nk^2(k-1)^2}{2} \right. \quad \times\notag\\
	&\qquad \qquad \times \quad \left.\sum_e \left(\frac{||\vec{\bm
\alpha}+\vec{\bm \alpha}'||^2}{2\lambda} + \frac{||\vec{\bm \alpha} - \vec{\bm
\alpha}'||^2}{2\lambda'} + \frac{2}{k^2(k-1)^2} \log \frac{1}{k^2(k-1)^2} \right) \right) \label{eq:inner_sum_almost}
\end{align}

Now recall
\begin{align*}
\det \mathcal L &= \sqrt{ \frac{1}{2k^2}\prod_{i,j} \lambda_{i,j}
\prod_{\mathclap{\substack{i',j' \\ (i',j') \neq (k,k)}}} \lambda'_{i',j'} }\\
	&= \left( \frac{1}{2k^2} (\lambda)^{(k-1)^2}((k-1)(k-2))^{2(k-1)}(2(k-1)^2)(\lambda')^{(k-1)^2}(k(k-1))^{2(k-1)}\right)^{1/2}\\
	&= (\lambda)^{\frac{1}{2}(k-1)^2}(k-2)^{(k-1)+\frac{1}{2}(k-1)^2}(k-1)^{2(k-1)-1}k^{\frac{1}{2}(k-1)^2+(k-1)-1}
\end{align*}
so
\begin{align*}
\left(\frac{(k^2(k-1)^2)^{
k^2-\frac{1}{2}}}{(2\pi)^{\frac{1}{2}(2k^2-1)}
\det\mathcal L}\right)^{|E|} &= \left(\frac{k^{\frac{1}{2}(3k^2+1)}(k-1)^{
2k^2-2k}}{(2\pi n)^{\frac{1}{2}(2k^2-1)}
(\lambda)^{\frac{1}{2}(k-1)^2}(k-2)^{\frac{1}{2}(k^2-1)}}\right)^{|E|}\\
	&= \left(\frac{k^{(3k^2+1)}(k-1)^{
4k(k-1)}}{(2\pi n)^{(2k^2-1)}
(\lambda)^{(k-1)^2}(k-2)^{(k^2-1)}}\right)^{\frac{1}{2}|E|}\\
	&= \gamma(n, k)^{\frac{1}{2}|E|}
\end{align*}
In addition, we have $\alpha_{i,j} = a_{v,i,j}-\frac{1}{k^2}$ and ${\alpha'}_{i',j'} = a_{v',i',j'}-\frac{1}{k^2}$. Making these substitutions into \eqref{eq:inner_sum_almost} completes the proof.
\end{proof}

\subsection{Outer Sum/Proof of Proposition~\ref{prop:expectation_y_squared}} \label{ssec:outer}

Having estimated the inner sum over the $\mbf B$s, we now return to \eqref{eq:ABsplit}. Using Propositions~\ref{prop:central_sum} with $\gamma = 1$, the definition of $S$ and the result of Proposition~\ref{prop:inner_sum}, the definitions of $\lambda$ and $\lambda'$ from \eqref{eq:lambdas}, and applying Stirling's approximation gives
\begin{align*}
\ex Y^2 &= \frac{1}{n!^{|E|}} \sum_{\mbf A} \prod_{v\in V}\binom{n}{n A_v}
\prod_{e\in E}(n A_v)!(nA_{v'})! \sum_{\mbf B}\prod_{e\in E}\frac{1}{(nB_e)!}\\
	&\sim \frac{1}{n!^{|E|}} \sum_{\mbf A \in S} \prod_{v\in V}\binom{n}{n A_v}
\prod_{e\in E}(n A_v)!(nA_{v'})! (\nicefrac{n}{e})^{-n|E|}\gamma(n,k)^{\frac{1}{2}|E|} \quad \times \\
	& \qquad \qquad \times \quad \exp \left( \frac{- nk^2(k-1)^2}{2} \left(\sum_{vv' \in E} \left( \frac{1}{2\lambda} \sum_{i,j}\left (a_{v,i,j} + a_{v',i,j} - \frac{2}{k^2}\right)^2 \quad + \right. \right. \right .\\
	& \qquad \qquad + \quad \left. \left. \left. \frac{1}{2\lambda'} \sum_{i,j} (a_{v,i,j} - a_{v',i,j})^2 + \frac{2}{k^2(k-1)^2}\log \frac{1}{k^2(k-1)^2}\right) \right) \right )\\
	&\sim (\xi(n))^{|E|} (\nicefrac{n}{e})^{-2n|E|} \sum_{\mbf{A} \in S} \prod_{v
\in V} \frac{\xi(n) (n/e)^n}{ \prod_{i,j} \xi(a_{v,i,j} n) (a_{v,i,j} n/e)^{a_{v,i,j} n}} \quad \times
\\
	&\qquad\qquad \times\quad
\proda{e\in E\\e=vv'} \left( \prod_{i,j}\xi(a_{v,i,j} n) \left(\tfrac{a_{v,i,j} n}{e}\right)^{a_{v,i,j} n} \right) \left( \prod_{i',j'} \xi(a_{v',i',j'} n) \left(\tfrac{a_{v',i',j'} n}{e}\right)^{a_{v',i',j'} n} \right)\quad \times\\
	& \qquad \qquad \times \quad \gamma(n,k)^{\frac{1}{2}|E|} \exp
\left( \frac{- nk^2(k-1)^2}{2} \left(\sum_{vv' \in E} \left( \frac{1}{2\lambda} \sum_{i,j}\left (a_{v,i,j} + a_{v',i,j} - \frac{2}{k^2}\right)^2 + \right. \right. \right. \\
	& \qquad \qquad + \quad \left. \left. \left. \frac{1}{2\lambda'} \sum_{i,j} (a_{v,i,j} - a_{v',i,j})^2 + \frac{2}{k^2(k-1)^2}\log \frac{1}{k^2(k-1)^2}\right) \right) \right )\\
	&\sim \sum_{\mbf{A} \in S} \prod_{v \in V} \prod_{i,j} a_{v,i,j}^{(d-1)a_{v,i,j}n}\prod_{v
\in V} \frac{\xi(n)}{ \prod_{i,j} \xi(a_{v,i,j} n)} \quad \times \\
	& \qquad \qquad \times \quad 
\proda{e\in E\\e=vv'} \frac{\left( \prod_{i,j}\xi(a_{v,i,j} n)\right)\left( \prod_{i',j'} \xi(a_{v',i',j'} n) \right)}{\xi(n)} \gamma(n,k)^{\frac{1}{2}|E|} \quad \times \\
	& \qquad \qquad \times \quad \exp
\left( \frac{- nk^2(k-1)^2}{2} \left(\sum_{vv' \in E} \left( \frac{1}{2\lambda} \sum_{i,j}\left (a_{v,i,j} + a_{v',i,j} - \frac{2}{k^2}\right)^2 + \right. \right. \right. \\
	& \qquad \qquad + \quad \left. \left. \left. \frac{1}{2\lambda'} \sum_{i,j} (a_{v,i,j} - a_{v',i,j})^2 + \frac{2}{k^2(k-1)^2}\log \frac{1}{k^2(k-1)^2}\right) \right) \right )\\
	&= \sum_{\mbf A \in S} P(\mbf A, n) e^{nF(\mbf{A})}
\end{align*}
where
\begin{align}
P(\mbf A,n) &= \gamma(n,k)^{\frac{1}{2}|E|} \prod_{v
\in V} \frac{\xi(n)}{ \prod_{i,j} \xi(a_{v,i,j} n)} \proda{vv'\in E} \frac{\left( \prod_{i,j}\xi(a_{v,i,j} n)\right)\left( \prod_{i',j'} \xi(a_{v',i',j'} n) \right)}{\xi(n)} \label{eq:pAequi}\\
F(\mbf A) &= (d-1)\sum_{v \in V} \sum_{i,j} a_{v,i,j} \log a_{v,i,j} \quad - \notag\\
	& \qquad \qquad - \quad  \frac{k^2(k-1)^2}{2}\left(\sum_{vv' \in E} \left( \frac{1}{2\lambda} \sum_{i,j}\left (a_{v,i,j} + a_{v',i,j} - \frac{2}{k^2}\right)^2 + \right. \right. \notag\\
	& \qquad \qquad + \quad \left. \left. \frac{1}{2\lambda'} \sum_{i,j} (a_{v,i,j} - a_{v',i,j})^2 + \frac{2}{k^2(k-1)^2}\log \frac{1}{k^2(k-1)^2}\right) \right)
\end{align}

First we bound $F(\mbf{A})$. Note that in our proof of Proposition~\ref{prop:exp_bound} we relaxed the constraints on $\mbf B$ in order to bound $f(\mbf A, \mbf B)$. However, in Section~\ref{ssec:inner} we have shown that 
\[ F(\mbf A) = \max\limits_{\mbf B \text{ s.t.~\eqref{eq:abcons2}}} f(\mbf A, \mbf B) \le  \max\limits_{\mbf B \text{ s.t.~\eqref{eq:brelaxedcons}}} f(\mbf A, \mbf B) \]
as the constraints of \eqref{eq:abcons2} are more strict than those of \eqref{eq:brelaxedcons}. Thus combining Proposition~\ref{prop:exp_bound} and Proposition \ref{prop:inner_sum} gives the following corollary:

\begin{corollary} \label{cor:exp_bound}
Suppose $d < \ell_k$. The maximum of $F(\mbf A)$ subject to \eqref{eq:abcons1}
is uniquely attained at the point $\mbf A = \mbf{\hat{A}}$ and equals
\[ F(\mbf{\hat{A}}) = f(\mbf{\hat{A}}, \mbf{\hat{B}}) = \log\left(k^{|V|}\left(\frac{k-1}{k}\right)^{|E|}\right)^{2n}. \]
\end{corollary}

We have the trivial bound $P(\mbf A, n) \le n^{\frac{1}{2}|V|+k^2|E|}$ and therefore Proposition~\ref{prop:central_sum} tells us that asymptotically the sum over all $\mbf A$ satisfying \eqref{eq:acons2} is the same as the sum over those just in $S$. Therefore we now once again use Proposition~\ref{prop:laplacian_summation_gamma} to bound the sum
\[ \sum_{\mbf A\text{ s.t.~\eqref{eq:acons2}}} P(\mbf A, n) e^{nF(\mbf{A})} .\]
As in the proof
of Proposition~\ref{prop:EY}, we define a bipartite graph $\Gamma=\Gamma(V_\Gamma,
E_\Gamma)$ that allows us to express the equality constraints
in~\eqref{eq:acons2} in terms of its unsigned incidence matrix $D$.
The idea is to associate each equation in~\eqref{eq:acons2} to a vertex of
$\Gamma$ and every variable to an edge in a way that preserves the incidence
relations. To do this, we assign label $w_{v,1,i}$ to equation $\sum_{j}
a_{v,i,j} = 1/k$ and label $w_{v,2,j}$ to equation $\sum_{i} a_{v,i,j} = 1/k$.
The vertex set of $\Gamma$ is $V_\Gamma = V_{\Gamma,1} \cup  V_{\Gamma,2}$, where
\[
 V_{\Gamma,1} = \{ w_{v,1,i} : v\in V, i\in[k] \}
\and
 V_{\Gamma,2} = \{ w_{v,2,j} : v\in V, j\in[k] \}
\]
are the two sides of a bipartition.
The edge set is
\[
E_\Gamma = \{ a_{v,i,j}: v\in V, i,j\in [k]\},
\]
where each edge $a_{v,i,j}$ has endpoints $w_{v,1,i}$ and $w_{v,2,j}$ (i.e.~the
labels of the two equations in which variable $a_{v,i,j}$ appears). Note that
unlike in the proof of Proposition~\ref{prop:EY}, here we do not require $i \neq j$.
Then the equality constraints in~\eqref{eq:acons2} are equivalent to
\begin{equation}
D \mbf b = \mbf y,
\label{eq:Dby2}
\end{equation}
where $D$ is the unsigned incidence matrix of $\Gamma$ and $\mbf y$ is the vector in $\real^{|V_\Gamma|}$ whose entries are all $1/k$.
The equations in~\eqref{eq:Dby2} are consistent, since they admit the solution
\begin{align}
a_{v,i,j} &= 0 \quad \forall v\in V, i,j\in[k], j \neq i
\notag\\
a_{v,i,i} &= \frac{1}{k} \quad \forall v \in V, i\in[k],
\label{eq:sol2}
\end{align}
We observe a few easy facts about $\Gamma$. First,
\[
|V_\Gamma| = 2k|V|
\and
|E_\Gamma| = k^2|V|
\]
Also, $\Gamma$ has exactly $|V|$ connected components. More precisely, for each
$v\in V$, the set of all vertices of the form $w_{v,1,i}$ or $w_{v,2,j}$ induces
a connected component of $\Gamma$. Each of these components is isomorphic to the
complete bipartite graph $K_{k,k}$. In particular, $\Gamma$ has at least one
cycle (since $k\ge 3$). Since $\Gamma$ is bipartite, it is well known
(see~e.g.~Theorem~8.2.1 in~\cite{GR01}) that $D$ has rank $|V_\Gamma| -
|V|$, and therefore $\mathbb V=\Ker(D)$ has dimension
\[
r = |E_\Gamma| - |V_\Gamma| + |V| = (k^2-2k+1) |V| = (k-1)^2 |V|. 
\]

Now we calculate $\tau(\Gamma)$. Since each maximal forest in $\Gamma$ is
bijectively determined by selecting a spanning tree in each component, we
conclude that the number of maximal forests in $\Gamma$ is
\begin{equation} \label{eq:tau2}
\tau(\Gamma) = \tau(K_{k,k})^{|V|} = \left( k^{2k-2} \right)^{|V|}.
\end{equation}

Let
\[
K = \{ \mbf A \in\real^{|E_\Gamma|} : 0 \le a_{v,i,j} \le 1/k\} 
\and
K_1 = \left\{ \mbf A \in\real^{|E_\Gamma|} : \tfrac{0.9}{k^2} \le a_{v,i,j} \le
\tfrac{1.1}{k^2}\right\}.
\]
Clearly, $K$ is a compact convex set with non-empty interior $K^\circ$, and any
choice of $\mbf A \in S_1$ lies inside $K$. Let $\phi(\mbf A) = F(\mbf A)$,
which is continuous on $K$, and
\[
\psi(\mbf A) = 
\prod_{v\in V} \left(\prod_{i,j} \sqrt{a_{v,i,j}}\right)^{d-1},
\]
which is continuous and positive on $K_1$.
By Corollary~\ref{cor:exp_bound}, the maximum of $\phi(\mbf A)$ in $K$ subject
to~\eqref{eq:Dby2} is uniquely attained at $\mbf{\hat A} \in K_1 \subset
K^\circ$. Moreover, $\phi(\mbf A)$ is twice continuously differentiable in the
interior $K^\circ$. Set $\phi =
(d-1)G_1 - \frac{k^2(k-1)^2}{2}G_2$. Then
\[ \frac{\partial^2 G_1}{\partial a_{\nu,\ell,m} \partial a_{\nu',\ell',m'}} =
\begin{cases} \frac{1}{a_{\nu,\ell,m}} & \nu=\nu', \ell=\ell', m=m'\\ 0 &
\text{otherwise} \end{cases} \]
The second derivatives of $G_2$ are more difficult. First,
\[ \frac{\partial G_2}{\partial a_{\nu,\ell,m}} = \sum_{v' \text{ s.t. } \nu v' \in E} \left(
\frac{1}{\lambda} (\alpha_{\nu,\ell,m} + \alpha_{v',\ell,m}) +
\frac{1}{\lambda'} (\alpha_{\nu,\ell,m}-\alpha_{v',\ell,m})\right) \]

Then
\[ \frac{\partial^2 G_2}{\partial a_{\nu,\ell,m} \partial a_{\nu',\ell',m'}} =
\begin{cases} d(\frac{1}{\lambda}+\frac{1}{\lambda'}) & \nu = \nu', \ell=\ell',
m=m' \\ \frac{1}{\lambda} - \frac{1}{\lambda'} & \nu\nu' \in E, \ell=\ell', m=m'
\\ 0 & \ell \neq \ell' \text{ or } m \neq m'\end{cases} \]

Noting that
\[ \frac{1}{\lambda} + \frac{1}{\lambda'} = \frac{1}{(k-1)^2+1} +
\frac{1}{(k-1)^2-1} = \frac{2(k-1)^2}{((k-1)^2+1)((k-1)^2-1)} =
\frac{2(k-1)^2}{\lambda\lambda'}\]
and
\[ \frac{1}{\lambda} - \frac{1}{\lambda'} = \frac{1}{(k-1)^2+1} -
\frac{1}{(k-1)^2-1} = \frac{-2}{((k-1)^2+1)((k-1)^2-1)} =
\frac{-2}{\lambda\lambda'}\]
we get
\[ \frac{\partial^2 \phi}{\partial a_{\nu,\ell,m} \partial a_{\nu',\ell',m'}}
= \begin{cases}
\frac{d-1}{a_{\nu,\ell,m}} - d\frac{k^2(k-1)^4}{\lambda\lambda'} &
	\nu = \nu', \ell=\ell', m=m' \\
\frac{k^2(k-1)^2}{\lambda\lambda'} & \nu\nu' \in E, \ell=\ell', m=m' \\
0 & \ell \neq \ell' \text{ or } m \neq m'\end{cases}
\]

Hence, recalling that $A$ is the adjacency matrix of $G$, the Hessian matrix of $\phi$ at $\mbf A = \mbf{\hat A}$ is
\[ H = \left((d-1)k^2-d\frac{k^2(k-1)^4}{\lambda \lambda'}\right) I_{|V|}
\otimes I_{k^2} + \left(\frac{k^2(k-1)^2}{\lambda\lambda'}\right)A\otimes
I_{k^2}. \]

Consider the structure of $D$, the unsigned incidence matrix of $\Gamma$. As each component of $\Gamma$ is isomorphic to $K_{k,k}$, we may write $D = I_{|V|} \otimes \hat{D}$, where $\hat{D}$ is a $2k \times k^2$ incidence matrix for $K_{k,k}$. Note that as $\hat{D}$ is bipartite with one connected component, $\rank(\hat{D}) = 2k-1$ and so $\dim(\ker(D)) = k^2-(2k-1) = (k-1)^2$. Let $U$ be any $k^2 \times (k-1)^2$ matrix whose columns form a basis of $\ker(\hat{D})$. Then we claim the columns of $I_{|V|} \otimes U$ form a basis of $\mathbb{V}$. Every column of $I_{|V|} \otimes U$ is in the kernel of $D$ as the corresponding column of $U$ is in the kernel of $\hat{D}$. The columns of $U$ are linearly independent, as they form a basis, so the columns of $I_{|V|} \otimes U$ are as well. Finally, there are $|V|r = \dim(\mathbb{V})$ columns in $I_{|V|} \otimes U$. Then, making use of several facts about the Kronecker product (see, for example, \cite{Zha11} Theorem 4.5 and Problem 4.3.3a) we have

\begin{align*}
\det(-H|_{\mathbb V}) &= \frac{\det((I_{|V|}\otimes U)^T(-H)(I_{|V|}\otimes U))}{\det((I_{|V|}\otimes U)^T(I_{|V|}\otimes U))}\\
	&= \frac{\det\left(\left(-(d-1)k^2+d\frac{k^2(k-1)^4}{\lambda \lambda'}\right) I_{|V|}
\otimes U^TU - \left(\frac{k^2(k-1)^2}{\lambda\lambda'}\right)A\otimes
U^TU\right)}{\det(I_{|V|}\otimes U^TU))}\\
	&=\det\left(\left(-(d-1)k^2+d\frac{k^2(k-1)^4}{\lambda
\lambda'}\right) I_{|V|} -
\frac{k^2(k-1)^2}{\lambda\lambda'}A\right)^{(k-1)^2} \quad \times\\
	&\qquad\qquad \times \quad \frac{\det(U^TU)^{|V|}}{1^{(k-1)^2}\det(U^TU)^{|V|}}\\
	&= \left(\prod_{i=1}^{|V|} \left(-(d-1)k^2+d\frac{k^2(k-1)^4}{\lambda
\lambda'}\right) - \frac{k^2(k-1)^2}{\lambda\lambda'} \alpha_i \right)^{(k-1)^2}\\
	&= \left(\left(\frac{k^2}{\lambda\lambda'}\right)^{|V|} \prod_{i=1}^{|V|} (-\lambda\lambda'(d-1) + d(k-1)^4 - \alpha_i(k-1)^2\right)^{(k-1)^2}\\
	&= \left(\left(\frac{k^2}{\lambda\lambda'}\right)^{|V|} \prod_{i=1}^{|V|} (\lambda\lambda' + d - \alpha_i(k-1)^2\right)^{(k-1)^2}\\	
	&= (h(d,k))^{(k-1)^2}
\end{align*}
where $h(d,k)$ is as defined in \eqref{eq:det_hessian}.

Let
\[
\mathbb X_n = \left\{ \mbf A\in K \cap \frac{1}{n} \ent^{|E_\Gamma|} : D\mbf A = \mbf y \right\}.
\]
The solution described in~\eqref{eq:sol2} belongs to $K$ and, since $k\mid n$,
also to $\frac{1}{n} \ent^{|E_\Gamma|}$, so $\mathbb X_n$ is not empty. For each
$\mbf A\in \mathbb X_n$, let
\[
T_n(\mbf A) = P(\mbf A, n)e^{nF(\mbf A)}
\]
and
\[
c_n = \gamma(n,k)^{\frac{1}{2}|E|}(2\pi n)^{-\frac{1}{2}(k^2-1)|V| +
\frac{1}{2}(2k^2-1)|E|}.
\]
First, from~\eqref{eq:pAequi} and since $1\le \xi(x) = O(\sqrt x)$ as
$x\to\infty$, we can bound the polynomial factor $P(\mbf A, n)$ in each term by
\begin{equation}
P(\mbf A, n) = O\left( n^{|V|+2k^2|E|} \right),
\label{eq:boundp2}
\end{equation}
where the hidden constant in the $O()$ notation does not depend on $\mbf A$.

In view of~\eqref{eq:boundp2}, for $\mbf A\in \mathbb X_n$,
\[
 P(\mbf A, n) / c_n = O(n^{|V|+2k^2|E|}/c_n) = e^{o(n)}
\]
and combined with Corollary~\ref{cor:exp_bound} this gives
\[ T_n(\mbf A) = c_n(P(\mbf A, n)/c_n)e^{nF(\mbf A)} = O(c_ne^{n\phi(\mbf{\hat
A})+o(n)}). \]
Now note that we have chosen $c_n$ such that because $\xi(x)\sim\sqrt{2\pi x}$
as $x\to\infty$, and after a few computations, we have that for $\mbf A\in
\mathbb X_n \cap K_1$, \[
P(\mbf A, n) = c_n (\psi(\mbf A)+o(1)).
\]

Finally, as we've met all of the conditions of Proposition~\ref{prop:laplacian_summation_gamma},

\begin{align*}
\ex Y^2 &\sim \sum_{\mbf A \in \mathbb X_n} T_n(\mbf A) \\
	&\sim \frac{\psi(\mbf{\hat A})}{\tau(\Gamma)^{1/2} \det(-H|_{\mathbb
V})^{1/2}} (2\pi n)^{r/2} c_n e^{n\phi(\mbf{\hat A})}\\
	&= \frac{k^{-k^2(d-1)|V|}}{(k^{2k-2})^{\frac{1}{2}|V|}
h(d,k)^{\frac{(k-1)^2}{2}}} (2\pi n)^{\frac{1}{2}(k-1)^2|V|}
\gamma(n,k)^{\frac{1}{2} |E|} (2\pi n)^{-\frac{1}{2}(k^2-1)|V| +
\frac{1}{2}(2k^2-1)|E|} \quad \times \\
	& \qquad \qquad \times \quad \left(k^{|V|}\left(\frac{k-1}{k}\right)^{|E|}\right)^{2n}\\
	&= \frac{\gamma(n,k)^{\frac{1}{2}|E|}
k^{-2k^2|E|+(k^2-k+1)|V|}}{h(d,k)^{\frac{(k-1)^2}{2}}} (2\pi n)^{-(k-1)|V| +
\frac{1}{2}(2k^2-1)|E|} \left(k^{|V|}\left(\frac{k-1}{k}\right)^{|E|}\right)^{2n}\\
	&= \frac{k^{(k^2-k+1)|V|-\frac{1}{2}(k^2-1)|E|} (k-1)^{(2k^2-2k)|E|}
}{\lambda^{\frac{1}{2}(k-1)^2|E|} (k-2)^{\frac{1}{2}(k^2-1)|E|}
h(d,k)^{\frac{(k-1)^2}{2}}} (2\pi n)^{-(k-1)|V|} 
\left(k^{|V|}\left(\frac{k-1}{k}\right)^{|E|}\right)^{2n}
\end{align*}
This completes the proof of Proposition~\ref{prop:expectation_y_squared}.

\section{Joint Moments}
\label{sec:joint}
Recall that $Y$ counts the number of strongly equitable $k$-colourings of a random $n$-lift of $G$. In this section we allow $G$ to be any $d$-regular graph. Recall that for fixed $j \ge 3$, we denote the number of $j$-cycles in a random lift by $Z_j$. In this section, we estimate the expected value of the joint moment $YZ_j$ by proving the following proposition:

\begin{proposition}\label{prop:single_joint_moment}
If $Y$ is the number of strongly equitable $k$-colourings of a random $n$-lift $L$ of $G$ and, for $j \ge 3$, $Z_j$ counts the number of $j$-cycles in $L$, then
\[
\frac{\ex (YZ_j)}{\ex Y} \sim \lambda_j(1+\delta_j)
\]
where $\lambda_j$ and $\delta_j$ are as defined in \eqref{eq:sscm_constants}.
\end{proposition}

In order to prove Proposition~\ref{prop:single_joint_moment} we require two lemmas. The first is a lemma of Friedman, originally proved in \cite{Fri08}, as presented in \cite{GJR10}:

\begin{lemma}\label{lem:place_j_cycle}
Suppose that $G$ is $d$-regular with $d \ge 3$ and let $\alpha_1, \ldots, \alpha_{|V|}$ be the eigenvalues of the adjacency matrix of $G$. For $i=1,\ldots, |V|$, let ${\beta_i}^+$ and ${\beta_i}^-$ denote the roots of the quadratic $x^2-\alpha_ix+d-1 = 0$. That is,
\[ {\beta_i}^+ = \tfrac{1}{2}\alpha_i + \sqrt{\tfrac{1}{4}\alpha_i^2-(d-1)}, \quad \beta_i^- = \tfrac{1}{2}\alpha_i - \sqrt{\tfrac{1}{4}\alpha_i^2-(d-1)}.\]
Then the number of non-backtracking closed $j$-walks in $G$ is given by
\begin{align*}
c_j &:= \frac{1}{2}|V|(d-2)(1+(-1)^j)+\sum_{i=1}^{|V|}((\beta_i^+)^j+(\beta_i^-)^j)\\
	&= (|E|-|V|)(1+(-1)^j)+\sum_{i=1}^{|V|}((\beta_i^+)^j+(\beta_i^-)^j)
\end{align*}
where we note that in a $d$-regular graph we have
\[ 2|E| = \sum_{v \in V} d(v) = d|V| \]
and thus
\[ \frac{1}{2}|V|(d-2) = \frac{1}{2}(d|V|-2|V|) = \frac{1}{2}(2|E|-2|V|) = |E|-|V|. \]
\end{lemma}

The second lemma considers properly colouring cycles.

\begin{lemma}\label{lem:colour_j_cycle}
The number of ways to properly $k$-colour a rooted, directed cycle of length $j$ is
\[ (k-1)^j + (k-1)(-1)^j. \]
\end{lemma}

\begin{proof}
Assign a unique color to each vertex in $K_k$. Each proper $k$-coloring of the
$j$-cycle corresponds to a closed walk of length $j$ in $K_k$: starting with the
color of the root, travel in the direction of the cycle to the vertex
corresponding to the next color. The total number of closed $j$-walks in $K_k$
is $\tr(A^j)$, where $A = J_k - I_k$ is the adjacency matrix of $K_k$. As $A$
has eigenvalues $(k-1)$ with multiplicity one and $-1$ with multiplicity $k-1$,
the number of directed, rooted $j$-walks is $(k-1)^j + (k-1)(-1)^j$.
\end{proof}

We are now ready to prove Proposition~\ref{prop:single_joint_moment}.

\begin{proof}[Proof of Propositon~\ref{prop:single_joint_moment}]
We estimate $\ex(YZ_j)$ by counting ordered triples (lift, cycle, colouring), where the lift contains the $j$-cycle and is properly coloured by the colouring, and then dividing by the total number of lifts.

First, it will be convenient to consider rooted, ordered cycles. Each cycle of length $j$ contains $2j$ rooted, oriented cycles, so we simply correct our answer by a factor of $2j$.

For any rooted, oriented cycle $C = (c_1, \ldots, c_j)$, we define the \emph{cycle type} of $C$ to be the sequence $(v_1, \ldots, v_j)$ such that $c_i \in \Pi^{-1}(v_i)$; that is, the cycle type is the sequence of fibers containing the vertices of the cycle. Note that every cycle type corresponds to a closed, non-backtracking walk in $G$. Thus our first step in counting ordered triples is to choose a closed, non-backtracking walk $w$ to determine the cycle type of our cycle.

Next we fix a colouring, $Q$, of the cycle, after which we choose the vertices in the lift to realize the cycle of the appropriate cycle type and colour them according to $Q$. If the cycle type revisits a fiber, there may be fewer than $n$ options for each vertex. However, as the cycle has fixed length $j$ and $n$ goes to infinity, there are $(1-o(1))n$ choices for each vertex, for a total contribution of $(1-o(1))n^j$.

Having specified $w, Q$ and the vertices of the cycle, we now follow a process very similar to that of Section~\ref{ssec:opti1} resulting in \eqref{eq:EX_pre_stirling}. For each $e = vv' \in E$ we define $b^\ast_e = (b^\ast_{e,i,i'})_{i,i' \in [k], i\ne i'}$ where $b^\ast_{e,i,i'}$ denotes the proportion of edges in $\Pi^{-1}(e)$ that connect a vertex of colour $i$ in $\Pi^{-1}(v)$ to a vertex of colour $i'$ in $\Pi^{-1}(v')$, excluding edges already prescribed by $w$ and $Q$. We set $\mbf{b^\ast} = (b^\ast_e)_{e \in E}$. The entries of each $b^\ast_e$ must be in $\frac{1}{n}\ent$ and satisfy
\begin{align}
b^\ast_{e,i,i'} &\ge 0 \quad \forall e\in E, i,i'\in[k], i\ne i'
\notag
\\
\sum_{i'\ne i} b^\ast_{e,i,i'} &= a(v,i,w,Q) \quad \forall e=vv'\in E, i\in [k]
\notag
\\
\sum_{i\ne i'} b^\ast_{e,i,i'} &= a(v',i,w,Q) \quad \forall e=vv'\in E, i'\in [k].
\label{eq:jointcons}
\end{align}
where $a(v,i,w,Q)$ is the proportion of vertices in $\Pi^{-1}(v)$ that still need to receive colour $i$, after accounting for the vertices already coloured by $Q$ according to $w$, so that each colour is assigned to $\frac{1}{k}$ of the vertices in $\Pi^{-1}(v)$ (as $Y$ is a strongly equitable colouring). Specifically, if $\ell_i$ vertices in $\Pi^{-1}(v)$ have already been assigned colour $i$ by $w$ and $Q$, then
\[ a(v,i,w,Q) = \frac{1}{k} - \frac{\ell_i}{n}. \]

For each vertex $v$ in $G$, we must colour the uncoloured vertices in the fiber $\Pi^{-1}(v)$. Define $\epsilon(v, w)$ to be the number number of times $v$ is encountered in $w$, and let
\[
a^\ast_v = (a^\ast_{v,i})_{i \in [k]} = (a(v,i,w,Q))_{i \in [k]}.
\]
Then the number of ways to colour the vertices is
\[ \prod_{v \in V} \binom{n-\epsilon(v,w)}{a^\ast_vn}. \]
We then decide, for every $e = vv' \in E$ and distinct colours $i,i' \in [k]$, which sets of $b_{e,i,i'}n$ in $\Pi^{-1}(v)$ and $\Pi^{-1}(v')$ will be matched. For each $e = vv'$, define $a^{\ast\ast}_{e,v} = (a^{\ast\ast}_{e,v,i})_{i \in [k]}$ to be the sequence of vertices in $\Pi^{-1}(v)$ that have been assigned colour $i$ and have not already been matched to a vertex in $\Pi^{-1}(v')$ by $w$. Then the number of ways to choose sets of vertices to be matched is
\[
\proda{e\in E\\e=vv'} \bigg( \prod_{i\in[k]} \frac{(a^{\ast\ast}_{e,v,i} n)!}{ \prod_{i'\ne i} (b^\ast_{e,i,i'} n)!} \bigg) \bigg( \prod_{i'\in[k]} \frac{(a^{\ast\ast}_{e,v',i'} n)!}{ \prod_{i\ne i'} (b^\ast_{e,i,i'} n)!} \bigg)
= \proda{e\in E\\e=vv'} \frac{(a^{\ast\ast}_{e,v} n)!}{(b^\ast_{e} n)!} \frac{(a^{\ast\ast}_{e,v'} n)!}{(b^\ast_{e} n)!}
\]
Finally, we need to choose a perfect matching between these sets, which can be done in
\[
\proda{e\in E}  \proda{i,i'\in[k]\\ i\ne i'}  (b^\ast_{e,i,i'} n)! = \proda{e\in E}  (b^\ast_e n)! 
\]
different ways. Putting everything together, we get
\begin{align}
\ex(YZ_j) &= \frac{1}{2j}\sum_w \sum_Q (1+o(1))n^j \frac{1}{n!^{|E|}} \prod_{v \in V} \binom{n-\epsilon(v,w)}{a^\ast_vn}  \quad \times \notag\\
	& \qquad \qquad \times \quad \sum_{\mbf{b^\ast} \text{ s.t. } \eqref{eq:jointcons}} \proda{e \in E \\ e=vv'} \frac{(a^{\ast\ast}_{e,v} n)!(a^{\ast\ast}_{e,v'} n)!}{(b^\ast_{e} n)!} \label{eq:unsimplified_joint}
\end{align}
We now seek estimates for the equation above. First, note that for any walk $w$,
\[
\sum_{v \in V} \epsilon(v,w) = j
\]
as the walk is of length $j$. Similarly, across all $v \in V$, there are exactly $j$ vertices already coloured with some colour. As $Y$ is a strongly equitable colouring, we have
\begin{equation}\label{eq:critical_term_1}
\prod_{v \in V} \binom{n-\epsilon(v,w)}{a^\ast_vn} \sim \frac{\left(\frac{n}{k}\right)^j}{n^j} \cdot \frac{(n!)^{|V|}}{(\left(\frac{n}{k}\right)!)^{|V|}}.
\end{equation}
Furthermore, for any choice of $\mbf{b^\ast}$, the product over all edges of the $a^{\ast\ast}_{e,v}$ encounters all $j$ edges of $w$. For $t = 1, \ldots, j$, set $b_t = b_{e,i,i'}$ where $e$ is the edge that connects vertices $t$ and $t+1$ of $w$ and $Q$ colours those vertices with colours $i$ and $i'$, respectively, where the $b_{e,i,i'}$ are as defined in \eqref{eq:bcons1equi} (as opposed to the $b^\ast_{e,i,i'}$ defined in \eqref{eq:jointcons}). Then
\begin{equation} \label{eq:critical_term_2}
\proda{e \in E \\ e=vv'} \frac{(a^{\ast\ast}_{e,v} n)!(a^{\ast\ast}_{e,v'} n)!}{(b^\ast_{e} n)!} \sim \prod_{t=1}^j \frac{b_tn}{(\frac{n}{k})^2} \proda{e \in E \\ e=vv'} \frac{((\frac{n}{k})!)^2}{(b_{e} n)!}.
\end{equation}
Putting these facts together allows us to rewrite \eqref{eq:unsimplified_joint}
in terms of $p(\mbf{\hat{a}}, \mbf b, n)$ and $f(\mbf a, \mbf b)$ as defined in
\eqref{eq:pabequi} and \eqref{eq:fabequi}, respectively:
\begin{equation}\label{eq:simplified_joint_moment}
\ex(YZ_j) \sim \frac{1}{2j}\sum_w \sum_Q k^j \sum_{\mbf{b} \text{ s.t. } \eqref{eq:bcons1equi}} \left(\prod_{t=1}^j b_t\right) p(\mbf{\hat{a}}, \mbf b, n) e^{nf(\mbf a, \mbf b)}.
\end{equation}
Note that we may bound
\[ \prod_{t=1}^j b_tn \le 1 \]
and that, furthermore, when $\mbf{b} = \mbf{\hat{b}}$, none of the $b_t$ vanish.
Therefore, Proposition~\ref{prop:lower_bound_optimization_G_equitable} holds. Though we omit the details, we may thus use Proposition~\ref{prop:laplacian_summation_gamma} almost identically as in Section~\ref{ssec:equi_expectation} to get
\[ \ex(YZ_j) \sim \frac{1}{2j}\sum_w \sum_Q \left(\frac{n}{k}\right)^{-j} \left(\frac{n}{k(k-1)}\right)^j \ex(Y). \]
Finally, as the terms no longer depend on $w$ or $Q$, we apply Lemmas~\ref{lem:place_j_cycle} and \ref{lem:colour_j_cycle} to get
\begin{align*}
\ex(YZ_j) &= \frac{1}{2j} \cdot c_j \cdot ((k-1)^j+(k-1)(-1)^j) \cdot \frac{1}{(k-1)^j} \ex(Y)\\
	&= \frac{c_j}{2j} \left(1+\frac{(-1)^j}{(k-1)^{j-1}}\right) \ex(Y)\\
	&= \lambda_j(1+\delta_j) \ex(Y)
\end{align*}
Dividing by $\ex(Y)$ completes the proof.
\end{proof}

In order to apply the small subgraph conditioning method, we need estimations for higher moments. The argument for Proposition~\ref{prop:all_joint_moments} is just an extension of the proof of Proposition~\ref{prop:single_joint_moment}. We count ordered tuples (lift, cycle, ..., cycle, colouring) so that the lift is properly coloured by the colouring and contains cycles of the specified length. We claim the contribution from cases where the cycles intersect turn out to be negligible, adapting an argument of \cite{KPW10}. Suppose that the cycles form a subgraph $H$ with $\nu$ vertices and $\mu$ edges. If the cycles are disjoint, then $\nu = \mu$. If they overlap, then as the minimum degree in $H$ is at least two and some vertex has degree at least three, we have $\nu < \mu$. We then follow the same argument as in the proof above with the following changes. When choosing vertices for the cycle, we have $(1+o(1))n^\nu$ choices. In \eqref{eq:critical_term_1}, the coefficient is $(\frac{n}{k})^{\nu} n^{-\nu}$. Then in \eqref{eq:critical_term_2}, the product ranges from 1 to $\mu$. Thus unlike in the proposition, where by \eqref{eq:simplified_joint_moment} all of the $n$'s have cancelled, we get a $\Theta(n^{\nu-\mu})$ term. In the disjoint case the $n$'s do cancel, but as there are finitely many isomorphism types of $H$, the contribution from terms with overlapping cycles are on the order $\frac{1}{n}$ times the rest. Finally, the disjoint terms decompose into a product of factors corresponding to the individual cycles, giving the desired result.

\section{The case where $n$ is not divisible by $k$} \label{sec:extension}

Thus far all arguments have assumed that $n$ is divisible by $k$ so that strongly equitable colourings exist. We now consider the case where $n = qk + r$ for some integer $r \in [1,k-1]$. We start by expanding the definition of equitable colourings.

Let $n = qk+r$ for $r \in [0,k-1]$. A $k$-colouring of an $n$-lift $L$ of a graph $G$ is strongly equitable if for every vertex $v \in G$, the fiber $f^{-1}(v)$ contains $q+1$ vertices of colour $1,\ldots, r$ and $q$ vertices of colour $r+1, \ldots, k$. Note that our previous definition of strongly equitable is exactly the case when $r = 0$.

Recall from Section~\ref{ssec:outline} that if $Y$ counts the number of strongly equitable $k$-colourings of a random lift $L$ of a $d$-regular graph $G$ then
\[ \ex Y^2 =\sum_{\mbf A,\mbf B} p(n, \mbf A, \mbf B) \exp\big( n f(\mbf A, \mbf B)\big) \]
where $p$ is some function bounded by a polynomial in $n$ and
\[
f(\mbf A, \mbf B) = -\sum_{v\in V}\sum_{i,j}a_{v,i,j}\log a_{v,i,j}
+\sum_{e\in E}\sum_{i,j,i',j'} b_{e,i,j,i',j'}\log\left( \frac{a_{v,i,j}a_{v',i',j'}}{b_{e,i,j,i',j'}}\right) 
\]

The next proposition describes the effect of adding a small constant $r$ vertices to each fiber.

\begin{proposition}\label{prop:second_moment_non_divisible}
Suppose $n = qk+r$ for some integer $r \in [1,k-1]$. Let $Y_n$ be the number of strongly equitable $k$-colourings of a random $n$-lift of a given $d$-regular graph $G$. Let $n' = n-r$ and let $Y_{n'}$ count the number of strongly equitable $k$-colourings of a random $n'$-lift of $G$. Then
\[ \ex Y_n^2 \sim \left(k^{|V|}\left(\frac{k-1}{k}\right)^{|E|}\right)^{2r} \ex Y_{n'}^2. \]
\end{proposition}

\begin{proof}
First assume $r > 1$. Recall the definition of $\mathbb{Y}_n(\gamma)$ given in Proposition~\ref{prop:central_sum}. Given $\mbf A, \mbf B \in \mathbb{Y}_n(1)$, we define:
\[ \mbf A^\ast = \left( \frac{1}{n'} (na_{v,i,j} - \chi_{i,j}) \right)_{v \in V, i,j \in [k]} \and \mbf B^\ast = \left( \frac{1}{n'} (nb_{e,i,j,i',j'} - \chi_{i,j,i',j'}) \right)_{e\in E, (i,j,i',j') \in K} \]
where
\[ \chi_{i,j} = \begin{cases} 1 & i=j \le r \\ 0 & \text{else} \end{cases} \]
and 
\[ \chi_{i,j,i',j'} = \begin{cases} 1 & i,j,i',j' \le r, i' = i+1 (\text{mod }r), j' = j+1 (\text{mod }r) \\ 0 & \text{else} \end{cases}. \]
Then $\mbf A^\ast, \mbf B^\ast$ define the overlap matrices of a pair strongly equitable $k$-colourings of $n'$-lifts of $G$. (Essentially, we remove one vertex of each of the first $r$ colours from each fiber from each colouring and removes edges of the appropriate colour pairs to ensure the $\mbf B^\ast$'s still satisfy \eqref{eq:abcons2}. This is possible because $\mbf A, \mbf B \in \mathbb{Y}_n(1)$ and therefore $\mbf A^\ast$ and $\mbf B^\ast$ remain non-negative.)

Recall from \eqref{eq:ABsplit} that the expected number of lifts respecting $\mbf A, \mbf B$ is
\[ \frac{1}{n!^{|E|}} \prod_{v\in V}\binom{n}{n A_v}
\prod_{e\in E} \frac{(n A_v)!(nA_{v'})!}{(nB_e)!}. \]
Call this quantity $\ell(\mbf A, \mbf B)$. Then as $\mbf A, \mbf B \in \mathbb{Y}_n(1)$,
\begin{align*}
\frac{\ell(\mbf A, \mbf B)}{\ell(\mbf A^\ast, \mbf B^\ast)} &= \left(\frac{n^r}{n_{(r)}}\right)^{|E|-|V|} \prod_{v \in V}\prod_{i \in [r]} \frac{1}{a_{v,i,i}} \prod_{e \in E}\prod_{i \in [r]} \frac{a_{v,i,i}a_{v',i,i}}{b_{e,i,i,i+1,i+1}}\\
	&<\left(\frac{n^r}{n_{(r)}}\right)^{|E|-|V|} \left(\frac{1}{\frac{1}{k^2}-\frac{\log n}{\sqrt{n}}}\right)^{r|V|} \left(\frac{(\frac{1}{k^2}+\frac{\log n}{\sqrt{n}})^2}{\frac{1}{k^2(k-1)^2}-\frac{\log n}{\sqrt{n}}}\right)^{r|E|}\\
	&\to \left(k^{|V|}\left(\frac{k-1}{k}\right)^{|E|}\right)^{2r}
\end{align*}
By Proposition~\ref{prop:central_sum}, noting $p(n, \mbf A, \mbf B)$ is bounded by a polynomial, we have
\begin{align*}
\ex Y_n^2 &\sim \sum_{\mbf A, \mbf B \in \mathbb{Y}_n(1)}p(n, \mbf A, \mbf B) \exp\big( n f(\mbf A, \mbf B)\big)\\
	&\sim \sum_{\{ \mbf A^\ast, \mbf B^\ast\}} \left(k^{|V|}\left(\frac{k-1}{k}\right)^{|E|}\right)^{2r} p(n', \mbf A^\ast, \mbf B^\ast) \exp\big( n' f(\mbf A^\ast, \mbf B^\ast)\big)
\end{align*}

Now we claim that there is $\gamma > 0$ such that $\mathbb{Y}_{n'}(\gamma) \subset \{\mbf
A^\ast, \mbf B^\ast\}$ and thus by Proposition~\ref{prop:central_sum}
\begin{align*}
\left(k^{|V|}\left(\frac{k-1}{k}\right)^{|E|}\right)^{2r} \ex Y_{n'}^2 &\sim
\sum_{\mbf A, \mbf B \in \mathbb{Y}_{n'}(\gamma)}
\left(k^{|V|}\left(\frac{k-1}{k}\right)^{|E|}\right)^{2r} p(n',\mbf A, \mbf
B)\exp(n' f(\mbf A, \mbf B))\\
	&\le \sum_{\{ \mbf A^\ast, \mbf B^\ast\}}
\left(k^{|V|}\left(\frac{k-1}{k}\right)^{|E|}\right)^{2r} p(n', \mbf A^\ast,
\mbf B^\ast) \exp\big( n' f(\mbf A^\ast, \mbf B^\ast)\big)\\
	&\le \sum_{\mbf A, \mbf B}
\left(k^{|V|}\left(\frac{k-1}{k}\right)^{|E|}\right)^{2r} p(n',\mbf A, \mbf
B)\exp(n' f(\mbf A, \mbf B))\\
	&= \left(k^{|V|}\left(\frac{k-1}{k}\right)^{|E|}\right)^{2r}\ex Y_{n'}^2
\end{align*}
and therefore
\[ \sum_{\{ \mbf A^\ast, \mbf B^\ast\}}
\left(k^{|V|}\left(\frac{k-1}{k}\right)^{|E|}\right)^{2r} p(n', \mbf A^\ast,
\mbf B^\ast) \exp\big( n' f(\mbf A^\ast, \mbf B^\ast)\big) \sim
\left(k^{|V|}\left(\frac{k-1}{k}\right)^{|E|}\right)^{2r} \ex Y_{n'}^2. \]

We will show that for all $\mbf C, \mbf D \in \mathbb{Y}_{n'}(\gamma)$, there is $\mbf A, \mbf B \in \mathbb{Y}_{n'}(1)$ such that
$\mbf A^\ast(\mbf A) = \mbf C$ and $\mbf B^\ast(\mbf B) = \mbf D$. Pick $c_{v,i,j}$ such that
\begin{equation}\label{eq:surj1}
\frac{1}{k^2}-\gamma\frac{\log n'}{\sqrt{n'}} \le c_{v,i,j} \le \frac{1}{k^2} +
\gamma\frac{\log n'}{\sqrt{n'}}.
\end{equation}
Then if $\mbf A^\ast(\mbf A) = \mbf C$ we have
\begin{equation} \label{eq:surj2}
\frac{1}{n}(n'a_{v,i,j}+\chi_{i,j}) = c_{v,i,j}
\end{equation}
for some
\begin{equation} \label{eq:surj3}
\frac{1}{k^2}-\frac{\log n}{\sqrt{n}} \le a_{v,i,j} \le \frac{1}{k^2} +
\frac{\log n}{\sqrt{n}}.
\end{equation}
Solving \eqref{eq:surj2} and plugging it into \eqref{eq:surj3} shows that we
require
\[ \frac{1}{k^2}-\frac{\log n}{\sqrt{n}} \le \frac{1}{n'}(nc_{v,i,j}+\chi_{i,j})
\le \frac{1}{k^2} + \frac{\log n}{\sqrt{n}}. \]
Combining with \eqref{eq:surj1} we see that we must choose a $\gamma > 0$ such that
\[ \frac{1}{k^2}-\frac{\log n}{\sqrt{n}} \le
\frac{n'}{n}\left(\frac{1}{k^2}-\gamma\frac{\log n'}{\sqrt{n'}}\right) \]
and
\[ \frac{1}{n}\left(n'\left(\frac{1}{k^2}+\gamma\frac{\log
n'}{\sqrt{n'}}\right)+1\right) \le \frac{1}{k^2} + \frac{\log n}{\sqrt{n}}. \]
Using loose bounds and the fact that $n \ge k \ge 3$, one can show that $\gamma =
0.4$ suffices. One can then repeat the process starting with $d_{e,i,j,i',j'}$ to find that $\gamma = 0.4$ still suffices.

Now if $r = 1$, we repeat the above process except we set $n' = n-2$ and continue as though $r=2$. Then $n' = n-2 = qk+1-2= (q-1)k+(k-1)$ we see
\[ \ex Y_n^2 \sim \left(k^{|V|}\left(\frac{k-1}{k}\right)^{|E|}\right)^{4} \ex Y_{n'}^2 \sim \left(k^{|V|}\left(\frac{k-1}{k}\right)^{|E|}\right)^{2(k+1)} \ex Y_{n-(k+1)}^2. \]
Then as $k\mid n-(k+1)$ and $k \mid n-1$,
Proposition~\ref{prop:expectation_y_squared} gives
\[ \ex Y_{n-(k+1)}^2 = \left(k^{|V|}\left(\frac{k-1}{k}\right)^{|E|}\right)^{-2k}\ex Y_{n-1}^2 \]
and so
\[ \ex Y_n^2 \sim \left(k^{|V|}\left(\frac{k-1}{k}\right)^{|E|}\right)^{2} \ex Y_{n-1}^2. \]
as required.
\end{proof}

Using the same process of removing vertices and edges of specified colour, one can prove analogues of Proposition~\ref{prop:second_moment_non_divisible} showing that
\[ \ex Y_n = \left(k^{|V|}\left(\frac{k-1}{k}\right)^{|E|}\right)^{r} \ex Y_{n'} \and \frac{\ex (Y_nZ_j)}{\ex Y_n} \sim \lambda_j(1+\delta_j).\]
The calculations are similar to and easier than those in the proof of Proposition~\ref{prop:second_moment_non_divisible}, and we omit them. Then the same proof of Theorem~\ref{thm:upper_bound} presented in Section~\ref{sec:main} holds for $n$ not divisible by $k$.

\bibliographystyle{abbrv}
\nocite{*}
\bibliography{bib}
\end{document}